\RequirePackage[l2tabu, orthodox]{nag}

\documentclass[11pt]{amsart}
\usepackage{fullpage,url,amssymb,enumerate}
\usepackage[all]{xy} 
\usepackage{comment}
\usepackage{graphicx}
\usepackage{tikz-cd}
\usepackage{nicematrix}


\usepackage{youngtab, ytableau}

\usepackage{mathdots}

\usepackage[utf8]{inputenc}
\usepackage[utf8]{inputenc}
\usepackage[english]{babel}
\usepackage[all]{xy}
\usepackage[margin=0.85in]{geometry}

\usepackage[OT2,T1]{fontenc}

\usepackage{amsthm}

\usepackage{color}

\def\bC{\mathbb{C}}
\def\cC{\mathcal{C}}
\def\bD{\mathbb{D}}

\def\cG{\mathcal{G}}

\def\cL{\mathcal{L}}
\def\oL{\overline{L}}
\def\cM{\mathcal{M}}

\def\cO{\mathcal{O}}
\def\bP{\mathbb{P}}
\def\cP{\mathcal{P}}
\def\bQ{\mathbb{Q}}

\def\cT{\mathcal{T}}
\def\cU{\mathcal{U}}

\def\cY{\mathcal{Y}}
\def\bZ{\mathbb{Z}}

\def\barM{\overline{\cM}}

\def\wt{\widetilde}

\DeclareMathOperator{\codim}{codim}

\DeclareMathOperator{\coker}{coker}
\DeclareMathOperator{\Coll}{Coll}

\DeclareMathOperator{\Fl}{Fl}

\DeclareMathOperator{\Gr}{Gr}

\DeclareMathOperator{\Hom}{Hom}

\DeclareMathOperator{\Id}{Id}
\DeclareMathOperator{\im}{im}

\DeclareMathOperator{\Pic}{Pic}
\DeclareMathOperator{\pr}{pr}
\DeclareMathOperator{\pt}{pt}

\DeclareMathOperator{\rank}{rank}

\DeclareMathOperator{\spine}{sp}

\newtheorem{thm}{Theorem}[section]
\newtheorem{lem}[thm]{Lemma}
\newtheorem{cor}[thm]{Corollary}
\newtheorem{prop}[thm]{Proposition}

\newtheorem{ex}[thm]{Example}
\newtheorem{defn}[thm]{Definition}

\newcommand{\SSYT}{{\mathsf{SSYT}}}
\newcommand{\Inc}{{\mathsf{Inc}}}

\newcommand{\T}{{\mathsf{T}}}
\newcommand{\Tev}{{\mathsf{Tev}}}
\newcommand{\vTev}{{\mathsf{vTev}}}

\newcommand{\vir}{{\mathsf{vir}}}

\newcommand{\vn}{{\overrightarrow{n}}}

\makeatletter
\g@addto@macro\bfseries{\boldmath} 
\makeatother

\usepackage[
	backref=false,
	pdfauthor={}, 
	pdftitle={tevelev_pn},
]{hyperref}

\begin{document}

\title[Degenerations of complete collineations and $\Tev^{\bP^r}$]{Degenerations of complete collineations and geometric Tevelev degrees of $\bP^r$}

\author{Carl Lian}

\address{Tufts University, Department of Mathematics, 177 College Ave
\hfill \newline\texttt{}
\indent Medford, MA 02155} \email{{\tt Carl.Lian@tufts.edu}}

\begin{abstract}
We consider the problem of enumerating maps $f$ of degree $d$ from a fixed general curve $C$ of genus $g$ to $\bP^r$ satisfying incidence conditions of the form $f(p_i)\in X_i$, where $p_i\in C$ are general points and $X_i\subset\bP^r$ are general linear spaces. We give a complete answer in the case where the $X_i$ are points, where the counts are known as the ``Tevelev degrees'' of $\bP^r$. These were previously known only when $r=1$, when $d$ is large compared to $r,g$, or virtually in Gromov-Witten theory. We also give a complete answer in the case $r=2$ with arbitrary incidence conditions. Our main approach studies the behavior of complete collineations under various degenerations. 
\end{abstract}

\maketitle


\section{Introduction}\label{intro}

\subsection{New results}\label{sec:new_results}

In this paper, we prove the following.

\begin{thm}\label{thm:tevPr}
Let $(C,p_1,\ldots,p_n)\in\cM_{g,n}$ be a general curve. Let $x_1,\ldots,x_n\in\bP^r$ be general points. Suppose that $n\ge r+1$ and that $d\ge0$ is an integer for which
\begin{equation}\label{virdim_equal_pr}
n=\frac{r+1}{r}\cdot d-g+1.
\end{equation}
Then, the number of maps $f:C\to\bP^r$ of degree $d$ for which $f(p_i)=x_i$ for all $i=1,\ldots,n$ is equal to
\begin{equation*}
\displaystyle\int_{\Gr(r+1,d+1)}\sigma^g_{1^r}\cdot \left(\sum_{\mu\subset(n-r-2)^{r}}\sigma_{\mu}\sigma_{\overline{\mu}}\right)_{\lambda_0\le n-r-1}.
\end{equation*}
\end{thm}

We use the standard notation for Schubert cycles on the Grassmannian $\Gr(r+1,d+1)$, see also \S\ref{sec:schubert}. The class $\sigma_{\mu}\in H^{*}(\Gr(r+1,d+1))$, where $\mu=(\mu_0,\ldots,\mu_r)$ is a partition with $d-r\ge \mu_0\ge\cdots\ge\mu_r\ge0$, is the Schubert cycle of $(r+1)$-planes whose intersection with the codimension $\mu_i+(r-i)$ component of a fixed complete flag in $\bC^{d+1}$ has dimension at least $i+1$. The notation $(n-r-2)^{r}$ denotes the partition $(n-r-2,\ldots,n-r-2,0)$, whose Young diagram is a rectangle, and $\mu\subset (n-r-2)^{r}$ denotes containment of Young diagrams, or equivalently that $\mu_i\le n-r-2$ for $i=0,1,\ldots,r-1$ and $\mu_{r}=0$. The partition $\overline{\mu}$ is the complement of $\mu$ inside $(n-r-2)^r$, so that $\overline{\mu}_i=(n-r-2)-\mu_{(r-1)-i}$ for $i=0,1,\ldots,r-1$ and $\overline{\mu}_r=0$. By convention, we take $\sigma_{\mu}=0$ if $\mu_0>d-r$, and similarly for $\overline{\mu}$. The subscript $\lambda_0\le n-r-1$ denotes that the class inside the parentheses is to be expanded in the Schubert basis of $H^{*}(\Gr(r+1,d+1))$, and then projected to the additive subgroup generated by Schubert cycles $\sigma_{\lambda}$ for which $\lambda_0\le n-r-1$. See Example \ref{formula_examples} for examples, and \S\ref{sec:tev_coeffs} for additional combinatorial properties of the formula. 

We emphasize that Theorem \ref{thm:tevPr} gives an actual count of curves, rather than a virtual intersection number. These counts are the \emph{geometric Tevelev degrees} $\Tev^{\bP^r}_{g,n,d}$ of $\bP^r$, and complete a picture studied by many authors \cite{castelnuovo,bdw,st,cps,bp,fl,cl} dating back to the 19th century and passing through the development of Gromov-Witten theory, see \S\ref{sec:tev_history}. Projective spaces are the first family of examples in which the geometric Tevelev degrees are fully understood.

We also consider the natural generalization in which the conditions $f(p_i)=x_i$ are replaced by conditions $f(p_i)\in X_i$, where the $X_i\subset\bP^r$ are general linear spaces of any dimension. Our most explicit result in this direction is the following.

\begin{thm}\label{thm:countP2}
Let $(C,p_1,\ldots,p_n)\in\cM_{g,n}$ be a general curve. Let $x_1,\ldots,x_{n_0}\in\bP^2$ be general points and let $X_{n_0+1},\ldots,X_{n}\subset\bP^2$ be general lines. Suppose that $d\ge0$ is an integer for which
\begin{equation}\label{virdim_equal_p2}
n+n_0=3d-2g+2.
\end{equation}
For any partition $\lambda$, let $\SSYT_3(\lambda)$ denote the set of semi-standard Young tableaux (SSYTs) of shape $\lambda$ and with entries in the set $\{1,2,3\}$. Then, the number of non-degenerate maps $f:C\to\bP^2$ of degree $d$ for which $f(p_i)=x_i$ for all $i=1,\ldots,n_0$ and $f(p_i)\in X_i$ for all $i=n_0+1,\ldots,n$ is equal to
\begin{equation*}
\displaystyle\int_{\Gr(3,d+1)}\sigma^g_{1^2}\cdot \left(\sum_{|\lambda|=n+n_0-8}\Gamma^\lambda_{2,\vn,d}\cdot\sigma_\lambda\right),
\end{equation*}
where the coefficient $\Gamma^\lambda_{2,\vn,d}\in\bZ_{\ge0}$ is equal to the cardinality of the following subset of $\SSYT_3(\lambda)$:
\begin{itemize}
\item the whole set $\SSYT_3(\lambda)$, if $\lambda_0\le n-5$,
\item the subset of SSYTs in which the entry $2$ appears at most $n_0-3$ times, if $\lambda_0=n-4$,
\item the subset of SSYTs containing neither a $(1,2)$-strip of length $n_0-3$, nor a $(2,3)$-strip of length $n-3$, if $\lambda_0=n-3$, and
\item $\emptyset$, if $\lambda_0>n-3$.
\end{itemize}
\end{thm}

See \S\ref{yt_sec} for notation and relevant definitions for Young tableaux. We will see that, in higher dimension, one should expect similar formulas involving coefficients $\Gamma^\lambda_{r,\vn,d}$ which are equal to the cardinality of a combinatorially meaningful subset of $\SSYT_{r+1}(\lambda)$.

In the rest of the introduction, we recall the history of the problem and give an overview of our new techniques.

\subsection{Tevelev degrees}

Tevelev degrees count the number of curves of fixed complex structure in an ambient space $X$ passing through the maximal number of points. More precisely, let $X$ be a smooth, projective, irreducible variety, and let $\beta\in H_2(X,\bZ)$ be an effective curve class. Consider the forgetful map
\begin{equation*}
\tau:\cM_{g,n}(X,\beta)\to \cM_{g,n}\times X^n
\end{equation*}
and suppose that
\begin{equation}\label{virdim_equal}
\int_\beta c_1(T_X)=\dim(X)(n+g-1),
\end{equation}
that is, the expected relative dimension of $\tau$ is zero. Suppose further that all dominating components of $\cM_{g,n}(X,\beta)$ are generically smooth of the expected dimension. Then, the \textbf{geometric Tevelev degree} $\Tev^X_{g,n,\beta}$ of $X$ is by definition equal to the degree of $\tau$. We will soon specialize to the case $X=\bP^r$, where the generic smoothness condition is satisfied whenever $n\ge r+1$, see \S\ref{sec:tev_history}. We refer to \cite{l_hyp,cl2} for recent partial computations of geometric Tevelev degrees of other targets.

Buch-Pandharipande \cite{bp} study systematically a parallel set of counts in Gromov-Witten theory. Consider now the forgetful map
\begin{equation*}
\overline{\tau}:\barM_{g,n}(X,\beta)\to \barM_{g,n}\times X^n
\end{equation*}
and assume \eqref{virdim_equal}. Then, the \textbf{virtual Tevelev degree} $\vTev^X_{g,n,\beta}$ of $X$ is defined to be the unique rational number for which
\begin{equation*}
\overline{\tau}_{*}[\barM_{g,n}(X,\beta)]^{\vir}=\vTev^X_{g,n,\beta}\cdot[\barM_{g,n}\times X^n].
\end{equation*}
For virtual degrees, no transversality hypothesis is required. The virtual and geometric Tevelev degrees often agree, but do not always; see \cite{lp,bllrst} for detailed investigations.

The term ``Tevelev degree'' was introduced in 2021 by Cela-Pandharipande-Schmitt \cite{cps}, after the formula (using our notation) $\Tev^{\bP^1}_{g,g+3,(g+1)[\bP^1]}=2^g$ appeared in the work of Tevelev \cite[Theorem 6.2]{tev}. On the other hand, much further-reaching calculations, now understood to be on the virtual side, had already appeared in the preceding decades, as we review in \S\ref{sec:tev_history}. Nevertheless, we will follow the most recent literature in using the name ``geometric Tevelev degree'' for the counts we consider, to highlight our interest in the geometric invariants that see only maps of degree $d$ out of smooth curves. 

\subsection{Tevelev degrees of $\bP^r$}\label{sec:tev_history}

In this section, we review the known results on the Tevelev degrees of $\bP^r$. We write $\beta=d$ for the homology class equal to $d$ times a line. Then, the condition \eqref{virdim_equal} becomes \eqref{virdim_equal_pr}, which we reproduce below:
\begin{equation}\tag{\ref{virdim_equal_pr}}
n=\frac{r+1}{r}\cdot d-g+1.
\end{equation}

This condition may be understood concretely as follows. There is a $g$-dimensional variety of line bundles $\cL$ of degree $d$ on $C$. For a fixed $\cL$, if $d$ is large, then the space of maps $f:C\to\bP^r$ with $f^{*}\cO_{\bP^r}(1)\cong\cL$ has dimension
\begin{equation*}
(r+1)\cdot h^0(C,\cL)-1=(r+1)(d-g+1)-1.
\end{equation*}
The $n$ incidence conditions $f(p_i)=x_i$ each impose $r$ conditions, so we expect finitely many $f$ as in Theorem \ref{thm:tevPr} if
\begin{equation*}
rn=g+(r+1)(d-g+1)-1=(r+1)(d+1)-1-rg,
\end{equation*}
which rearranges to \eqref{virdim_equal_pr}. The Brill-Noether theorem implies that, at least when $f$ is assumed non-degenerate (which is immediate when $n\ge r+1$, see below), the global parameter count remains correct for all $d$, so we have a finite answer to the \emph{geometric} enumerative problem under the assumption \eqref{virdim_equal_pr}.

The virtual degrees $\vTev^{\bP^r}_{g,n,d}$ are understood, following from a straightforward computation in the quantum cohomology ring of $\bP^r$, see \cite[(3)]{bp}.
\begin{thm}\label{thm:vTev}
Assume \eqref{virdim_equal_pr}. Then 
\begin{equation*}
\vTev^{\bP^r}_{g,n,d}=(r+1)^g.
\end{equation*}
\end{thm}
This formula was first obtained by Bertram-Daskalopoulos-Wentworth \cite{bdw} before the systematic development of the theory of virtual fundamental classes. In fact, the more general problem of counting curves on Grassmannians satisfying Schubert incidence conditions was considered by Bertram in \cite{bertram}, and it was proven by Siebert-Tian \cite{st} in large degree and by Marian-Oprea \cite{mo} virtually in all degrees that these much more general counts are determined by the Vafa-Intriligator formula. In particular, the virtual number of maps $f:C\to\bP^r$ with respect to arbitrary incidence conditions $f(p_i)\in X_i\subset\bP^r$ is also equal to $(r+1)^g$. Further virtual calculations can be carried out on moduli spaces of stable quotients \cite{mop}, which determine the Gromov-Witten theory of Grassmannians, and quasimaps \cite{ckm}. 

We focus in this paper on geometric curve counts in $\bP^r$, which are much more subtle, and are in general not determined in an apparent way by any of the aforementioned virtual theories. If $n\ge r+1$, then a map $f:C\to\bP^r$ out of a general curve with general point incidence conditions $f(p_i)=x_i$ is automatically non-degenerate, and the Brill-Noether-Petri theorem guarantees that the needed transversality hypothesis on
\begin{equation*}
\tau:\cM_{g,n}(\bP^r,d)\to \cM_{g,n}\times (\bP^r)^n
\end{equation*}
is satisfied. When instead $n<r+1$, it is often the case that there are more degenerate maps $f$ than expected\footnote{For example, when $r=2$, $n=2$, and $g=\frac{3}{2}d-1\ge 2$, there are infinitely many $f$ of degree $d$ from $C$ to the line between the two points $x_1,x_2\in\bP^2$, satisfying $f(p_i)=x_i$ for $i=1,2$.}. When $n=r+1$, it is always the case that $\Tev^{\bP^r}_{g,n,d}=0$, because in this case the Brill-Noether number
\begin{equation*}
\rho(g,r,d)=g-(r+1)(g+r-d)=-r
\end{equation*}
is negative, so the Brill-Noether theorem implies that no non-degenerate maps $f:C\to\bP^r$ of degree $d$ exist. The formula of Theorem \ref{thm:tevPr} also gives zero in this case. We therefore assume throughout the discussion of geometric Tevelev degrees of $\bP^r$ that $n\ge r+2$.

The case in which $d=r+\frac{rg}{r+1}$ and $n=r+2$ are as small as possible is classical. Then, the data of a map $f:C\to\bP^r$ with $f(p_i)=x_i$ for $i=1,2,\ldots,r+2$ is equivalent to the data of a linear series of minimal degree on $C$. The celebrated 19th century result of Castelnuovo \cite{castelnuovo} gives the number of such as
\begin{equation}\label{castelnuovo_formula}
\Tev^{\bP^r}_{g,n,d}=\int_{\Gr(r+1,d+1)}\sigma^g_{1^r}=g!\cdot\frac{1!\cdot 2!\cdot \cdots \cdot r!}{s!\cdot (s+1)!\cdot \cdots\cdot (s+r)!},
\end{equation} 
where $s=\frac{d}{r}-1=\frac{g}{r+1}$. In particular, Castelnuovo's count is quite far from the virtual (Gromov-Witten) count.\footnote{It follows from Corollary \ref{Tev_bound} that Castelnuovo's number is always less than or equal to virtual count, and strictly smaller except in genus 0. When $r=1$, Castelnuovo's count of maps $f:C\to\bP^1$ out of a curve of even genus $g=2h$ of the minimal degree $d=h+1$ is equal to the Catalan number $\frac{1}{h+1}\binom{2h}{h}$, whereas the virtual count is $2^g=4^h$.}

At the other extreme, where $d,n$ are sufficiently large, the geometric counts match the virtual ones, as had been essentially understood in the early work \cite{bdw,st}. 
\begin{thm}\cite[Theorem 1.1, Theorem 1.2]{fl}\label{geomtev_larged}
Assume \eqref{virdim_equal_pr} and that $d\ge rg+r$ (equivalently, $n\ge d+2$). Then 
\begin{align*}
\Tev^{\bP^r}_{g,n,d}&=(r+1)^g\\
&=\int_{\Gr(r+1,d+1)}\sigma^g_{1^r}\cdot\sum_{a_0+\cdots+a_r=r(n-r-2)}\sigma_{a_0}\cdots\sigma_{a_r}
\end{align*}
\end{thm}
The equality of the two formulas is non-trivial. They are obtained by two independent computations in \cite{fl}, but a direct combinatorial proof of their equality via the RSK algorithm was given in \cite{grb}.

When $r=1$, the geometric degrees $\Tev^{\bP^1}_{g,n,d}$ have been computed in full.
\begin{thm}\cite{cps,fl,cl}\label{tev_p1}
Assume \eqref{virdim_equal_pr} for $r=1$. Then 
\begin{align*}
\Tev^{\bP^1}_{g,n,d}&=2^g-2\sum_{j=0}^{g-d-1}\binom{g}{j}+(g-d-1)\binom{g}{g-d}+(d-g-1)\binom{g}{g-d+1}\\
&=\int_{\Gr(2,d+1)}\sigma^g_{1}\cdot\sum_{a_0+a_1=n-3}\sigma_{a_0}\sigma_{a_1}
\end{align*}
\end{thm}
Binomial coefficients $\binom{g}{j}$ with $j<0$ are interpreted to vanish. In particular, when $d\ge g+1$, we obtain simply $\Tev^{\bP^1}_{g,d,n}=2^g$, agreeing with the first formula of Theorem \ref{geomtev_larged}. The second formula of Theorem \ref{tev_p1} shows that the Schubert calculus formula of Theorem \ref{geomtev_larged} holds when $r=1$ for \emph{all} $d$ (whereas the $(r+1)^g=2^g$ formula does not), and also agrees with Theorem \ref{thm:tevPr} in this case. However, when $r>1$, the Schubert calculus formula of Theorem \ref{geomtev_larged} is only valid for large $d$.

The geometric counts when $r>1$ and $r+\frac{rg}{g+1}<d<rg+r$, which one may view as interpolating between Castelnuovo's count \eqref{castelnuovo_formula} and the Gromov-Witten count $(r+1)^g$, have remained open. Theorem \ref{thm:tevPr} completes the bridge between these two enumerative questions.

\begin{ex}\label{formula_examples}
Take $r=2$. Note by \eqref{virdim_equal_pr} that $d$ must be even.

The unique instance of $\Tev^{\bP^2}_{g,n,d}$ with $g\le 2$ not covered by the large degree regime in Theorem \ref{geomtev_larged} is the case $(g,n,d)=(2,5,4)$, in which case one can compute from Theorem \ref{thm:tevPr} that
\begin{equation*}
\Tev^{\bP^2}_{2,5,4}=\int_{\Gr(3,5)}\sigma_{11}^2\left(\sigma_{11}\sigma_\emptyset+\sigma_1^2+\sigma_\emptyset\sigma_{11}\right)=4.
\end{equation*}
Because $n>r+2=4$, this case is also not covered by Castelnuovo's formula.

As a check, we have 
\begin{equation*}
\vTev^{\bP^2}_{2,5,4}-\Tev^{\bP^2}_{2,5,4}=(2+1)^2-4=5.
\end{equation*}
This corresponds to the existence of exactly $n=5$ stable maps in the boundary of $\barM_{2,5}(\bP^2,4)$ over a general point of $\barM_{2,5}\times(\bP^2)^5$. These are given by a constant  map $C\to\bP^2$ with image $x_i$ (for a choice of $i=1,\ldots,5$), attached to 4 rational tails mapping isomorphically to the lines between $x_i$ and $x_j$, for $j\neq i$. More generally, we have $\Tev^{\bP^r}_{g,n,d}=(r+1)^g-n$ whenever $n=d+1$, or equivalently $d=rg$, the smallest possible value of $d$ outside the range covered by Theorem \ref{geomtev_larged}.

When $g=3$, the smallest case $(n,d)=(4,4)$ is determined by the Castelnuovo count:
\begin{equation*}
\Tev^{\bP^2}_{3,4,4}=\int_{\Gr(3,5)}\sigma_{11}^3=3!\cdot\frac{1!\cdot 2!}{1!\cdot 2!\cdot 3!}=1.
\end{equation*}
Indeed, a non-degenerate map $f:C\to\bP^2$ of degree 4 must be a canonical embedding, if $C$ is general. In the next case $(n,d)=(7,6)$, we have again
\begin{equation*}
\Tev^{\bP^2}_{3,7,6}=(2+1)^2-7=20.
\end{equation*}
In the remaining cases, we have $d\ge 8$ and $n=\frac{3}{2}d-2$, in which case $\Tev^{\bP^2}_{3,n,d}=27$ by Theorem \ref{geomtev_larged}.

The first case not determined by the Castelnuovo count, and in which the boundary contributions to the Gromov-Witten count $\vTev^{\bP^2}_{g,n,d}$ have positive dimension and thus cannot be easily integrated, is $(g,n,d)=(4,6,6)$. Here, Theorem \ref{thm:tevPr} gives
\begin{align*}
\Tev^{\bP^2}_{4,6,6}&=\int_{\Gr(3,7)}\sigma_{11}^4\left(\sum_{\lambda\subset(22)}\sigma_\lambda\sigma_{\overline{\lambda}} \right)_{\lambda_0\le 3}\\
&=\int_{\Gr(3,7)}\sigma_{11}^4\left(2\sigma_{22}\sigma_{\emptyset}+2\sigma_{21}\sigma_{1}+\sigma_{11}^2+\sigma_2^2\right)_{\lambda_0\le 3}\\
&=\int_{\Gr(3,7)}\sigma_{11}^4\left(3\sigma_{31}+6\sigma_{22}+3\sigma_{211}\right)\\
&=3\cdot3+6\cdot2+3\cdot3\\
&=30.
\end{align*}
The remaining cases when $g=4$ are $\Tev^{\bP^2}_{4,9,8}=3^4-9=72$ and $\Tev^{\bP^2}_{4,n,d}=81$ whenever $d\ge10$ and $n=\frac{3}{2}d-3$.
\end{ex}

\subsection{Complete collineations}\label{sec:intro_coll}

We now discuss the new ingredients of this paper.

The geometric approaches of \cite{bdw,fl} to Tevelev degrees encounter the same difficulty: in intersection theory calculations, one may obtain contributions from ``maps with base-points,'' that is, $(r+1)$-tuples of sections $[f_0:\cdots:f_r]$, where $f_j\in H^0(C,\cL)$ are sections all vanishing at some (or all) of the $p_i$. Essentially the same issue arises in the Gromov-Witten setting, where one has virtual contributions from stable maps obtained by attaching rational tails at $p_i\in C$ whose images pass through $x_i\in\mathbb{P}^r$. This in particular explains the discrepancy between the virtual and geometric degrees.

In order to access the geometric Tevelev degrees $\Tev^{\bP^r}_{g,n,d}$ in all cases, one needs to avoid such contributions. It is essentially a consequence of the Brill-Noether theorem that these only arise when the $f_j$ are linearly dependent. Indeed, if the $f_j$ are linearly independent, then the requirement that \emph{all} of them vanish at a fixed point $p_i$ imposes $r+1$ conditions, more than the expected $r$ conditions for the constraint $f(p_i)=x_i$. We therefore pass to the moduli space of \emph{complete collineations}, which is obtained by blowing up the loci where the linear map $\mathbb{C}^{r+1}\to H^0(C,\cL)$ defined by the $f_j$ drops rank (see \S\ref{coll_sec} or the discussion below). This on the one hand isolates the desired contributions from ``honest'' maps $f:C\to\mathbb{P}^r$ of degree $d$, and on the other facilitates a limit linear series degeneration from genus $g$ to genus $0$. In fact, the method works for arbitrary linear incidence conditions $f(p_i)\in X_i$.

Let us now describe this degeneration. Consider a 1-parameter family with general fiber $C$ and whose special fiber is a nodal curve $C_0$ given by a $\bP^1$ attached to $g$ elliptic tails at general points. (We will see in \S\ref{sec:lls} that we additionally attach a rational tail containing the marked points, but we suppress this here.) Such degenerations, which are amenable to the theory of limit linear series \cite{eh}, are standard in Brill-Noether theory. Indeed, we consider the linear series on $C$ underlying $f$ and its corresponding limit linear series on the special fiber, with the additional data of a complete collineation $\mathbb{C}^{r+1}\to H^0(\bP^1,\cO(d))$ on the rational component of $C_0$. The limit object is essentially determined by this data on this rational component, where there is now a unique line bundle of degree $d$, but with a space of sections of dimension $g$ more than the expected dimension of $H^0(C,\cL)$. Returning to the dimension count at the beginning of \S 1.3, we find that the (expected) number of parameters has so far increased by $rg$.

On the other hand, a standard limit linear series computation shows that the complete collineation $\mathbb{C}^{r+1}\to H^0(\bP^1,\cO(d))$ must in addition satisfy a codimension $r$ Schubert condition at each of the $g$ attachment points of the elliptic tails, due to constraints on the vanishing of limit sections on the elliptic components. These Schubert conditions restore the expected number of moduli to zero. This degeneration therefore reduces the calculation to studying complete collineations $\mathbb{C}^{r+1}\to H^0(\bP^1,\cO(d))$ whose target is a \emph{fixed} vector space. The geometry of the curve can then be abstracted away.

Let $V,W$ be vector spaces of dimensions $r+1,d+1$, respectively, and let $\Coll(V,W)$ be the moduli space of complete collineations $\phi:V\to W$. The space $\Coll(V,W)$ is obtained by an iterated blowup $b:\Coll(V,W)\to\bP\Hom(V,W)$ of the space of linear maps $\phi_0:V\to W$ at the loci where $\phi_0$ drops rank, in increasing order of rank. The data of a complete collineation includes, in addition to the linear map $\phi_0$, a linear map $\phi_1:\ker(\phi_0)\to\coker(\phi_0)$, and more generally additional maps $\phi_{j+1}:\ker(\phi_j)\to\coker(\phi_j)$ whenever $\phi_j$ fails to have full rank. See \S\ref{coll_sec} for a detailed discussion. The space $\Coll(V,W)$ also admits a map $\pi:\Coll(V,W)\to \Gr(r+1,W)$, which generically remembers the image of $\phi_0$, see Definition \ref{def:image}.

Let $L\subset V$ be a vector subspace of any dimension and let $M\subset W$ be a hyperplane through the origin. We will define the \emph{incidence locus} $\Inc(L,M)\subset\Coll(V,W)$ by the proper transform under $b$ of the locus of maps for which $\phi_0(L)\subset M$, and define $\gamma_{\dim(L)}=[\Inc(L,M)]\in H^{2\dim(L)}(\Coll(V,W))$ to be the corresponding cycle class. These may be regarded as tautological classes in the cohomology of $\Coll(V,W)$ which may be of independent interest.

The relevance of these classes for us is the following. We may regard $L\subset V\cong\bC^{r+1}$ as corresponding to a subspace of $X_i\subset \bP^r$ of codimension $\dim(L)$, and $M\subset W$ as the hyperplane of sections of $H^{0}(C,\cL)$ vanishing at a fixed point $p_i$. Then, the locus $\Inc(L,M)$ corresponds to the condition $f(p_i)\in X_i$.

We prove:

\begin{thm}\label{thm:coll}
Let $(C,p_1,\ldots,p_n)\in\cM_{g,n}$ be a general curve. Let $X_1,\ldots,X_n\subset\bP^r$ be general linear spaces of dimensions $k_1,\ldots,k_n<r$, respectively. For $j=0,1,\ldots,r-1$, let $n_j$ be the number of $X_i$ of dimension $j$, and write $\vn=(n_0,\ldots,n_{r-1})$. 

Assume that
\begin{equation}\label{general_numerology}
(r+1)(d+1)-1-rg=\sum_{i=1}^{n}(r-\dim X_i)=\sum_{k=0}^{r-1}(r-k)n_k=:|\vn|.
\end{equation}
Then, the number $\T^{\bP^r}_{g,\vn,d}$ of non-degenerate maps $f:C\to\bP^r$ of degree $d$ with $f(p_i)\in X_i$ is equal to
\begin{equation*}
\int_{\Coll(V,W)}\pi^{*}(\sigma^g_{1^r})\cdot\prod_{j=1}^{r}\gamma_{j}^{n_{r-j}}.
\end{equation*}
\end{thm}

Recall here that the class $\gamma_j$ is the cycle class $[\Inc(L,M)]$, where $L\subset V$ has dimension $j$ and $M\subset W$ is a hyperplane. In this way, the product $\prod_{j=1}^{r}\gamma_{j}^{n_{r-j}}$ is the cycle class of an intersection 
\begin{equation*}
Y_{r,\vn,d}=\bigcap_{i=1}^{n}\Inc(L_i,M_i)\subset\Coll(V,W),
\end{equation*}
where the $L_i\subset V$ are general subspaces of dimension $\codim(X_i)$ and the $M_i\subset W$ are general hyperplanes. The classes $\sigma_{1^r}$ appearing in Theorem \ref{thm:coll} are the Schubert conditions that arise at each of the $g$ nodes of the rational component $\bP^1$ of our degenerate curve, as discussed above. Theorem \ref{thm:coll} should be viewed as a tranversality statement: the space of complete collineations resolves all excess intersections from (for example) the moduli space of stable maps, and is therefore the correct place to access the geometric fixed-domain curve counts for $\bP^r$.

By the projection formula, to compute $\T^{\bP^r}_{g,\vn,d}$ (and, by taking $n_0=n$, the geometric Tevelev degrees of $\bP^r$), it suffices to understand the classes
\begin{equation*}
\Gamma_{r,\vn,d}:=\pi_{*}([Y_{r,\vn,d}])
\end{equation*}
on $\Gr(r+1,W)=\Gr(r+1,d+1)$. These classes seem to be new and may be of independent interest. In fact, we will reduce further in \S\ref{reduction_sec} to the case $d=n-1$, showing that the classes $\Gamma_{r,\vn,n-1}$ determine all classes $\Gamma_{r,\vn,d}$.

\subsection{Degenerations of subspace arrangements}

Our main approach to studying the classes $\Gamma_{r,\vn,d}$ is by degeneration. First, consider the case in which the $L_i$ are hyperplanes, corresponding to the geometric Tevelev degrees; we abuse notation and write $n=\vn=(n,0,\ldots,0)$. When $d=n-1$, we show that $Z_{r,n,d}:=\pi(Y_{r,n,d})$ is equal to a generic torus orbit closure in the Grassmannian $\Gr(r+1,n)$. In this case, the class $\Gamma_{r,n,n-1}=[Z_{r,n,n-1}]$ has been computed by Klyachko \cite{k} and Berget-Fink \cite{bf}, see also Theorem \ref{orbit_formula}. Along with the reductions in \S\ref{reduction_sec}, this will complete the proof of Theorem \ref{thm:tevPr}. 

However, we outline an independent computation of the class $\Gamma_{r,n,n-1}$ that proceeds by gradually degenerating the $L_i$ to contain successively larger subspaces $\Lambda\subset V$, showing that the subscheme $Z_{r,n,n-1}$ degenerates to a union of Richardson varieties whose classes are transparent. In the language of maps $f:C\to\bP^r$, we degenerate the points $x_i$ to lie in successively smaller linear spaces and study the corresponding degenerations of $f$ (viewed as complete collineations). We have essentially carried out this degeneration in the more general setting of torus orbit closures on full flag varieties in \cite{l_flag}, but we revisit the method in order to apply it in the more general setting of $L_i$ of arbitrary dimension.

Specifically, we prove Theorem \ref{thm:countP2} by moving the points and lines $x_1,\ldots,x_{n_0},X_{n_0+1},\ldots,X_n$ into successively more special position. The method here is much more delicate. A key feature present in the orbit closure case but missing here is that $Z_{2,\vn,d}=\pi(Y_{2,\vn,d})$ is no longer toric, and its degenerations are no longer controlled by polyhedral subdivisions. In particular, we must prove by hand that no components appear in the limit of $Z_{2,\vn,d}$ other than those that we construct explicitly.

%
%
%

\subsection{Outline}

The paper is organized as follows. We discuss preliminaries in \S\ref{sec:prelim} and \S\ref{coll_sec}. In particular, we introduce and study the geometry of the loci $\Inc(L,M)\subset\Coll(V,W)$ in \S\ref{coll_sec}. Sections \S\ref{coll_to_count_sec}, \S\ref{asymptotic_sec}, and \S\ref{reduction_sec} establish generalities relating to the counts $\T^{\bP^r}_{g,\vn,d}$. In particular, we prove Theorem \ref{thm:coll} in \S\ref{coll_to_count_sec}. The coefficients $\Gamma^{\lambda}_{r,\vn,d}$ in the Schubert basis of the classes $\Gamma_{r,\vn,d}$ are compared to the numbers $|\SSYT_{r+1}(\lambda)|$ in \S\ref{asymptotic_sec}. We reduce the computation of $\Gamma_{r,\vn,d}$ to the case $d=n-1$ in \S\ref{reduction_sec}. Finally, our main calculations take place in \S\ref{tevelev_sec} and \S\ref{p2_sec}. In \S\ref{tevelev_sec}, we relate geometric Tevelev degrees of $\bP^r$ to torus orbit closures on $\Gr(r+1,W)$, proving Theorem \ref{thm:tevPr}. It is here where our main degeneration technique is introduced. Finally, we prove Theorem \ref{thm:countP2} in \S\ref{p2_sec}.

\subsection{Conventions}\label{conventions}
\begin{itemize}
\item We work exclusively over $\bC$.
\item If $S$ is a finite set, then we denote its cardinality by $|S|$.
\item If $V$ is a vector space, then $\bP(V)$ is the space of lines in $V$, and $\bP(V^{\vee})$ is the space of hyperplanes in $V$. Similarly, if $W$ is a vector space, then $\Gr(r+1,W)$ is the Grassmanian of $(r+1)$-dimensional subspaces of $W$.
\item Angle brackets $\langle-\rangle$ denote linear span in a vector space or projective space.
\item Let $Y\subset X$ be a pure-dimensional subscheme of a smooth, projective variety $X$. Then, the cycle class in $H^{2*}(X)$ associated to $Y$ is denoted $[Y]$.
\item Let $W$ be a vector space, let $U\subset W$ be a subspace, and let $g\in GL(W)$ be an automorphism. We say that $g$ \textit{stabilizes} $U$ if $g(U)=U$. We say that $g$ \textit{fixes} $U$ if $g$ stabilizes $U$ and in addition restricts to the identity map on $U$.
\end{itemize}

\subsection{Acknowledgments}
We thank Alessio Cela, Izzet Coskun, Gavril Farkas, Alex Fink, Maria Gillespie, Eric Larson, Alina Marian, Rahul Pandharipande, Andrew Reimer-Berg, Johannes Schmitt, and Hunter Spink for helpful discussions. We are grateful to the referee for a careful reading and numerous helpful suggestions. This project was completed with support from an NSF postdoctoral fellowship, grant DMS-2001976, the MATH+ incubator grant ``Tevelev degrees,'' and an AMS-Simons travel grant. 

\section{Preliminaries}\label{sec:prelim}

\subsection{Schubert calculus}\label{sec:schubert}

Let $W$ be a vector space of dimension $d+1$ and fix a complete flag $F$ of subspaces
\begin{equation*}
0=F_0\subset F_1\subset\cdots \subset F_{d+1}=W.
\end{equation*}
Let $\lambda=(\lambda_0,\ldots,\lambda_r)$ be a partition, where $d-r\ge \lambda_0\ge\cdots\ge \lambda_r\ge0$. 
Then, the Schubert cycle $\Sigma^F_{\lambda}\subset\Gr(r+1,W)=\Gr(r+1,d+1)$ is by definition the subvariety consisting of subspaces $V\subset W$ of dimension $r+1$ for which
\begin{equation*}
\dim(V\cap F_{d-r+i+1-\lambda_{i}})\ge i+1
\end{equation*}
for $i=0,1,\ldots,r$. The class of $\Sigma^F_{\lambda}$ in $H^{2|\lambda|}(\Gr(r+1,W))$, where $|\lambda|=\lambda_0+\cdots+\lambda_r$ denotes the size of the partition, is denoted $\sigma_\lambda$.

The top cohomology group $H^{2(r+1)(d-r)}(\Gr(r+1,W))$ is generated by the single class $\sigma_{(d-r)^{r+1}}$. The unique $\bQ$-linear map $H^{2(r+1)(d-r)}(\Gr(r+1,W))\to\bQ$ sending $\sigma_{(d-r)^{r+1}}$ to 1 is denoted by $\int_{\Gr(r+1,d+1)}$.

If $\zeta\in H^{*}(\Gr(r+1,W))$ is any class, then $(\zeta)_{\lambda_0\le m}$ is defined by expanding $\zeta$ in the basis of Schubert classes $\sigma_\lambda$ and projecting to the subgroup generated by classes with $\lambda_0\le m$. Similarly, one can replace $\lambda_0\le m$ with other inequalities on the parts of $\lambda$.

\subsection{Young tableaux}\label{yt_sec}
 
Let $\lambda=(\lambda_0,\ldots,\lambda_r)$ be a partition. We adopt the convention throughout that $\lambda_0\ge\cdots\ge\lambda_r$. We identify $\lambda$ with its Young diagram.

\begin{defn}
A \emph{strip} $S$ of $\lambda$ is a collection of $k$ boxes in the Young diagram of $\lambda$ with the property that:
\begin{itemize}
\item Exactly one box of $S$ lies in each of the first $k$ columns of $\lambda$, and 
\item Given any distinct boxes $b_1,b_2$ of $S$, if $b_1$ lies in a column to the left of $b_2$, then $b_1$ does not also lie in a row above $b_2$.
\end{itemize}
\end{defn}

Unlike in \cite{l_torus}, we do not require $k$ to be as large as possible with respect to an ambient rectangle; in particular, we allow $k<\lambda_0$. An example of a strip in the partition $\lambda=(12,9,4,2)$ is shaded below.
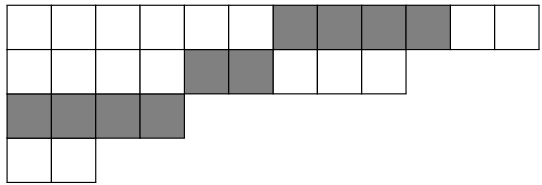
\begin{equation*}
\ydiagram[*(white) ]
{6,4}
*[*(gray)]{6+4,4+2,4}
*[*(white)]{10+2,6+3,4+0,2}
\end{equation*}

Recall that a \emph{semi-standard Young Tableau (SSYT)} of shape $\lambda$ is a filling of the boxes of $\lambda$ with entries in the set $\{1,2,\ldots,r+1\}$ so that entries increase weakly across rows and strictly down columns. The number of SSYTs with $r+1$ allowed entries of shape $\lambda$ is denoted $|\SSYT_{r+1}(\lambda)|$, and is given by a hook length formula, see \cite[Corollary 7.21.4]{EC2}. For $i=1,2,\ldots,r+1$ and a fixed SSYT, we denote the number of appearances of the entry $i$ by $c_i$.

\begin{defn}\label{(i,i+1)-strip}
For $i=1,2,\ldots,r$, an \emph{$(i,i+1)$-strip} of a SSYT is a strip, all of whose boxes are filled with the entry $i$ or $i+1$, and for which all instances of $i$ all appear to the left of all instances of $i+1$. A \emph{1-strip} is, by definition, an $(i,i+1)$-strip for some $i$. 
\end{defn}

Note that a strip filled entirely with the entry $i$ is both a $(i-1,i)$- and a $(i,i+1)$-strip. Below, the SSYT of shape $\lambda=(10,9,4,2)$ has a $(2,3)$-strip of length 10. The longest $(1,2)$-strip has length 6 and the longest $(3,4)$-strip has length 4.
\begin{equation*}
\begin{ytableau}
1 & 1 & 1 & 1 & 2 & 2 & 3 & 3 & 3 & 3 \\
2 & 2 & 2 & 2 & 3 & 4 & 4 & 4 & 4\\
3 & 3 & 3 & 4\\
4 & 4
\end{ytableau}
\end{equation*}

\section{Complete Collineations}\label{coll_sec}

\subsection{Basic notions}

The main reference is the work of Vainsencher \cite{vainsencher}; see also the work of Thaddeus \cite{thaddeus} for a more modern treatment.

\begin{defn}
Let $V,W$ be vector spaces of dimensions $r+1,d+1$, respectively. We assume throughout that $r\le d$. A \textbf{complete collineation} $\phi=\{\phi_i:V_j\to W_j\}_{j=0}^{\ell}$, often abusively denoted $\phi:V\to W$, is a collection of non-zero linear maps $\phi_j:V_j\to W_j$, each considered up to scaling, such that $V=V_0$, $W=W_0$, $V_j=\ker(\phi_{j-1})$ and $W_i=\coker(\phi_{j-1})$, and the last map $\phi_\ell:V_\ell\to W_\ell$ is injective.
\end{defn}

Due to the requirement that the $\phi_j$ be non-zero, the dimensions of the $V_j$ and $W_j$ are strictly decreasing, so any sequence of such maps must terminate in one of full rank. For sake of brevity, we often refer to complete collineations throughout this paper as simply ``collineations.'' 

\begin{defn}
We refer to the integer $\ell=\ell(\phi)$ as the \textbf{length} of $\phi$, and the $(\ell+1)$-tuple of positive integers $\overrightarrow{r}=(\rank(\phi_0),\ldots,\rank(\phi_\ell))=(r_0,\ldots,r_\ell)$, with $r_0+\cdots+r_\ell=r+1$, as the \textbf{type} of $\phi$. 

We say that $\phi$ is of \textbf{full rank} if it is of type $(r+1)$ (equivalently, of length 0), and is \textbf{totally degenerate} if it is of type $(1,1,\ldots,1)$ (equivalently, of length $r$).
\end{defn}

\begin{defn}
Denote by $\Coll(V,W)=\Coll(r+1,d+1)$ the moduli space of complete collineations from $V$ to $W$. Let $b:\Coll(V,W)\to\bP\Hom(V,W)$ be the canonical morphism remembering the map $\phi_0:V\to W$. 
\end{defn}

In fact, $b$ is an iterated blowup at smooth subvarieties: the locus of rank 1 maps $\phi_0:V\to W$ in $\bP\Hom(V,W)$ is blown up first, followed by the proper transform of the locus of maps of rank at most 2, and so on. In particular, $\Coll(V,W)$ is smooth, projective, and irreducible of dimension $(r+1)(d+1)-1$. There is a natural action of the group $GL(V)\times GL(W)$ on $\Coll(V,W)$, whose orbits are indexed by the possible types $\overrightarrow{r}$ of $\phi$.

\begin{defn}\label{def:image}
Let $\pi:\Coll(V,W)\to\Gr(r+1,W)$ be the map sending $\phi$ to $\pr_\ell^{-1}(\im(\phi_\ell)) \subset W$, where $\pr_j:W\to W_j$ denotes the canonical quotient map.
\end{defn}
The subspace $\pr_\ell^{-1}(\im(\phi_\ell)) \subset W$ may alternatively be regarded as the span of the images of \textit{all} of the $\phi_i$ upon pullback to $W$. We will refer to this subspace of $W$ as the \textbf{image} of $\phi$, denoted $\im(\phi)$. For a fixed subspace $W'\subset W$ of dimension $r+1$, the fiber of $\pi$ over $W'\in\Gr(r+1,W)$ is isomorphic to $\Coll(V,W')$. Note now that $V,W'$ have the same dimension; the space $\Coll(V,W')$ may be identified with the wonderful compactification of $PGL(r+1)$. We will not use this identification in what follows.

\begin{defn}
Fix a type $\overrightarrow{r}=(r_0,\ldots,r_\ell)$ with $r_0+\cdots+r_\ell=r+1$. Let  $\Coll^{\circ}_{\overrightarrow{r}}(V,W)$ be the locally closed subvariety of $\Coll(V,W)$ of collineations of type $\overrightarrow{r}$, and let $\Coll_{\overrightarrow{r}}(V,W)$ be its closure.
\end{defn}

As a scheme, $\Coll_{\overrightarrow{r}}(V,W)$ may be constructed inductively, as an open subset of a sequence of iterated projective and Grassmannian bundles. In particular, we find that $\Coll_{\overrightarrow{r}}(V,W)\subset \Coll(V,W)$ is smooth and irreducible of codimension $\ell$. The stratification of $\Coll(V,W)$ by the $\Coll_{\overrightarrow{r}}(V,W)$ is that induced by the normal crossings divisor given by the exceptional loci $\Coll_{(r_0,r_1)}(V,W)$. 

\subsection{Limits of full rank maps}\label{sec_limits}

Let $\phi^t:V\to W$ be a 1-parameter family of linear maps, which we assume are injective near, but not at, $t=0$. We regard
\begin{equation*}
\phi^t=\sum_{k\ge0} t^k\cdot \phi^k,
\end{equation*}
as an element of $\Hom(V,W)\otimes_{\bC}\bC[[t]]$, where $\phi^0=\phi_0$ and more generally
\begin{equation*}
\phi^k=\frac{1}{k!}\cdot\frac{d^k}{dt^k}\phi^t\Big|_{t=0}
\end{equation*}

We explain how to compute $\lim_{t\to 0}\phi^t$ in the space $\Coll(V,W)$. First, define $\phi_0=\lim_{t\to 0}\phi^t$ as a projectivized linear map, and define $V_1=\ker(\phi_0)$, $W_1=\coker(\phi_0)$. Write $r_0$ for the rank of $\phi_0$; we may assume $r_0\le r$, or else $\phi_0$ is also the limit of $\phi^t$ as a collineation. 

Next, let $k_1>0$ to the minimum value of $k$ for which $\phi^k$ restricts to a non-zero map $V_1\to W_1$, and define $\phi_1=\phi^{k_1}$. The integer $k_1$ measures the speed at which $\phi^t$ drops rank. Concretely, if we choose a decomposition $V=V_0=V_1\oplus V'_1$, then for any $v_1\in V_1$ with $\phi_1(v_1)\neq0$, we have, up to scaling,
\begin{equation*}
\phi_1(v_1)=\lim_{t\to 0}\phi^t(\langle V'_1,v_1\rangle)\subset W,
\end{equation*}
where we view the right hand side as a line in $W_1=W/\phi_0(V'_1)$. The right hand side defines a non-zero element of $W_1$ for any $v_1\in V_1$, but we have $\phi_1(v_1)=0$ (that is, $v_1\in V_2$) if the limit is only realized at order larger than $k_1$, which is to say that $\phi^{k_1}(v_1)\in \im(\phi_0)$.

The further maps $\phi_j:V_j\to W_j$ are determined similarly by iterating this procedure.

\subsection{The incidence loci}

Fix integers $d\ge r\ge 1$ and vector spaces $V,W$ of dimensions $r+1,d+1$, respectively. Let $L\subset V$ be a non-zero vector subspace of codimension $k+1\ge1$, and let $M\subset W$ be a vector subspace of codimension 1 (hyperplane).

\begin{defn}
Let $\Inc'(L,M)\subset \bP\Hom(V,W)$ be the locus of $\phi:V\to W$ for which $\phi(L)\subset M$. 

Let $\Inc(L,M)\subset\Coll(V,W)$ be the proper transform of $\Inc'(L,M)$ under the iterated blowup $b:\Coll(V,W)\to\bP\Hom(V,W)$.
\end{defn}

The locus $\Inc'(L,M)\subset \bP\Hom(V,W)$ is a linear subspace of codimension $r-k$, so $\Inc(L,M)\subset\Coll(V,W)$ also has codimension $r-k$. We denote by $\gamma_{r-k}\in H^{2(r-k)}(\Coll(L,M))$ the cycle class of $\Inc(L,M)$.

%

\begin{ex}\label{map_and_coll}
Take $V=\bC^{r+1}$ and $W=H^0(\bP^1,d)$. Let $L\subset V$ be a linear space of codimension $k+1$, let $p\in\bP^1$, and let $M\subset W$ be the hyperplane of sections vanishing at $p$. We interpret $\bP\Hom(V,W)$ as the locus of maps (possibly with base-points) $f:\bP^1\to\bP^r=\bP(V^{\vee})$ of degree $d$.

In this case, $\Inc(L,M)$ is (away from the locus where $\im(\phi)\subset M$) the locus of maps sending $p$ to the linear subspace $X_L\subset\bP^r$ of dimension $k$ corresponding to $L$.
\end{ex}

We now give a set-theoretic description of $\Inc(L,M)$. 

\begin{defn}
For any collineation $\phi$, for integers $j=0,1,\ldots,\ell(\phi)$, let $\pr_j:W\to W_j$ be the projection map. We denote by $(\dagger)_j$ the condition that
\begin{equation*}
\phi_j(L\cap V_j)\subset\pr_j(M).
\end{equation*}
\end{defn}

\begin{prop}\label{inc_set_theoretic}
We have $\phi\in\Inc(L,M)\subset\Coll(V,W)$ if and only if property $(\dagger)_j$ holds for all $j$ for which $L,V_j\subset V$ are transverse, that is, intersect in the expected codimension of $(k+1)+(r_0+\cdots+r_{j-1})$.
\end{prop}

The transversality requirement in the case $(k+1)+(r_0+\cdots+r_{j-1})\ge r+1$ is that the two subspaces intersect only at zero. In this case, the condition $(\dagger)_j$ is an empty one.

We first observe that, in order to verify that $\phi\in\Inc(L,M)$ in practice, one only has to check the property $(\dagger)_j$ for one value of $j$.

\begin{lem}\label{check_dagger_once}
Assume the statement of Proposition \ref{inc_set_theoretic}. Given $\phi,L,M$ as above, define:\
\begin{itemize}
\item $j_V$ to be be the largest value of $j$ with the property that $(k+1)+(r_0+\cdots+r_{j-1})\le r$ and $V_j$ is transverse to $L$, and
\item $j_W$ be the largest value of $j$ for which $j\le j_V$ and in addition $\pr_j(M)\neq W_j$.
\end{itemize}
Then, we have $\phi\in\Inc(L,M)$ if $(\dagger)_j$ holds for $j=j_W$.
\end{lem}

Note first that such integers $j_V,j_W$ exist, because $j=0$ satisfies both properties. 

\begin{proof}
If $j>j_V$, then either $L,V_j$ vail to be transverse, in which the condition $(\dagger)_j$ does not need to be checked, or  $(k+1)+(r_0+\cdots+r_{j-1})\ge r+1$ and $L\cap V_j=0$, in which case, as remarked above, $(\dagger)_j$ is automatic. If $j\in(j_W,j_V]$, then we have $\pr_j(M)=W_j$, so $(\dagger)_j$ is again automatic. Finally, suppose that $j\in [0,j_W)$. Then, we have that $\pr_{j}(M)$ is a hyperplane in $W_{j}$ which does not surject onto $W_{j+1}=W_{j}/\im(\phi_{j})$. This means that $\im(\phi_{j})\subset \pr_{j}(M)$, so $(\dagger)_j$ is satisfied. Therefore, if $(\dagger)_{j_W}$ also holds, then the criterion of Proposition \ref{inc_set_theoretic} is satisfied, and $\phi\in\Inc(L,M)\subset\Coll(V,W)$.
\end{proof}

\begin{proof}[Proof of Proposition \ref{inc_set_theoretic}]
We first show that the conditions $(\dagger)_j$ are necessary. Suppose that $\phi=\{\phi_j\}_{j=0}^{\ell}$ is the limit of a one-parameter family of full rank collineations $\phi^t$ in $\Inc'(L,M)\subset\bP\Hom(V,W)$. We wish to show that $\phi$ satisfies $(\dagger)_j$ whenever $V_j$ is transverse to $L$. For $j=0$, the claim is that $\phi_0(L)\subset M$; this is clear from taking the limit in $\bP\Hom(V,W)$.

We now proceed by induction on $j$; we have already noted that the transversality condition for $V_j$ implies the same for $V_0,\ldots,V_{j-1}$. We assume further that $L\cap V_j\neq 0$, otherwise there is nothing further to prove. For $h=0,\ldots, j-1$, the map $L\cap V_h\to \im(\phi_h)$ induced by $\phi_h$ is surjective, as $L\cap V_h$ is transverse to $\ker(\phi_h)=V_{h+1}$ in $V_h$ and $L\cap V_{h+1}\neq0$. Thus, there exist decompositions $V_h=V_{h+1}\oplus V'_{h+1}$, where $V'_{h+1}\subset L$, for $h=0,\ldots, j-1$, and $\phi_{h}$ maps $V'_{h+1}$ isomorphically to a subspace of $\pr_h(M)$, by the inductive hypothesis.

Now, let $v_j\in L\cap V_j$ be any vector. If $\phi_j(v_j)\neq0$, then by the discussion of \S\ref{sec_limits}, we have
\begin{equation*}
\phi_j(v_j)=\lim_{t\to 0}\phi^t(\langle V'_1,\ldots,V'_{j},v_j\rangle),
\end{equation*}
where the right-hand side is viewed as a line in $W_j=W/\langle \phi_0(V'_1),\ldots,\phi_{j-1}(V'_{j})\rangle$. Because $\langle V'_1,\ldots,V'_{j},v_j\rangle\subset L$ and $\phi^t\in\Inc'(L,M)$, it follows that $\phi_j(v_j)\subset \pr_j(M)$.

Conversely, suppose that $\phi$ satisfies $(\dagger)_j$ whenever $V_j$ is transverse to $L$. We wish to express $\phi$ as the limit of linear maps $\phi^t:V\to W$ of full rank for which $\phi^t(L)\subset M$. We first reduce by induction to the case $\ell=1$. The case $\ell=0$ is easy, dealing only with linear equations on $\bP\Hom(V,W)$, so we assume that $\ell>1$.

If either the intersection of $L$ and $V_{\ell-1}$ is non-transverse or zero (in the language of Lemma \ref{check_dagger_once}, if $j_V<\ell-1$), then we have no constraints on the maps $\phi_{\ell-1},\phi_\ell$, and the length 1 collineation $\{\phi_{\ell-1},\phi_\ell\}$ from $V_{\ell-1}$ to $W_{\ell-1}$ is a limit of full rank collineations $\phi^t_{\ell-1}:V_{\ell-1}\to W_{\ell-1}$, by the irreducibility of $\Coll(V_{\ell-1},W_{\ell-1})$. We may then replace the last two maps $\phi_{\ell-1},\phi_\ell$ with a generic $\phi^t_{\ell-1}$. On the other hand, suppose that the intersection of $L$ and $V_{\ell-1}$ is transverse and non-zero. First, if $\pr_{\ell-1}(M)=W_{\ell-1}$, which is to say, that $j_W<\ell-1$, then again there are no constraints on the maps $\phi_{\ell-1},\phi_\ell$, and we again invoke the irreducibility of $\Coll(V_{\ell-1},W_{\ell-1})$. Finally, if instead $\pr_{\ell-1}(M)$ is a hyperplane in $W_{\ell-1}$, then we may apply the case $\ell=1$ (along with, if applicable, the property $(\dagger)_\ell$) to replace $\{\phi_{\ell-1},\phi_\ell\}$ with a generic $\phi^t_{\ell-1}:V_{\ell-1}\to W_{\ell-1}$, with the property that $\phi^t_{\ell-1}(L\cap V_{\ell-1})\subset\pr_{\ell-1}(M)$. The claim for all $\ell$ now follows by induction.

We therefore assume henceforth that $\ell=1$. First, suppose that $L$ is transverse to $V_1$ and that $r_1>k$, in which case the intersection of $V_1$ and $L$ is non-trivial. Then, consider a family (parametrized by $t$) of $(d+1)\times(r+1)$ matrices of the form
\vspace{25pt}
\begin{equation*}
A^t=
\begin{bNiceMatrix}
a'_{0,0}t & \cdots & a'_{0,k}t & 0 & \cdots & 0 &  0 & \cdots & 0 \\
a'_{1,0}t & \cdots & a'_{1,k}t & a'_{1,k+1}t & \cdots & a'_{1,r_1-1}t & a_{1,r_1} & \cdots & a_{1,r}  \\
\vdots & \vdots &\vdots & \vdots & \vdots & \vdots & \vdots & \vdots & \\
a'_{d,0}t & \cdots & a'_{d,k}t & a'_{d,k+1}t & \cdots & a'_{d,r_1-1}t & a_{d,r_1} & \cdots & a_{d,r}  \\
\CodeAfter
  \OverBrace[shorten,yshift=3pt]{1-1}{2-6}{V_1}
  \UnderBrace[shorten,yshift=1.5mm]{last-4}{last-last}{L}
 \end{bNiceMatrix}
\end{equation*}
\vspace{25pt}

\noindent representing a family of linear maps $\phi^t:V\to W$. The subspace $V_1\subset V$ of dimension $r_1=r+1-r_0$ is represented by the first $r_1$ columns, and the subspace $L\subset V$ of codimension $k+1$ is represented by the last $r-k$ columns. The subspace $M\subset W$ is the hyperplane of column vectors vanishing in the first coordiante. Now, if the $a'_{i,j}$ and $a_{i,j}$ are chosen generally, then the maps $\phi^t$ are injective in a neighborhood of $t=0$, and limit to a general collineation of type $(r_0,r_1)$ with $\ker(\phi_0)=V_1$ and satisfying $(\dagger)_0$ and $(\dagger)_1$. Indeed, $\phi_0$ is the linear map obtained by setting $t=0$, and $\phi_1$ is obtained by restricting to the first $r_1$ columns, dividing by $t$, and post-composing with the quotient $W\to W/\im(\phi_0)$.

In exactly the same way, if $r_1\le k$ and $V_1\cap L=0$, then one can take instead the one-parameter family of matrices
\vspace{25pt}
\begin{equation*}
A^t=
\begin{bNiceMatrix}
a'_{0,0}t & \cdots & a'_{0,r_1-1}t & a_{0,r_1} & \cdots & a_{0,k} &  0 & \cdots & 0 \\
a'_{1,0}t & \cdots & a'_{1,r_1-1}t & a_{1,r_1} & \cdots & a_{1,k} & a_{1,k+1} & \cdots & a_{1,r}  \\
\vdots & \vdots &\vdots & \vdots & \vdots & \vdots & \vdots & \vdots & \\
a'_{d,0}t & \cdots & a'_{d,r_1-1}t & a_{d,r_1} & \cdots & a_{d,k} & a_{d,k+1} & \cdots & a_{d,r}  \\
\CodeAfter
  \OverBrace[shorten,yshift=3pt]{1-1}{2-3}{V_1}
  \UnderBrace[shorten,yshift=1.5mm]{last-7}{last-last}{L}
\end{bNiceMatrix}
\end{equation*}
\vspace{25pt}

\noindent to read off a one-parameter family in $\Inc(L,M)\subset \Coll(V,W)$ with limit equal to $\phi$.

Now, suppose that $L$ is not transverse to $V_1$. Then, for some $s\le k,r_1-1$, we may represent a 1-parameter family of maps $\phi^t:V\to W$ by the matrix
\vspace{25pt}
\setcounter{MaxMatrixCols}{20}
\begin{equation*}
A^t=
\begin{bNiceMatrix}
a'_{0,0}t & \cdots & a'_{0,s-1}t & 0 & \cdots & 0 &  0 & \cdots & 0 & a_{0,s+r-k} & \cdots & a_{0,r} \\
a'_{1,0}t & \cdots & a'_{1,s-1}t & a'_{1,s}t & \cdots & a'_{1,r_1-1}t &  a_{1,r_1} & \cdots & a_{1,s+r-k-1} & a_{1,s+r-k} & \cdots & a_{1,r} \\
\vdots & \vdots & \vdots & \vdots & \vdots & \vdots & \vdots & \vdots & \vdots & \vdots & \vdots & \vdots \\
a'_{d,0}t & \cdots & a'_{d,s-1}t & a'_{d,s}t & \cdots & a'_{d,r_1-1}t &  a_{d,r_1} & \cdots & a_{d,s+r-k-1} & a_{d,s+r-k} & \cdots & a_{d,r}
\CodeAfter
  \OverBrace[shorten,yshift=3pt]{1-1}{2-6}{V_1}
  \UnderBrace[shorten,yshift=1.5mm]{last-4}{last-9}{L}
\end{bNiceMatrix}
\end{equation*}
\vspace{25pt}

\noindent The subspace $V_1\subset V$ is again represented by the first $r_1$ columns, but now $L$ is represented by columns $s$ through $s+r-k-1<r$. In particular, the last column corresponds to a non-zero vector of $V$ not in $\langle V_1,L\rangle $, yet $V_1\cap L\neq0$.

If the $a'_{i,j}$ and $a_{i,j}$ are chosen generally, then the maps $\phi^t$ are injective in a neighborhood of $t=0$, and limit to a general collineation of type $(r_0,r_1)$ with $\ker(\phi_0)=V_1$ and satisfying $(\dagger)_0$. Here, there is \emph{no condition} on $\phi_1$, which is again obtained by dividing the first $r_1$ columns by $t$. Indeed, the image of $\phi_0$ is not constrained to lie in $M$, corresponding to the fact that $a_{0,r}\neq0$ generically, so any map $\phi_1:V_1\to W/\im(\phi_0)$ can be obtained with the appropriate choice of $a'_{i,j}$ . This completes the proof.
\end{proof}

By a straightforward parameter count, one can deduce:

\begin{cor}\label{inc_exp_dim}
For any type $\overrightarrow{r}=(r_0,\ldots,r_\ell)$ and any $L,M$ as above, the intersection 
\begin{equation*}
\Inc(L,M)\cap \Coll_{\overrightarrow{r}}(V,W)\subset\Coll(V,W)
\end{equation*}
is pure of the expected codimension $(r-k)+\ell$.
\end{cor}

\section{From complete collineations to curve counts}\label{coll_to_count_sec}

In this section, we relate the curve counts $\T^{\bP^r}_{g,\vn,d}$ to the cycles $\Inc(L,M)$ defined in the previous section, proving Theorem \ref{thm:coll}.

\subsection{Setup}

We first recall the relevant definitions. Let $(C,p_1,\ldots,p_n)\in\cM_{g,n}$ be a general curve. Let $X_1,\ldots,X_n\subset\bP^r$ be general linear spaces of dimensions $k_1,\ldots,k_n<r$, respectively. For $j=0,1,\ldots,r-1$, let $n_j$ be the number of $X_i$ of dimension $j$, and write $\vn=(n_0,\ldots,n_{r-1})$

\begin{defn}\label{def:counts}
Let $d\ge0$ be an integer for which \eqref{general_numerology} holds. Then, we define $\T^{\bP^r}_{g,\vn,d}$ to be the number of non-degenerate maps $f:C\to\bP^r$ of degree $d$ with $f(p_i)\in X_i$.

When $\vn=(n,0,\ldots,0)$, the counts $\Tev^{\bP^r}_{g,n,d}:=\T^{\bP^r}_{g,\vn,d}$ are the \emph{geometric Tevelev degrees} of $\bP^r$.
\end{defn}

Throughout, we will often abbreviate $\vn=(n,0,\ldots,0)$ by simply $n$ when there is no opportunity for confusion.

The Brill-Noether-Petri theorem ensures that, under the assumption \eqref{general_numerology} and the stated generality assumptions, the number of maps in question is finite and transverse, in the sense that such maps admit no first-order deformations.

Let $V,W$ be vector spaces of dimensions $r+1,d+1$ as before, with $r\le d$. Let $L_1,\ldots,L_n\subset V$ be general linear subspaces of dimensions $r-k_1,\ldots,r-k_n$, and let $M_1,\ldots,M_n\subset W$ be general hyperplanes. The $L_i$ will eventually play the roles of the (duals of the) subspaces $X_i\subset\bP^r$ in the enumerative problem of interest, and the $M_i$ will play the roles of the (spaces of sections vanishing at the) points $p_i\in C$. 

\begin{defn}
We define the subscheme
\begin{equation*}
Y_{r,\vn,d}:=\bigcap_{i=1}^{n}\Inc(L_i,M_i)\subset \Coll(V,W).
\end{equation*}
and the subscheme
\begin{equation*}
Z_{r,\vn,d}:=\pi(Y_{r,\vn,d})\subset \Gr(r+1,W).
\end{equation*}
\end{defn}

\subsection{Specializing the $M_i\subset W$}

We now specialize the previous discussion of complete collineations to the following setting. Let $V=\bC^{r+1}$ and let $W=H^0(\bP^1,\cO(d))$. As discussed in \S\ref{sec:intro_coll}, the vector space $W$ will eventually appear as the space of sections on a rational component of our degenerated curve, see also Example \ref{map_and_coll}. Let $L_1,\ldots,L_n\subset V$ be general linear subspaces of dimensions $r-k_1,\ldots,r-k_n$ as above, and let $p_0,p_1,\ldots,p_n\in\bP^1$ be \emph{arbitrary} distinct points. For each $p_i$, let $M_{i}\subset W$ be the hyperplane of $d$-forms vanishing at $p_i$. Define the subschemes $\Inc(L_i,M_i)\subset\Coll(V,W)$ as before.  

More generally, to any point $p\in\bP^1$ we may associate the hyperplane $M_p\subset W$ of $d$-forms vanishing at $p$. We refer to this 1-parameter family of hyperplanes as the \textit{bp-hyperplanes} (where bp stands for ``base-point''). We say that a collineation $\phi\in\Coll(V,W)$ has a \emph{base-point} at $p$ if $\im(\phi)\subset M_p$, and is \textit{bpf} (base-point free) if the image of $\phi$ is not contained in any bp-hyperplane.

Let
\begin{equation*}
0=F_0\subset F_1\subset\cdots\subset F_{d+1}=W
\end{equation*}
be the complete flag defined by $F_h=H^0(\bP^1,\cO(d)(-(d+1-h)p_0))$, that is, $F_h\subset W$ is the subspace of sections vanishing to order $d+1-h$ at $p_0$.

We consider the subschemes
\begin{equation*}
Y^{\pt}_{r,\vn,d}:=\bigcap_{i=1}^{n}\Inc(L_i,M_i)\subset\Coll(V,W)
\end{equation*}
If $Y^{\pt}_{r,\vn,d}$ is irreducible of the expected codimension $|\vn|$, then as cycle \textit{classes} in $H^{2|\vn|}(\Coll(V,W))$, then we have 
\begin{equation*}
[Y^{\pt}_{r,n,d}]=\prod_{j=1}^r\gamma_j^{n_{r-j}}.
\end{equation*}
Note, however, that this transversality is not immediate from a Kleiman-Bertini-type argument, as the $M_i\subset W$ are not a general collection of hyperplanes. Nevertheless, we will show that the needed transversality holds, in addition, in the presence of a Schubert condition defined with respect to $F$.

\begin{prop}\label{coll_genericity_points}
Let $\lambda=(\lambda_0,\cdots,\lambda_r)$ be a partition, and let $\Sigma^F_{\lambda}\subset\Gr(r+1,W)$ be the corresponding Schubert variety with respect to the flag $F$ defined above. Assume as above that the points $p_0,p_1,\ldots,p_n\in \bP^1$ are distinct and the subspaces $L_1,\ldots,L_n\subset V$ are general. Then, the intersection
\begin{equation*}
\pi^{-1}(\Sigma^F_{\lambda})\cap Y^{\pt}_{r,\vn,d}\subset\Coll(V,W)
\end{equation*}
is irreducible and generically smooth of the expected codimension $|\lambda|+|\vn|$.

Moreover, any general point $\phi$ of the intersection is a collineation of full rank, bpf except possibly at $p_0$ (that is, with image contained in no bp-hyperplane except possibly $M_{p_0}=F_d$), and $\im(\phi)$ has ramification sequence at $p_0$ given \textit{exactly} by $\lambda$.
\end{prop}

Recall that the \emph{ramification sequence} of a linear series $U\subset H^0(C,\cL)$ of rank $r$ (which is to say that $\dim(U)=r+1$) on a curve $C$ at a point $p\in C$ is the sequence of integers $(\mu_0,\ldots,\mu_r)$ for which $\dim(U\cap H^0(C,\cL(-\mu_j-r+j)p))\ge j+1$ for $j=1,\ldots,r+1$ and each $\mu_j$ is as large as possible. Here, we have $C=\bP^1$, $p=p_0$, and $\cL=H^0(\bP^1,\cO(d))$. In particular, taking $j=r+1$, part of the claim is that $\im(\phi)\subset F_{d+1-\mu_r}$, but $\im(\phi)\not\subset F_{d-\mu_r}$.

We will prove Proposition \ref{coll_genericity_points} via a dimension count on strata of $\Coll(V,W)$. Fix a collineation type $\overrightarrow{r}=(r_0,\ldots,r_\ell)$ and a nested sequence of subsets 
\begin{equation*}
S_\ell\subseteq S_{\ell-1}\subseteq\cdots\subseteq S_0\subseteq\{p_1,\ldots,p_n\}.
\end{equation*} 
Write $s_j=|S_j|$ for each $j$. Let $\Coll^\circ_{\overrightarrow{r},S}\subset \Coll(V,W)$ be the locally closed subscheme of collineations $$\phi=\{\phi_j:V_j\to W_j\}_{j=0}^{\ell}$$ of type $\overrightarrow{r}$ which furthermore have the property, for all $j=0,1,\ldots,\ell$, that $$\phi_j(V_j)\subset \pr_j(M_i)$$ if and only if $p_i\in S_j$. That is, $S_\ell$ is the set of base-points of $\phi$, and more generally, $S_j$ is the set of base-points of $\phi$ ``up to level $j$.''

We will first need the following lemma. 

\begin{lem}\label{coll_schubert_pullback}
The intersection $\Coll^\circ_{\lambda,\overrightarrow{r},S}:=\pi^{-1}(\Sigma^F_{\lambda})\cap\Coll^\circ_{\overrightarrow{r},S}$, if non-empty, is irreducible of the expected codimension $$|\lambda|+\ell+\sum_{j=0}^{\ell}r_js_j$$ in $\Coll(V,W)$. Moreover, given a general point $\phi$ of $\Coll^\circ_{\lambda,\overrightarrow{r},S}$, the ramification sequence of $\im(\phi)$ at $p_0$ is exactly $\lambda$.
\end{lem}

\begin{proof}
Recall that we have defined the complete flag
\begin{equation*}
0=F_0\subset F_1\subset\cdots\subset F_{d+1}=W
\end{equation*}
where $F_h\subset H^0(\cO(d))$ is the subspace of sections vanishing to order $d+1-h$ at $p_0$. Define also a complete flag
\begin{equation*}
0=F'_0\subset F'_1\subset\cdots\subset F'_{d+1}=W
\end{equation*}
with the following property: for $j=0,1,\ldots,\ell$, the component $F'_{d+1-|S_j|}\subset H^0(\cO(d))$ is the subspace of sections vanishing along all of the points in $S_j$. This only determines some of the components of $F'$, but all others may be chosen generally (and play no role in what follows). Note that the flags $F,F'$ are transverse, in the sense that $\dim(F_h\cap F'_{h'})=\max(h+h'-(r+1),0)$ for all $h,h'$; this follows from the fact that distinct points impose independent conditions of sections of $\cO(d)$.

Now, consider the partial flag variety $\Fl_{\overrightarrow{r}}(W)$ of flags
\begin{equation*}
U_0\subset U_1\subset\cdots\subset U_\ell\subset W,
\end{equation*}
where $\dim(U_j)=r_0+\cdots+r_j$ for each $j$. Then, the condition that $U_\ell\in \Sigma^F_{\lambda}$ cuts out a Schubert subvariety $\Sigma^F$ of $\Fl_{\overrightarrow{r}}(W)$ of codimension $|\lambda|$. The condition that $U_j\subset F'_j$ for all $j=0,\ldots,\ell$ cuts out a Schubert subvariety $\Sigma^{F'}$ of $\Fl_{\overrightarrow{r}}(W)$ codimension $\sum_{j=0}^{\ell}r_js_j$. Because these two Schubert subvarieties are cut out by transverse flags, their intersection (a \emph{Richardson variety}) $\Sigma^{F,F'}$ is irreducible of codimension $|\lambda|+\sum_{j=0}^{\ell}r_js_j$.

Now, the scheme $\Coll^\circ_{\lambda,\overrightarrow{r},S}\subset\Coll(V,W)$ is the pullback of $\Sigma^{F,F'}$ by the smooth map
\begin{equation*}
\Coll^{\circ}_{\overrightarrow{r}}(V,W)\to \Fl_{\overrightarrow{r}}(W),
\end{equation*}
sending $U_j$ to the preimage of $\im(\phi_j)$ in $W$. As $\Coll^{\circ}_{\overrightarrow{r}}(V,W)$ has codimension $\ell$ in $\Coll(V,W)$, it follows that $\Coll^\circ_{\lambda,\overrightarrow{r},S}$ is irreducible of the claimed codimension in $\Coll(V,W)$.

The remaining claim about the ramification sequence of $\im(\phi)$ at $p_0$ at a general point of $\Coll^\circ_{\lambda,\overrightarrow{r},S}$ follows from repeating the argument with any partition $\lambda'$ strictly containing $\lambda$, as the locus of $\phi$ with larger ramification sequence has positive codimension in $\Coll^\circ_{\lambda,\overrightarrow{r},S}$.
\end{proof}

\begin{proof}[Proof of Proposition \ref{coll_genericity_points}]
Fix $\overrightarrow{r},S$, and consider the product $\Coll^\circ_{\lambda,\overrightarrow{r},S}\times(\bP^r)^n$. Regarding the $M_i\subset W$ as fixed and the $L_i\subset V$ as varying, we form the loci $\Inc(L_i,M_i)$ \textit{relatively} over $P=\prod_{i=1}\Gr(r-k_i,V)$, parametrizing the possible choices of $L_i$. Let $\cY\subset \Coll^\circ_{\lambda,\overrightarrow{r},S}\times P$ be the intersection of the relative incidence loci, whose fiber over a point $(L_1,\ldots,L_n)\in P$ is the restriction of 
\begin{equation*}
\pi^{-1}(\Sigma^F_{\lambda})\cap Y^{\pt}_{r,\vn,d}\subset\Coll(V,W)
\end{equation*}
to $\Coll^\circ_{\lambda,\overrightarrow{r},S}$.

Consider now the other projection $\cY\to \Coll^\circ_{\lambda,\overrightarrow{r},S}=\pi^{-1}(\Sigma^F_{\lambda})\cap\Coll^\circ_{\overrightarrow{r},S}$. By Proposition \ref{inc_set_theoretic}, the fiber over $\phi$ consists of collections of linear subspaces $L_i\subset V$ with the property that:
\begin{itemize}
\item either $\phi_{j_i+1}(L_{i}\cap V_{j_i+1})\subset \pr_{j_i+1}(M_i)$ (property $(\dagger)_{j_i+1}$), or
\item $L_{i}$ and $V_{j_i+1}$ fail to be transverse in $V$,
\end{itemize}
where $j_i$ is the largest value of $j$ for which $\im(\phi_j)\subset\pr_{j}(M_i)$, by Lemma \ref{check_dagger_once} and its proof. By convention, we take $j_i=-1$ if $\im(\phi_0)\not\subset M_i$. The second condition is never satisfied when $j_i=-1$, and both of the conditions are vacuously satisfied for $j_i=\ell$. 

The first condition imposes $\max((r-k_i)-(r_0+\cdots+r_{j_i}),0)$ conditions on $L_i$ (note that this is valid even when $j_i=\ell$). The second imposes $|(r-k_i)-(r_0+\cdots+r_{j_i})|+1$ conditions if $j_i\ge0$ and is never satisfied if $j_i=-1$; in particular, we have strictly more conditions than in the first case. In both cases, the subvariety of $\Gr(r-k_i,V)$ defined by these conditions is a Schubert variety, and in particular is irreducible and generically smooth. Combining with Lemma \ref{coll_schubert_pullback} and restricting to a general point $(L_1,\ldots,L_n)\in\prod_i\Gr(r-k_i,V)$, we conclude that the intersection of interest is irreducible and generically smooth of codimension at least
\begin{equation*}
\left(\sum_{i=1}^{n}(r-k_i)-\sum_{j=0}^{\ell}r_js_j\right)+\left(|\lambda|+\ell+\sum_{j=0}^{\ell}r_js_j\right)=\sum_{i=1}^{n}(r-k_i)+|\lambda|+\ell
\end{equation*}
in $\Coll(V,W)$. This number is at least the expected codimension, with equality only when $\ell=0$, so it follows that the intersection in question must be supported on the open stratum, on which it has expected dimension.

The statements about $\phi$ being bpf (except possibly at $p_0$) and having ramification exactly given by $\lambda$ at $p_0$ follow from repeating the dimension counts with $W$ replaced by an intersection of bp-hyperplanes or $\lambda$ replaced with a larger partition. Note in particular that requiring $\im(\phi)\subset M_i$ imposes $r+1$ conditions, which is strictly more than the $r-k_i$ conditions lost from $\Inc(L_i,M_i)$; requiring instead that $\im(\phi)$ is contained in \textit{any} bp-hyperplane imposes $r>0$ additional conditions, as the space of such hyperplanes is 1-dimensional.
\end{proof}

We remark that Proposition \ref{coll_genericity_points} immediately implies the following genericity statement by regenerating the $M_i$ back to general hyperplanes and $F$ back to a general flag. (Here, one can alternatively argue directly by Kleiman-Bertini.)

\begin{cor}\label{kleiman_inc}
Let $F$ be a general complete flag $0=F_0\subset F_1\subset\cdots\subset F_{d+1}=W$. Let $\lambda=(\lambda_0,\cdots,\lambda_r)$ be a partition, and let $\Sigma^F_{\lambda}\subset\Gr(r+1,W)$ be the corresponding Schubert cycle with respect to $F$.

Let $M_1,\ldots,M_n\subset W$ be general hyperplanes, and let $L_1,\ldots,L_n\subset V$ be general subspaces. Then, the intersection
\begin{equation*}
\pi^{-1}(\Sigma^F_{\lambda})\cap Y_{r,\vn,d}\subset\Coll(V,W)
\end{equation*}
is irreducible and generically smooth of the expected codimension $|\lambda|+|\vn|$.

Moreover, any general point $\phi$ of the intersection is a collineation of full rank, for which $\im(\phi)$ is not contained in $M_i$. Finally, for such a point, the dimensions $\dim(F_j\cap \im(\phi))$ are dictated \emph{exactly} by $\lambda$.
\end{cor}

We also record the following:

\begin{cor}\label{cor:lambdaempty}
The subschemes $Y^{\pt}_{r,\vn,d}$ and $Y_{r,\vn,d}$ of $\Coll(V,W)$ are irreducible and generically smooth of the expected codimension of $|\vn|=\sum_{i=1}^{n}(r-k_i)$. 
\end{cor}

\begin{proof}
Take $\lambda=(0,\ldots,0)$ in Proposition \ref{coll_genericity_points} and Corollary \ref{kleiman_inc}, respectively.
\end{proof}

In particular, we have $[Y^{\pt}_{r,\vn,d}]=[Y_{r,\vn,d}]=\prod_{j=1}^{r}\gamma_{j}^{n_{r-j}}$. 

\begin{defn}
We define
\begin{equation*}
\Gamma_{r,\vn,d}:=\pi_{*}([Y_{r,\vn,d}])=\pi_{*}\left(\prod_{j=1}^{r}\gamma_r^{n_{r-j}}\right) \in H^{2(|\vn|-r(r+2))}(\Gr(r+1,W)).
\end{equation*}
\end{defn}

\subsection{Degenerations of limit linear collineations}\label{sec:lls}

Let $B$ be the spectrum of a discrete valuation ring. Let $(\cC\to B,p_1,\ldots,p_n)$ be a stable pointed curve of genus $g$ whose general fiber $C_\eta$ is smooth and whose special fiber $C_0$ has the following form, depicted in Figure \ref{the curve C0}. To a rational component (spine) $C_{\spine}\cong\bP^1$, elliptic tails $E_1,\ldots,E_g$ are attached at general points $s_1,\ldots,s_g$, and a rational tail $R$ is attached at a general point $s_0$, so that $R$ contains the markings $p_1,\ldots,p_n$.

\begin{figure}
\begin{center}
\begin{tikzpicture}[xscale=0.4,yscale=0.30] [xscale=0.36,yscale=0.36]

    \node at (1,16) {$C_{\mathrm{sp}}$};
	\draw [ultra thick, black] (1,0) to (1,15);

	\draw [ultra thick, black] (-0.5,1) to (15,1);
	\node at (16,1) {$R$};
	\node at (1,1) {$\bullet$};
	\node at (0.3,1.8) {$s_0$};

	\node at (3,1) {$\bullet$};
	\node at (4,0) {$p_{1}$};
	\node at (7.5,0) {$\dots$};
	\node at (10,1) {$\bullet$};
	\node at (11,0) {$p_{n}$};

	\draw [ultra thick, black] (-0.5,7) to [out=0,in=180] (8,7) to [out=0,in=180] (13,8) to [out=0,in=180] (15,8);
	
	\draw [ultra thick, black] (-0.5,10) to [out=0,in=180] (8,10) to [out=0,in=180] (13,11) to [out=0,in=180] (15,11);
	
	\node at (16,8) {$E_1$};
	\node at (16,10) {$\vdots$};
	\node at (16,11) {$E_g$};
	
	\node at (1,7) {$\bullet$};
	\node at (1,10) {$\bullet$};
	\node at (0.3,7.8) {$s_1$};
	\node at (0.3,10.8) {$s_g$};

\end{tikzpicture}
\caption{The curve $C_0$.}\label{the curve C0}
\end{center}
\end{figure}

Let $\cG(\cC/B)$ be the space of relative limit linear series on $\cC$ of rank $r$ and degree $d$, see \cite{eh}. Let $\cG(\cC/B)^\circ$ be the open subset of limit linear series that do not have base-points at any of the $p_i$. On the special fiber, the base-point condition is taken as a constraint on the $R$-aspect of the limit linear series.

Let $\cP(\cC/B)\to \cG(\cC/B)$ be the projective bundle of $(r+1)$-tuples of sections, taken up to simultaneous scaling, where on the special fiber we take sections of the $R$-aspect. The associated vector bundle is $\Hom(\bC^{r+1},\cU_R)$, where $\cU_R$ is the universal rank $(r+1)$ bundle given by the universal linear series on $\cC_\eta$ and the $R$-aspect of the universal series on $\cC_0$. 

Finally, let $\Coll(\cC/B)\to \cP(\cC/B)$ be the fiber-wise blowup whose fibers are complete collineations $V\cong\bC^{r+1}\to U_R$, where the $U_R$ are the fibers of $\cU_R$. The points of $\Coll(\cC/B)$ on the general fiber are therefore linear series $(\cL_\eta,U_\eta\subset H^0(\cC,\cL_\eta))$ of rank $r$ and degree $d$ on $\cC_\eta$ with a collineation $V\to U_\eta$. If $h^0(\cC,\cL_\eta)\ge r+1$, then this data is equivalent to that of a collineation $V\to H^0(\cC,\cL_\eta)$, as the $(r+1)$-dimensional vector space $U_{\eta}$ can be recovered as the image of the collineation $V\to H^0(\cC,\cL_\eta)$. If instead $h^0(\cC,\cL_\eta)<r+1$, then the general fiber of $\Coll(\cC/B)$ contains no points over $\cL_\eta\in\Pic^d(\cC_\eta)$. The points of $\Coll(\cC,B)$ on the special fiber are limit linear series on $C_0$ with a collineation from $V$ to the $R$-aspect $U_R\subset H^0(C_{\spine},\cO(d))$. See also \cite[\S 6]{cl} for a detailed construction in the case $r=1$.

The loci $\Inc(L_i,M_i)$ can be globalized to $\Coll(\cC/B)$. More precisely, fix general linear spaces $L_i\subset V\cong\bC^{r+1}$ of codimension $k_i+1$. On the pullback $\Coll(\cC/B)^\circ$ of $\Coll(\cC/B)$ over $\cG(\cC/B)^\circ$, we have, for each $i$, a rank $r$ sub-bundle $\cM_i\subset \cU_R$ of sections vanishing at $p_i$. Then, define $\Inc(L_i,\cM_i)^\circ\subset\Coll(\cC/B)_{bpf}$ fiberwise by the usual incidence locus $\Inc(L_i,M_i)$. Define $\Inc(L_i,\cM_i)\subset\Coll(\cC/B)$ to be the closure of $\Inc(L_i,\cM_i)^\circ$, and $\cY^{\pt}_{g,r,\vn,d}\subset\Coll(\cC/B)$ to be the intersection of all of the $\Inc(L_i,\cM_i)$.

\begin{prop}\label{lls_count}
Assume \eqref{general_numerology} and denote the $(d+1)$-dimensional vector space $H^0(C_{\spine},\cO(d))$ by $W$. Then, the subscheme $\cY^{\pt}_{g,r,\vn,d}\subset\Coll(\cC/B)$ is finite and \'{e}tale of degree
\begin{equation*}
\int_{\Coll(V,W)}\pi^{*}(\sigma^g_{1^r})\cdot\prod_{j=1}^{r}\gamma_{j}^{n_{r-j}}
\end{equation*}
over $B$. 

Furthermore, any geometric point of $\cY^{\pt}_{g,r,\vn,d}$ over the generic point is a collineation of full rank, and is base-point free.
\end{prop}

\begin{proof}
We first show that the restriction of $\cY^{\pt}_{g,r,\vn,d}$ to the special fiber is reduced with the claimed number of points. Let $U$ be a limit linear series on $C_0$ with a collineation $\phi:\bC^{r+1}\to U_R\subset W$ lying in $\cY^{\pt}_{g,r,\vn,d}$. We first claim that the ramification sequence of $U$ on each $E_j$ must be $(d-r,d-r-1,\ldots,d-r-1)$ at the node, in which case there is a unique linear series on $E_j$ with this property. 

The argument is as in \cite[\S 2]{fl}. First, if there are strictly more than $r(d-r-1)+(d-r)$ ramification conditions on $E_j$, then no such $E_j$-aspect can exist, and when there are exactly this number of conditions, then no such $E_j$-aspect can exist unless the ramification sequence is equal to $(d-r,d-r-1,\ldots,d-r-1)$, in which case it is uniquely determined. On the other hand, if there are strictly fewer than this number of conditions at any $E_j$, then the total number of ramification conditions imposed on $C_{\spine}$ at the points $s_j$ is at least $rg+1$, and so the same must be true at $s_0$ on $R$. However, Proposition \ref{coll_genericity_points} then shows that no collineation $\bC^{r+1}\to W$ coming from a point of $\cY^{\pt}_{g,r,\vn,d}$ with $rg+1$ Schubert conditions at $p_0$ can exist.

Therefore, the ramification sequence at each $s_j$ on $C_{\spine}$ must be (at least) $(1,\ldots,1,0)$, corresponding to the class $\sigma_{1^r}$. Arguing as above, the number of Schubert conditions imposed at $s_0$ on $C_{\spine}$ must be exactly $(r+1)(d-r)-rg$, and at $s_0$ on $R$ must be exactly $rg=(r+1)(d-r)-r(n-r-2)$. For any pair $(\lambda,\overline{\lambda})$ of complementary partitions in $(r+1)^{d-r}$ of these sizes, we obtain, by the Mukhin-Tarasov-Varchenko Theorem \cite[Theorem 6.6]{mtv} (guaranteeing the transversality of the cycles in the first integral) and Proposition \ref{coll_genericity_points} (guaranteeing the transversality in the second),
\begin{equation*}
\left(\int_{\Gr(r+1,d+1)}\sigma^g_{1^r}\cdot \sigma_\lambda\right)\cdot\left(\int_{\Coll(V,W)}\pi^{*}(\sigma_{\overline{\lambda}})\cdot\prod_{j=1}^{r}\gamma_{j}^{n_{r-j}}\right)
\end{equation*}
reduced points of $\cY^{\pt}_{g,r,\vn,d}$. Summing over all pairs of complementary partitions $(\lambda,\overline{\lambda})$ inside $(d-r)^{r+1}$ and using the fact that $\sigma_\lambda,\sigma_{\overline{\lambda}}$ are dual bases of $H^{*}(\Gr(r+1,d+1))$ yields the first part of the claim.

Now, because $\cY^{\pt}_{g,r,\vn,d}$ is cut out locally on the special fiber of $\Coll(\cC/B)$ by $rg+|\vn|=(r+1)(d+1)-1$ equations, and is reduced of the expected dimension of zero in the special fiber, the same is true in the general fiber. Finally, a geometric point of the generic fiber of length greater than 0 or with a base point would specialize to a limit linear series with a complete collineation on $R$ with the same property (possibly after semi-stable reduction, if the base-point specializes to the node $s_0\in R$), which does not exist by Proposition \ref{coll_genericity_points}. This completes the proof.
\end{proof}

\begin{proof}[Proof of Theorem \ref{thm:coll}]
A point of $\cM^\circ_{g,n}(\bP(V^\vee),d)$ is interpreted as an \emph{injective} map $\phi:V\to H^0(C,\cL)$, where $\cL$ is a line bundle of degree $d$ on $C$, such that the image of $\phi$ is not contained in any $H^0(C,\cL(-p_i))$. Proposition \ref{lls_count} implies that any point in the general fiber of $\cY^{\pt}_{g,r,\vn,d}\to B$ has this form. In fact, a standard local computation shows that the pullback of the map $\tau:\cM^\circ_{g,n}(\bP(V^\vee),d)\to\cM_{g,n}\times X^n$ by the inclusion of $[(C,p_1,\ldots,p_n)]\times\prod X_i$ is isomorphic as a \textit{scheme} to the general fiber of $\cY^{\pt}_{g,r,\vn,d}\to B$. The Theorem now follows from the rest of Proposition \ref{lls_count}, giving the degree of $\cY^{\pt}_{g,r,\vn,d}$ over $B$.
\end{proof}

\section{Comparison to Young tableaux}\label{asymptotic_sec}

\begin{defn}\label{def_coeffs}
Define the integers $\Gamma^\lambda_{r,\vn,d}$ by
\begin{equation*}
\Gamma_{r,\vn,d}=:\sum_{\lambda}\Gamma^\lambda_{r,\vn,d}\cdot\sigma_\lambda,
\end{equation*}
where the sum is over $\lambda\subset(d-r)^{r+1}$ with $|\lambda|=|\vn|-r(r+2)$.
\end{defn}
Clearly, these coefficients determine the classes $\Gamma_{r,\vn,d}$, and by Theorem \ref{thm:coll}, the curve counts $\T_{r,\vn,d}$ via Schubert calculus on $\Gr(r+1,d+1)$. Moreover, because $\Gamma_{r,\vn,d}$ is the push-forward of an effective cycle from $\Coll(V,W)$, it is itself effective on $\Gr(r+1,W)$. Therefore, for any $\lambda$, the coefficient $\Gamma^\lambda_{r,\vn,d}$ is equal, by Kleiman-Bertini, to the number of points of intersection in $\Gr(r+1,d+1)$ of $Z_{r,\vn,d}$ and a generic Schubert variety of class $\sigma_{\overline{\lambda}}$, where $\overline{\lambda}$ is the complement of $\lambda$ in the rectangular partition $(d-r)^{r+1}$. In particular, we have $\Gamma^\lambda_{r,\vn,d}\ge0$ for all $\lambda$.

In this section, we show that, under some conditions, we have the simple formula $\Gamma^\lambda_{r,\vn,d}=|\SSYT_{r+1}(\lambda)|$, where the right hand side is a count of semi-standard Young tableaux, see \S\ref{yt_sec}. We do so by replacing the space $\Coll(V,W)$ with a more na\"{i}ve resolution of the rational map $\pi:\bP\Hom(V,W)\dashrightarrow\Gr(r+1,W)$. The calculation is essentially the main one of \cite{fl}, in a more general setting. We will deduce that $\Gamma^\lambda_{r,\vn,d}\le |\SSYT_{r+1}(\lambda)|$ in general, and that the virtual (Gromov-Witten) count $(r+1)^g$ gives an upper bound for $\T^{\bP^r}_{g,\vn,d}$.

\subsection{Calculation on $S(V,W)$}

\begin{defn}
Let $S(V,W)\subset \bP\Hom(V,W)\times\Gr(r+1,W)$ be the locus of pairs $(\phi,U)$, where $\phi:V\to W$ is a non-zero linear map (up to scaling) and $U\in \Gr(r+1,W)$ satisfies $U\supset \im(\phi)$. 

Abusing notation, let $\pi:S(V,W)\to \Gr(r+1,W)$ be the projection remembering $U$, and let $\psi:S(V,W)\to\bP\Hom(V,W)$ be the projection remembering $\phi$.
\end{defn}

\begin{equation*}
\xymatrix{
 & S(V,W) \ar[dl]_{\pi} \ar[dr]^{\psi} & \\
 \Gr(r+1,W) & & \bP\Hom(V,W)
 }
 \end{equation*}

The map $\pi$ is the $\bP^{(r+1)^2-1}$-bundle associated to the vector bundle $\cU^{r+1}$ on $\Gr(r+1,W)$, where $\cU$ is the universal subbundle. Indeed, the fiber over a point $U\in \Gr(r+1,W)$ is the space of non-zero linear maps $\bC^{r+1}\to U$, up to scaling. Over the locus of injective maps $\phi\in\bP\Hom(V,W)$, the map $\psi$ is an isomorphism, as it must be the case that $U=\im(\phi)$. When $r=1$, we have $S(V,W)=\Coll(V,W)$. Indeed, if $\phi:V\to W$ has rank 1, the point $(\phi,U)\in S(V,W)$ is identified with the length 1 collineation given by $\phi_0=\phi$ and $\phi_1:\bC\cong\ker(\phi)\to\coker(\phi)$ the unique map for which $\im(\{\phi_j\}_{j=0}^1)=U$.

Let $L\subset V$ be a linear subspace of dimension $r-k$ and let $M\subset W$ be a hyperplane. We (again, abusively) denote by $\Inc'(L,M)\subset S(V,W)$ the \emph{pullback} of $\Inc'(L,M)\subset\bP\Hom(V,W)$ by $\psi$. Then, in terms of the projective bundle description of $\pi:S(V,W)\to \Gr(r+1,W)$, the locus $\Inc'(L,M)$ is scheme-theoretically a relative linear subspace of codimension $r-k$ and class $c_1(\cO_{\bP(\cU^{r+1})}(1))^{r-k}$.

Fix now general linear subspaces $L_1,\ldots,L_n\subset V$, where $L_i$ has codimension $r-k_i$, and general hyperplanes $M_1,\ldots,M_n\subset W$. Define
\begin{equation*}
Y'_{r,\vn,d}=\bigcap_{i=1}^{n}\Inc'(L_i,M_i)\subset S(V,W).
\end{equation*}

\begin{prop}\label{asymptotic_genericity}
Assume one of the following:
\begin{enumerate}
\item[(a)] $n_0\ge d+2$, or
\item[(b)] $n_0=\cdots=n_{r-2}=0$ (that is, $k_i=r-1$ for all $i$, so the $L_i$ are all lines), or
\item[(c)] $r=2$ and $n\ge d+3$.
\end{enumerate}
Then, the intersection $Y'_{r,\vn,d}$ lies generically in the locus where $\phi$ is injective and has image not contained in any of the $M_i$. Furthermore, the subscheme $Y'_{r,\vn,d}$ is generically smooth of the expected codimension of $|\vn|$.
\end{prop}

\begin{proof}
We first prove the claim under assumption (a). Without loss of generality, suppose that $L_1,\ldots,L_{d+2}\subset V$ are all  (general) hyperplanes.

Let $(\phi,U)$ be any point (we will see that we need not assume this point be general) of the intersection of $Y'_{r,\vn,d}$, and suppose that $\phi$ has rank $r_0\le r$. Let $L=\ker(\phi)\subset V$. Let $k$ be the number of hyperplanes among $L_1,\ldots,L_{d+2}\subset V$ containing $L$. As the $L_i$ are general, we have $k\le r_0$. Now, for any $i$ for which $L\not\subset L_i$, we must have $\im(\phi)\subset M_i$. Indeed, if $\im(\phi)\not\subset M_i$, then $\phi^{-1}(M_i)$ is a proper hyperplane containing, and thus equal to, $L_i$, but $\phi^{-1}(0)=L$, a contradiction. Therefore, $\im(\phi)\subset W$ has dimension $r_0$, but is contained in the intersection of at least $d+2-k\ge d+2-r_0$ general hyperplanes, which is impossible. We conclude that $\phi$ must be injective.

From here onward, it suffices to work instead on the open locus where $\phi$ is injective, and therefore on the open subset of $\bP\Hom(V,W)^{\circ}$ consisting of injective maps. We now relativize the construction of the loci $\Inc'(L_i,M_i)\subset\bP\Hom(V,W)^{\circ}$. Define
\begin{equation*}
\cY\subset \bP\Hom(V,W)^{\circ}\times \prod_{i=1}^{n}\Gr(r-k_i,V)\times (\bP(W^\vee))^n
\end{equation*}
by the intersection of the linear subspaces $\Inc'(L_i,M_i)$ upon restriction to any collection of $L_i\in V$ and $M_i\in W$. We claim that $\cY$ is smooth of the expected codimension $|\vn|$. Indeed, the projection 
\begin{equation*}
\cY\to \bP\Hom(V,W)^{\circ}\times \prod_{i=1}^{n}\Gr(r-k_i,V)
\end{equation*}
is a relative product of Grassmannian bundles of relative dimension $\sum_{i=1}^{n}k_i$: the fiber over a point $(\phi,L_1,\ldots,L_n)$ consists of space of hyperplanes $M_i$ with $M_i\supset \phi(L_i)\cong\bC^{r-k_i}$. Now, by generic smoothness, restricting to a general point $(\{L_i\},\{M_i\})\in   \prod_{i=1}^{n}\Gr(r-k_i,V)\times (\bP(W^\vee))^n$ yields that the intersection of the $\Inc'(L_i,M_i)$ is irreducible and generically smooth of the expected dimension. (In fact, this intersection is smooth, because it is a linear subspace of $\bP\Hom(V,W)$.)

Finally, suppose now that for a general point $\phi\in Y'_{r,\vn,d}$, we have $\im(\phi)\subset M_i$ for some $i$; without loss of generality, take $i=n$. (We may drop the assumption that the $L_1,\ldots,L_{d+2}$ are all hyperplanes, in case $n_0=d+2$, which would not allow $L_n$ to be a hyperplane; the argument that follows works for $L_n$ of any dimension.) Then, we are in the following situation: we have general subspaces $L_1,\ldots,L_{n-1}\subset V$ and hyperplanes $M_1\cap M_n,\ldots,M_1\cap M_{n-1}\subset M_n$, and, at least set-theoretically, we have
\begin{equation*}
\bigcap_{i=1}^{n}\Inc'(L_i,M_i)=\left(\bigcap_{i=1}^{n-1}\Inc'(L_i,M_i\cap M_n)\right)\cap S(V,M)
\end{equation*}
We may argue by induction that the right hand side (if non-empty) has codimension $|\vn|-(r-k_n)$ in $S(V,M)$, and hence the left hand side has codimension $|\vn|+k_n+1$ in $S(V,W)$. Thus, the closure of the original point $\phi$ has codimension greater than the expected, and $\phi$ cannot be general. This completes the proof under assumption (a).

Under assumption (b), Proposition \ref{inc_set_theoretic} implies that $Y'_{r,\vn,d}$ pulls back exactly to $Y_{r,\vn,d}\subset\Coll(V,W)$ under the map $b':\Coll(V,W)\to S(V,W)$ remembering the data of $\phi_0$ and $\im(\phi)$. Then, the claim follows from Proposition \ref{cor:lambdaempty}.

Finally, assume that we are in the setting of assumption (c). Let $(\phi,U)$ be a general point of $Y'_{r,\vn,d}$. If $\phi$ is injective, then the further statements follows from Proposition \ref{cor:lambdaempty}. If $\phi$ has rank 2, then $(\phi,U)$ can be identified with a point of $\Coll_{(2,1)}(V,W)$, and we get a contradiction, again from Proposition \ref{cor:lambdaempty}. Finally, if $\phi$ has rank 1, then for each $i$, either $M_i\supset \im(\phi)$ or $L_i\subset \ker(\phi)$. The latter can hold for at most two $i$ if the $L_i$ are general, so at least $n-2\ge d+1$ of the $M_i$ must contain $\im(\phi)$. However, if the $M_i$ are general, the intersection of $d+1$ of the $M_i$ is zero, a contradiction.
\end{proof}

One can refine the arguments of Proposition \ref{asymptotic_genericity} to obtain the same conclusion under more general assumptions, but we do not carry out a detailed analysis here.

\begin{cor}\label{asymptotic_formula}
Assume the setup of Proposition \ref{asymptotic_genericity}, and one of the hypotheses (a), (b), (c). Then,
\begin{equation*}
\Gamma_{r,n,d}=\sum_{a_0+\cdots+a_r=|\vn|-r(r+2)}\sigma_{a_0}\cdots\sigma_{a_r}.
\end{equation*}
In particular,
\begin{equation*}
\Gamma^{\lambda}_{r,\vn,d}=|\SSYT_{r+1}(\lambda)|.
\end{equation*}
for all $\lambda\subset(d-r)^{(r+1)}$ with $|\lambda|=|\vn|-r(r+2)$.
\end{cor}

\begin{proof}
The second formula follows from the first by the Pieri rule. Proposition \ref{asymptotic_genericity} shows that 
\begin{equation*}
\Gamma_{r,n,d}=\pi_{*}([Y'_{r,\vn,d}])=\pi_{*}\left(\prod_{i=1}^{n}[\Inc'(L_i,M_i)]\right),
\end{equation*}
where here $\pi$ denotes the map $\pi:S(V,W)\to \Gr(r+1,W)$. The class of $\Inc'(L_i,M_i)$ is cut out by $r-k_i$ relative linear conditions in the projective bundle $\pi:S(V,W)\to \Gr(r+1,W)$. Therefore, the class of $\bigcap_{i=1}^{n} \Inc'(L_i,M_i)$ is given by $c_1(\cO(1))^{|\vn|}$, where $\cO(1)$ denotes the relative hyperplane class, and upon pushforward, we obtain the Segre class
\begin{align*}
s_{|\vn|-r(r+2)}(\cU^{r+1})&=\{s(\cU)^{r+1}\}_{|\vn|-r(r+2)}\\
&=\left\{\left(\sum_{a\ge0}\sigma_a\right)^{r+1}\right\}_{|\vn|-r(r+2)}\\
&=\sum_{a_0+\cdots+a_r=|\vn|-r(r+2)}\sigma_{a_0}\cdots\sigma_{a_r},
\end{align*}
as needed.
\end{proof}

\begin{cor}\label{asymptotic_tev}
Assume \eqref{general_numerology} and one of the hypotheses (a), (b), (c), and furthermore that $d\ge g+r$ (this is immediate in case (a)). Then,
\begin{equation*}
\T^{\bP^r}_{g,\overrightarrow{n},d}=(r+1)^g.
\end{equation*}
\end{cor}
\begin{proof}
By Theorem \ref{thm:coll} and Corollary \ref{asymptotic_formula}, this reduces to the statement
\begin{equation*}
\int_{\Gr(r+1,d+1)}\sigma_{1^r}^g\cdot\left(\sum_{a_0+\cdots+a_r=|\vn|-r(r+2)}\sigma_{a_0}\cdots\sigma_{a_r}\right)=(r+1)^g,
\end{equation*}
which is \cite[Theorem 1.3]{grb}. Alternatively, one can compare $\T^{\bP^r}_{g,\overrightarrow{n},d}$ to the Gromov-Witten count, which is equal to $(r+1)^g$. The bound $d\ge g+r$ ensures that only non-degenerate maps out of $C$ contribute to the virtual count.
\end{proof}

\subsection{Upper bounds}\label{upper_bound}

\begin{lem}\label{2incs_intersection}
Fix general linear spaces $L,L'\subset V$ and hyperplanes $M,M'\subset W$. Let $\iota:\Coll(V,M)\hookrightarrow\Coll(V,W)$ be the inclusion. Let $M''$ be a hyperplane in $M$. Suppose that $\dim(L)+\dim(L')\ge r+1$. 

Then, we have an equality in $H^{*}(\Coll(V,W))$
\begin{equation*}
[\Inc(L,M)]\cdot[\Inc(L',M')]=[Y]+\iota_{*}[\Inc(L\cap L',M'')],
\end{equation*}
where $Y\subset \Coll(V,M)$ is a subscheme of pure codimension $\dim(L)+\dim(L')$.
\end{lem}

If $\dim(L)+\dim(L')=r+1$, then $[\Inc(L_1\cap L_2,M')]=[\Coll(V,M)]$. If $\dim(L)+\dim(L')<r+1$, then one has an analogous statement that we will not need.

\begin{proof}
Consider a 1-parameter family in which $M'$ degenerates to $M$. We get a corresponding family of intersections $\Inc(L,M)\cap\Inc(L',M')$. The class in $H^{*}(\Coll(V,W))$ of the flat limit is then equal to $[\Inc(L,M)]\cdot[\Inc(L',M')]$. Let $M''=\lim(M\cap M')\subset M$. Then, it suffices to identify $\iota(\Inc(L\cap L',M''))$ (scheme-theoretically) with a component of the flat limit; $Y$ will then be the union of the remaining components of the flat limit.

In fact, because the intersection $\Inc(L,M)\cap\Inc(L',M')$ is contained generically in the locus of collineations of full rank, as is $\iota(\Inc(L\cap L',M''))$, it suffices to work on the locus of collineations of full rank, which we identify with an open subset of $\bP\Hom(V,W)$. Then, the loci in question are cut out by linear equations, and it is straightforward to verify the claim directly.
\end{proof}
%

\begin{prop}\label{bound_schubert}
Assume $r,\vn,d$ are arbitrary, and $\lambda\subset(d-r)^{(r+1)}$ with $|\lambda|=|\vn|-r(r+2)$. Then, we have $\Gamma^{\lambda}_{r,\vn,d}\le|\SSYT_{r+1}(\lambda)|$.
\end{prop}

\begin{proof}
For any integer $k\ge0$, we write $\lambda+kr$ for the partition $(\lambda_0+kr,\ldots\lambda_r+kr)$ and $\vn+k(r+1)$ for the tuple $(n_0+k(r+1),n_1,\ldots,n_r)$. We will show more generally that
\begin{equation*}
\Gamma^{\lambda}_{r,\vn,d} \le \Gamma^{\lambda+r}_{r,\vn+(r+1),d+r}
\end{equation*}
so that
\begin{equation*}
\Gamma^{\lambda}_{r,\vn,d} \le \Gamma^{\lambda+kr}_{r,\vn+k(r+1),d+kr}=|\SSYT_{r+1}(\lambda)|
\end{equation*}
for any $k$ sufficiently large, by Corollary \ref{asymptotic_formula} under assumption (a).

Fix now $L_1,\ldots,L_{n+r+1}\subset V$ general, where $L_{n+1},\ldots,L_{n+r+1}\subset V$ are hyperplanes and the dimensions of the remaining $L_i$ are determined (in some order) by $\vn$. Fix general hyperplanes $M_1,\ldots,M_{n+r+1}\subset W$. Consider the product
\begin{equation*}
\bigcap_{i=1}^{n+r+1}[\Inc(L_i,M_i)]
\end{equation*}
and its pushforward by $\pi$. We will apply Lemma \ref{2incs_intersection} $r$ times, first with $$(L,L',M,M')=(L_{n+r},L_{n+r+1},M_{n+r},M_{n+r+1}).$$ We have
\begin{equation*}
\prod_{i=1}^{n+r+1}[\Inc(L_i,M_i)]=\left(\prod_{i=1}^{n+r-1}[\Inc(L_i,M_i)]\right)\cdot\left([Y]+\iota_{*}[\Inc(L_{n+r}\cap L_{n+r+1},M'')]\right),
\end{equation*}
for some $Y\subset\Coll(V,W)$ and hyperplane $M''\subset M_{n+r}$. 

For general choices of $L_i,M_i$, the intersection of $Y$ with the remaining $\Inc(L_i,M_i)$ may be taken to have the expected codimension. Indeed, we may apply Kleiman-Bertini to a stratum $\Coll_{\overrightarrow{r}}(V,W)$ on which $Y$ is supported, and on which, by Corollary \ref{inc_exp_dim}, the $\Inc(L_i,M_i)$ have the expected codimension. In particular, this intersection with $Y$ pushes forward to an effective cycle on $\Gr(r+1,W)$. On the other hand, by the projection formula and the fact that $\iota^{*}\Inc(L_i,M_i)=\Inc(L_i,M_i\cap M_{n+r})$, we have
\begin{align*}
&\left(\prod_{i=1}^{n+r-1}[\Inc(L_i,M_i)]\right)\cdot\iota_{*}[\Inc(L_{n+r}\cap L_{n+r+1},M'')]\\
=\iota_{*}&\left(\left(\prod_{i=1}^{n+r-1}\Inc(L_i,M_i\cap M_{n+r})\right)\cdot[\Inc(L_{n+r}\cap L_{n+r+1},M'')]\right).
\end{align*}

We repeat the argument, next with $L=L_{n+r-1}$ and $L'=L_{n+r}\cap L_{n+r+1}$, which now has dimension $r-1$. After $r+1$ iterations, we obtain that
\begin{equation*}
\pi_{*}\left(\bigcap_{i=1}^{n+r+1}[\Inc(L_i,M_i)]\right)-\iota^{(r)}_{*}\left(\bigcap_{i=1}^{n}[\Inc(L_i,M_i)]\right)
\end{equation*}
is an effective cycle on $\Gr(r+1,W)$, where $\iota^{(r)}:\Gr(r+1,W^r)\to\Gr(r+1,W)$ is the inclusion induced by a subspace $W^r\subset W$ of codimension $r$ (which can be taken to be equal to $M_{n+2}\cap\cdots \cap M_{n+r+1}$).\

The coefficient of $\sigma_{\lambda+r}$ in this difference is on the one hand non-negative, and on the other equal to precisely $\Gamma^{\lambda+r}_{r,\vn+(r+1),d+r}-\Gamma^{\lambda}_{r,\vn,d}$. This completes the proof.
\end{proof}


\begin{cor}\label{Tev_bound}
Assume \eqref{general_numerology}. Then, $\T^{\bP^r}_{g,\overrightarrow{n},d}\le (r+1)^g$.
\end{cor}

\begin{proof}
Recall from the proof of Proposition \ref{bound_schubert} that $\Gamma^{\lambda}_{r,\vn,d} \le \Gamma^{\lambda+kr}_{r,\vn+k(r+1),d+kr}$ for any $k\ge0$. Therefore, we also have, by Theorem \ref{thm:coll}, that
\begin{align*}
\T^{\bP^r}_{g,\overrightarrow{n},d}&=\sum_\lambda \Gamma^{\lambda}_{r,\vn,d} \cdot \int_{\Gr(r+1,d+1)}\sigma_{1^r}^g\cdot\sigma_\lambda\\
&\le\sum_\lambda \Gamma^{\lambda+kr}_{r,\vn+k(r+1),d+kr} \cdot \int_{\Gr(r+1,d+1+kr)}\sigma_{1^r}^g\cdot\sigma_{\lambda+kr}\\
&\le \T^{\bP^r}_{g,\overrightarrow{n}+k(r+1),d+kr}\\
&=(r+1)^g
\end{align*}
if $k$ is sufficiently large.
\end{proof}

In particular, we have the inequality $\Tev^{\bP^r}_{g,n,d}\le \Tev^{\bP^r}_{g,n+r+1,d+r}$ for geometric Tevelev degrees. Thus, $\Tev^{\bP^r}_{g,n,d}$ is bounded below by the Castelnuovo count \eqref{castelnuovo_formula}, so is therefore strictly positive. This implies that the map $\cM^{\circ}_{g,n}(\bP^r,d)\to\cM_{g,n}\times X^n$ is dominant whenever $n\ge r+2$ and $n=\frac{r+1}{r}\cdot d-g+1$, which was previously proven  in more generality by E. Larson \cite[Corollary 1.3]{larson}.

\section{Reduction to $d=n-1$}\label{reduction_sec}

We show in this section that the classes $\Gamma_{r,\vn,d}$ with $d=n-1$, where we recall that $n$ is the sum of the entries of the vector $\vn$, determine the classes $\Gamma_{r,\vn,d}$ for $d,n$ arbitrary. This will follow from the two statements below.

\begin{prop}\label{increase_degree}
Let $\lambda\subset(d-r)^{r+1}$ be a partition with $|\lambda|=|\vn|-r(r+2)$. Then, we have $\Gamma^\lambda_{r,\vn,d}=\Gamma^\lambda_{r,\vn,d+1}$.
\end{prop}

Note that the right hand side of Proposition \ref{increase_degree} makes sense because we also have $\lambda\subset((d+1)-r)^{r+1}$ and the relation $|\lambda|=|\vn|-r(r+2)$ does not depend on $d$.

\begin{prop}\label{vanishing_coeff}
Let $\lambda\subset(d-r)^{r+1}$ be a partition with $|\lambda|=|\vn|-r(r+2)$. Suppose further that $\lambda_0>n-r-1$. Then, we have $\Gamma^\lambda_{r,\vn,d}=0$.
\end{prop}

Note that in order for the condition $\lambda_0>n-1-r$ of Proposition \ref{vanishing_coeff} to be satisfied, we need
\begin{equation*}
n\le \lambda_0+r \le (d-r)+r\le d.
\end{equation*}

\begin{figure}
\caption{The line $n=d-1$ determines all values of $\Gamma^\lambda_{r,\vn,d}$}\label{d_reduction_diagram}
\begin{tikzpicture} [xscale=0.35,yscale=0.35]

\draw (0,0) -- (0,10);
\draw (0,0) -- (13,0);
\draw (3,0) -- (13,10);
\node at (0,10.5) {\tiny{$\lambda_0$}};
\node at (13.5,0) {\tiny{$d$}};
\node at (14,11) {\tiny{$\lambda_0=d-r$}};
\node at (-3,4) {\tiny{$\lambda_0=n-r-1$}};
\draw[densely dotted] (0,4) -- (13,4);
\draw[line width = 0.75mm] (7,4) -- (7,0);
\node at (7,-1) {\tiny{
$d=n-1$}};
\node at (12,6) {\tiny{$\Gamma^{\lambda}_{r,\vn,d}=0$}};
\node at (11,2.5) {\tiny{$\Gamma^{\lambda}_{r,\vn,d}=\Gamma^{\lambda}_{r,\vn,n-1}$}};
\usetikzlibrary{arrows.meta}
\draw [->] (7,1) -- (8.5,1);
\draw [->] (7,1) -- (5.5,1);
\end{tikzpicture}
\end{figure}

Fix $r,\vn$, and suppose that all classes are $\Gamma_{r,\vn,n-1}$ known, which is to say that all coefficients $\Gamma^\lambda_{r,\vn,d}$ are known for any $\lambda\subset(d-r)^{r+1}$ whenever $n=d-1$. We explain how to determine all coefficients $\Gamma^\lambda_{r,\vn,d}$ for $d,\lambda$ arbitrary. We refer to Figure \ref{d_reduction_diagram}, which depicts the $(d,\lambda_0)$-plane. The coefficient $\Gamma^\lambda_{r,\vn,d}$ from Definition \ref{def_coeffs} only makes sense when $\lambda_0\le d-r$, that is, for points lying below the diagonal line $\lambda_0=d-r$. Above the horizontal line $\lambda_0=n-r-1$, all coefficients $\Gamma^\lambda_{r,\vn,d}$ vanish, by Proposition \ref{vanishing_coeff}. Proposition \ref{increase_degree} shows that, for any given $\lambda$, the coefficient $\Gamma^\lambda_{r,\vn,d}$ is independent of $d$, as long as $\lambda_0\le d-r$. In particular, below the line $\lambda_0=n-r-1$, we may take $d=n-1$. Thus, the values of $\Gamma^\lambda_{r,\vn,d}$ are determined by those along the bold segment.

\begin{proof}[Proof of Proposition \ref{increase_degree}]
Write $\overline{\lambda}$ for the complement of $\lambda$ in $(d-r)^{r+1}$ and $\overline{\lambda}'$ for the complement of $\lambda$ in $(d+1-r)^{r+1}$. Note that 
\begin{equation*}
\sigma_{\overline{\lambda}'}=\sigma_{\overline{\lambda}}\cdot\sigma_{1^{r+1}},
\end{equation*}
in $H^{*}(\Gr(r+1,d+2))$, and furthermore that $\sigma_{1^{r+1}}$ may be regarded as the class of the loci of $(r+1)$-planes contained in a fixed hyperplane.

Fix vector spaces $V,W,W'$ of dimensions $r+1,d+1,d+2$, respectively, along with an inclusion $W\subset W'$. Fix general linear spaces $L_i\subset V$ of codimension $k_i+1$, and hyperplanes $M'_i\subset W'$ intersecting $W$ transversely in the hyperplanes $M_i\subset W$. Then, the space $\Coll(V,W)$ is realized naturally as the subscheme of $\Coll(V,W')$ of collineations with image contained in $W$. Furthermore, the subschemes $\Inc(L_i,M'_i)$ pull back to $\Inc(L_i,M_i)$ under this inclusion.

We therefore conclude:
\begin{align*}
\Gamma^\lambda_{r,\vn,d}&=\int_{\Gr(r+1,W)}\Gamma_{r,\vn,d}\cdot\sigma_{\overline{\lambda}}\\
&=\int_{\Coll(V,W)}\prod_{j=1}^{r}\gamma_{j}^{n_{r-j}}\cdot\pi^{*}(\sigma_{\overline{\lambda}})\\
&=\int_{\Coll(V,W')}\prod_{j=1}^{r}\gamma_{j}^{n_{r-j}}\cdot\pi^{*}(\sigma_{\overline{\lambda}})\cdot\pi^{*}(\sigma_{1^{(r+1)}})\\
&=\int_{\Gr(r+1,W')}\Gamma_{r,\vn,d+1}\cdot\sigma_{\overline{\lambda}'}\\
&=\Gamma^\lambda_{r,\vn,d+1}.
\end{align*}
\end{proof}

We now turn to the proof of Proposition \ref{vanishing_coeff}. Fix $d+1$ hyperplanes $M_0,\ldots,M_{d}\subset W$ in linearly general position. In this case, the $M_i$ have no moduli: if we take $w_0,\ldots,w_d$ to be a basis of $W$, then we may suppose that $M_i$ is the hyperplane $\langle w_0,\ldots,\widehat{w_i},\ldots,w_d\rangle $.

\begin{defn}\label{subgroup_def}
For any $n\le d+1$, let $K_n\subset GL(W)$ denote the subgroup consisting of automorphisms of $W$ that stabilize (but do not necessarily fix) all of $M_0,\ldots,M_{n-1}$.
\end{defn}
With respect to the basis $w_0,\ldots,w_d$, the subgroup $K_n$ consists of the matrices $g$ for which $g_{ij}=0$ whenever $i\le n-1$ and $i\neq j$. From here, it is clear that $\dim(K_n)=n+(d+1)(d+1-n)$. When $n=d+1$, the subgroup $K_n$ is the maximal torus $T\subset GL(W)$ of diagonal matrices.

\begin{defn}
Fix a subspace $U\subset W$ of dimension $r+1$. We denote by $K_n(U)\subset K_n$ the subgroup of automorphisms that stabilize the $M_i$ and \emph{fix} $U$. 
\end{defn}

\begin{lem}\label{stab_exp_dim}
Suppose $U$ satisfies the property that it intersects any intersection of $k\le n$ of the hyperplanes $M_i$ (for \emph{all} $i=0,1,\ldots,d$) in the expected dimension of $\max(0,r+1-k)$. Then, we have $\dim(K_n(U))=(d-r)(d+1-n)$.
\end{lem}

In the language of matroids, the hypothesis is that the matroid determined by $U$ with respect to the $M_i$ is \emph{uniform}.

\begin{proof}
Regarding $K_n(U)$ as a locally closed subset of the $(d+1)^2$-dimensional affine space of linear operators $g:W\to W$, we will show that $K_n(U)$ is cut out by $(r+1)(d+1)+n(d-r)$ independent (affine-)linear equations on the open subset $GL(W)$ of invertible operators. The claim will then follow from the fact that $K_n(U)$ is non-empty (as it contains the identity) and that $$(d+1)^2-(r+1)(d+1)-n(d-r)=(d-r)(d+1-n).$$

We work in the basis $\langle w_0,\ldots,w_d\rangle$ of $W$ as above. For each hyperplane $M_i$, $i=0,1,\ldots,n-1$, fix an arbitrary $(d-r)$-element subset $S_i$ of $\{0,1,\ldots,d\}-\{i\}$. We claim that the dimension $(d-r)$ subspace $W_i\subset M_i$ spanned by the vectors $w_j$, $j\in S_i$, is a complement to $U\cap M_i$ in $M_i$. Indeed, we have $\dim(U\cap M_i)=r$, $\dim(W_i)=d-r$, and $$\dim((U\cap M_i)\cap W_i)=\dim(U\cap W_i)=0,$$ because $W_i$ is an intersection of $r+1$ of the hyperplanes $M_1,\ldots,M_{d+1}$.

Fix now a basis $$u_i=\sum_{j=0}^{d}u_{ij}w_j$$ of $U$, where $i=0,1,\ldots,r$, and abusively denote by $U$ the $(r+1)\times(d+1)$ matrix $\{u_{ij}\}$, where $0\le i\le r$ and $0\le j\le d$.

We now cut out $K_n(U)\subset GL(W)$ by linear equations coming from the following two sources:
\begin{itemize}
\item $g(u)=u$ for $u\in U$ ($(r+1)(d+1)$ equations)
\item $g(W_i)\subset M_i$, $i=0,1,\ldots,n-1$ ($n(d-r)$ equations)
\end{itemize}
First, we see that the above two conditions cut out $K_n(U)$ set-theoretically. Indeed, once we have the first condition that $g$ restricts to the identity on $U$, to ensure that $g(M_i)\subset M_i$, we only need to check this condition on the complement $W_i$ of $U\cap M_i\subset M_i$.

Regarding $g\in GL(W)$ as a matrix $\{g_{ij}\}$, where $0\le i,j\le d$, let us make the conditions more explicit. That $g$ restricts to the identity on $U$ amounts to the affine-linear equation $$u_{ij}=u_{i0}g_{j0}+u_{i1}g_{j1}+\cdots+u_{id}g_{jd}$$ for $0\le i\le r$ and $0\le j\le d$. On the other hand, the condition $g(W_i)\subset M_i$ amounts to the linear condition $g_{ij}=0$ for $i=0,1,\ldots,n-1$ and $j\in S_i$.

It remains to check that these $(r+1)(d+1)+n(d-r)$ equations are linearly independent. Suppose instead that 
\begin{equation*}
\sum_{0\le i\le r,0\le j\le d}\alpha_{ij}(u_{i0}g_{j0}+u_{i1}g_{j1}+\cdots+u_{id}g_{jd})+\sum_{0\le i\le n-1,j\in S_{i}}\beta_{ij}g_{ij}=0 
\end{equation*}
for scalars $\alpha_{ij},\beta_{ij}$. For fixed $i,j$ with $0\le i,j\le d$, the coefficient of $g_{ij}$ above is
\begin{equation*}
\sum_{k=0}^{r} \alpha_{ki}u_{kj}+\beta_{ij}
\end{equation*}
if $i\le n-1$ and $j\in S_i$, and simply $\sum_{k=0}^{r} \alpha_{ki}u_{kj}$ otherwise (e.g. if $i\ge n$); by assumption, this coefficient is 0 for all $i,j$. Summing over all $j$, we have
\begin{align*}
0&=\sum_{j=0}^d\left(\sum_{k=0}^{r} \alpha_{ki}u_{kj}\right)w_j+\sum_{j\in S_i}\beta_{ij}w_j\\
&=\sum_{k=0}^{r} \alpha_{ki}\left(\sum_{j=0}^du_{kj}w_j\right)+\sum_{j\in S_i}\beta_{ij}w_j\\
&=\sum_{k=0}^{r} \alpha_{ki}u_k+\sum_{w_j\in W_i}\beta_{ij}w_j.
\end{align*}
The first term is an element of $U$ and the second is an element of $W_i$, but because $U\cap W_i=0$, it must be the case that all of the $\alpha_{ij}$ (reindexed now as $\alpha_{ki}$) and $\beta_{ij}$ are equal to 0.
\end{proof}

\begin{proof}[Proof of Proposition \ref{vanishing_coeff}]
Recall that we are free to assume that $n\le d$. In particular, the general hyperplanes $M_1,\ldots,M_n$ may be taken to be those in the previous discussion above.

Let $\mu=(\mu_0,\ldots,\mu_r)=\overline{\lambda}$ be the complementary partition to $\lambda$ in $(d-r)^{r+1}$, so that $\mu_i=\lambda_{d-r-i}$ for $i=0,1,\ldots,r$, and in particular, $\mu_r\le d-n$. It suffices to show that $$\int_{\Gr(r+1,W)}\Gamma_{r,\vn,d}\cdot\sigma_\mu=0.$$

By Corollary \ref{kleiman_inc}, for a general complete flag $0=F_0\subset F_1\subset\cdots\subset F_{d+1}=W$, the subscheme
\begin{equation*}
Y_{r,\vn,d}=\bigcap_{i=1}^{n}\Inc(L_i,M_i)\subset\Coll(V,W)
\end{equation*}
intersects the pullback by $\pi$ of the Schubert cycle $\Sigma^F_{\mu}$ in finitely many points, all of which come from full rank collineations whose image is not contained in any $M_i$. We will show that if a general flag $F$ has the property that the intersection $Y_{r,\vn,d}\cap \pi^{-1}(\Sigma^F_{\mu})$ is non-empty, then, in fact, the intersection is positive-dimensional, meaning that the generic number of intersection points must be zero.

Let $\phi\in\Coll(V,W)$ be a point of $Y_{r,\vn,d}\cap \pi^{-1}(\Sigma^F_{\mu})$. Then, $\phi:V\to W$ is a full-rank map with image contained in $F_{d+1-\mu_r}\supset F_{n+1}$. We may therefore take $\phi$ to be a full rank point of $\bP\Hom(V,F_{d+1-\mu_r})^\circ\subset \Coll(V,F_{d+1-\mu_r})$ contained in the incidence loci $\Inc(L_i,M_i\cap F_{d+1-\mu_r})$. 

Let $r'<r$ be the largest integer for which $\mu_{r'}>\mu_r$. Such an $r'$ exists because, if $\lambda_0=\cdots=\lambda_r\ge n-r$, then $|\lambda|>rn-r(r+2)=\vn-r(r+2)$, a contradiction. Then, writing $U=\im(\phi)$, we require $\dim(U\cap F_{d-r+1+r'-\mu_{r'}})\ge r'+1$; we may assume by a further application of Corollary \ref{kleiman_inc} that we have an equality. Let $U'=U\cap F_{d-r+1+r'-\mu_{r'}}$.

By the generality hypotheses, we may assume that the $n\le (d+1-\mu_r)-1$ hyperplanes $M_i\cap F_{d+1-\mu_r}\subset F_{d+1-\mu_r}$ are in linearly general position. We may also assume the subspaces $U',U\subset F_{d+1-\mu_r}$ satisfy the hypothesis of Lemma \ref{stab_exp_dim}, where we replace $W$ with $F_{d+1-\mu_r}$, by a standard incidence correspondence argument. Applying Lemma \ref{stab_exp_dim}, we have $\dim(K_n(U))<\dim(K_n(U'))$, and in particular there is a positive-dimensional family of automorphisms of $F_{d+1-\mu_r}$ fixing $U'$ but not $U$, and also stabilizing each $M_i\cap F_{d+1-\mu_r}$. Let $g\in K_n(U')\subset GL(F_{d+1-\mu_r})$ be any such automorphism.

Now, the full rank collineation $g\circ \phi:V\to F_{d+1-\mu_r}\subset W$, by construction, lies in $Y_{r,\vn,d}\cap \pi^{-1}(\Sigma^F_{\mu})$. Indeed, we have $g(U)=\im(g\circ \phi)\subset F_{d+1-\mu_r}$, so $\dim(g(U)\cap F_{d-r+i+1-\mu_i})\ge i+1$ for $i=r'+1,\ldots,r$, and the fact that $g(U')=U'$ guarantees the same for $i=0,1,\ldots,r'$. Moreover, because $g$ stabilizes the $M_i\cap F_{d+1-\mu_r}$, it remains the case that $g\circ\phi\in \Inc(L_i,M_i)$ for $i=1,2,\ldots,n$. Finally, because $g\notin K_n(U)$, we have $g\circ \phi\neq \phi$ as points of $\Coll(V,W)$ (note that $g$ cannot act on $U$ by $\lambda\cdot\Id$ for $\lambda\neq1$, because $g$ fixes $U'$ by assumption). Varying over all $g$, we obtain a positive-dimensional family in the intersection in question, and the proof is complete.
\end{proof}

\section{Torus orbits and geometric Tevelev degrees}\label{tevelev_sec}

We now complete the calculation of the geometric Tevelev degrees of $\bP^r$ via torus orbit closures on Grassmannians, proving Theorem \ref{thm:tevPr}. By Theorem \ref{thm:coll} and the results of the previous section, it suffices to compute the classes $\Gamma_{r,n,n-1}$. We therefore assume that $\dim(W)=n$, and that $L_1,\ldots,L_n\subset V$ and $M_1,\ldots,M_n\subset W$ are general hyperplanes. Recall from \ref{sec:tev_history} that we also assume that $n\ge r+2$.

Write simply
\begin{equation*}
Y:=Y_{r,n,d}=\bigcap_{i=1}^{n}\Inc(L_i,M_i)\subset\Coll(V,W),
\end{equation*}
which is irreducible and generically smooth of codimension $nr$ by Proposition \ref{cor:lambdaempty}, and write $Z=\pi(Y)$ for its scheme-theoretic image in $\Gr(r+1,W)$. Recall that we have defined
\begin{equation*}
\Gamma_{r,n,n-1}=\pi_{*}[Y]\in H^{2r(n-r-2)}(\Gr(r+1,W)).
\end{equation*}

\subsection{Torus orbit closures}\label{t-orbit}

Consider the standard action of $T\cong(\bC^{\times})^{n}$ on $W$, such that $M_1,\ldots,M_n\subset W$ are the torus-invariant hyperplanes, which induces an action on $\Gr(r+1,W)$. The key observation, which one may regard as an incarnation of the Gelfan'd-Macpherson correspondence \cite{gm}, is the following.

\begin{prop}\label{orbit_closure_interpretation}
The map $\pi$ is generically injective on $Y$. The image $Z$ is a generic $T$-orbit closure in $\Gr(r+1,W)$.
\end{prop}
\begin{proof}
For the first statement, it is enough to show that $\pi$ is injective on $Y$ upon restriction to the locus of full-rank collineations whose image $U$ is transverse to any intersection of the $M_i$. Indeed, this locus is non-empty and therefore (being open) dense in $Y$, as the  $L_i$ can be chosen generally. For such a $U$, any full-rank collineation in the fiber over $U$ in $Y$ is given by a linear isomorphism $\phi:V\to U$ with the property that $\phi(L_i)=M_i$ for $i=1,2,\ldots,r+1$. If such an isomorphism exists, by the transversality assumption, it must be unique.

In particular, the scheme-theoretic image $Z\subset\Gr(r+1,W)$ of $Y$ is irreducible and reduced of dimension $(r+1)(d+1)-1-nr=n-1$. Furthermore, because the $T$-action on $W$ stabilizes the $M_i$, it also stabilizes $Z$. On the other hand, the orbit closure of a general point $U\in Z$ as above is also irreducible of dimension $n-1$, as, by Lemma \ref{stab_exp_dim}, the stabilizer of $U$ consists only of scalar matrices, and is also contained in $Z$. Therefore, these two subschemes coincide.
\end{proof}

The final ingredient we will need to compute $\Tev^{\bP^r}_{g,n,d}$ is the following formula of Berget-Fink \cite[Theorem 5.1]{bf}. 

\begin{thm}\label{orbit_formula}
Let $Z_{gen}$ be a generic torus orbit closure on $\Gr(r+1,W)$. Then,
\begin{equation*}
[Z_{gen}]=\sum_{\mu\subset(n-r-2)^{r}}\sigma_{\mu}\sigma_{\overline{\mu}},
\end{equation*}
where $\overline{\mu}$ denotes the complement of $\mu$ inside the rectangular partition $(n-r-2)^{r}$. More precisely, if $\mu=(\mu_0,\ldots,\mu_{r-1},0)$, then $\overline{\mu}=(n-r-2-\mu_{r-1},\ldots,n-r-2-\mu_0,0)$.
\end{thm}

\begin{proof}[Proof of Theorem \ref{thm:tevPr}]
Let $V,W'$ be vector spaces of dimensions $r+1,d+1$, respectively, where we use $W'$ to distinguish from the vector space $W$ above of dimension $n$. Applying, in order, Theorem \ref{thm:coll}, Proposition \ref{vanishing_coeff}, and Proposition \ref{increase_degree}, we have
\begin{align*}
\Tev^{\bP^r}_{g,n,d}&=\int_{\Coll(V,W')}\pi^{*}(\sigma_{1^r}^g)\cdot (\gamma_r)^n\\
&=\int_{\Gr(r+1,W')}\sigma_{1^r}^g\cdot \Gamma_{r,n,d}\\
&=\int_{\Gr(r+1,W')}\sigma_{1^r}^g\cdot\left(\sum_{|\lambda|=r(n-r-2)}\Gamma^\lambda_{r,n,d}\cdot\sigma_{\lambda}\right)\\
&=\int_{\Gr(r+1,W')}\sigma_{1^r}^g\cdot\left(\sum_{\substack{|\lambda|=r(n-r-2) \\ \lambda_0\le n-r-1}}\Gamma^\lambda_{r,n,d}\cdot\sigma_{\lambda}\right)\\
&=\int_{\Gr(r+1,W')}\sigma_{1^r}^g\cdot\left(\sum_{\substack{|\lambda|=r(n-r-2) \\ \lambda_0\le n-r-1}}\Gamma^\lambda_{r,n,n-1}\cdot\sigma_{\lambda}\right).
\end{align*}

On the other hand, we also have, by Proposition \ref{orbit_closure_interpretation} and Theorem \ref{orbit_formula}, an equality of cycle classes on $\Gr(r+1,W)$
\begin{align*}
\sum_{\substack{|\lambda|=r(n-r-2) \\ \lambda_0\le n-r-1}}\Gamma^\lambda_{r,n,n-1}\cdot\sigma_{\lambda}&=\Gamma_{r,n,n-1}\\
&=[Z_{gen}]\\
&=\sum_{\mu\subset(n-r-2)^{r}}\sigma_{\mu}\sigma_{\overline{\mu}},
\end{align*}
where we recall that $\dim(W)=n$. The expansion in the Schubert basis of the last sum on $\Gr(r+1,d')$ is independent of the value of $d'$, as long as the classes $\sigma_\mu$ are interpreted to be zero whenever $\mu\not\subset(d'-r-1)^{r+1}$. Thus,
\begin{equation*}
\sum_{\substack{|\lambda|=r(n-r-2) \\ \lambda_0\le n-r-1}}\Gamma^\lambda_{r,n,n-1}\cdot\sigma_{\lambda}=\left(\sum_{\mu\subset(n-r-2)^{r}}\sigma_{\mu}\sigma_{\overline{\mu}}\right)_{\lambda_0\le n-r-1},
\end{equation*}
from which Theorem \ref{thm:tevPr} follows.
\end{proof}

\subsection{Degenerations of orbit closures: outline}\label{degen_orbit}

In what follows, we sketch an independent proof of Theorem \ref{orbit_formula}. The argument essentially appears in \cite[\S 7]{l_flag} after specializing from the more general setting of torus orbits in the full flag variety, but we explain the method in preparation of our proof of Theorem \ref{thm:countP2}. 

We outline the calculation in the language of maps to $\bP^r$, to emphasize the geometry. We are interested in the locus of linear series on $\bP^1$ underlying non-degenerate $f:\bP^1\to\bP^r$ of degree $n-1$ satisfying $n$ incidence conditions $f(p_i)=x_i$. We study such $f$ under degeneration of the points $x_i$. Specifically, choose a general hyperplane $H\subset\bP^r$, and move the points $x_i$ onto $H$ one-by-one. At each step, $f$ either remains non-degenerate, or becomes contained in $H$.

If the latter occurs after $x_\alpha\in H$ for some $\alpha$, then $p_{\alpha+1},\ldots,p_{n}$ must become base-points of $f$; one may regard this as a Schubert condition on the linear series underlying $f$. On the other hand, on the general fiber, there is also a section (secant) underlying $f$ that vanishes at the divisor $p_1+\cdots+p_{\alpha-1}$; this section vanishes in the limit. However, if we take the limit of $f$ in the space of complete collineations, rather than the space of maps, then the limit remembers this secant, giving a second Schubert condition on the special fiber.

From here, one repeats the degeneration, now moving the points $x_1,\ldots,x_\alpha\in H$ into a codimension 2 linear space, and so forth. As $f$ degenerates further, we keep track of two sets of Schubert conditions, which may be regarded as coming from base-point and secancy conditions on $f$, respectively. When $f$ becomes totally degenerate, these base-point and secancy conditions eventually become the classes $\sigma_{\mu}$, $\sigma_{\overline{\mu}}$, respectively, appearing in Theorem \ref{orbit_formula}. Our degeneration therefore replaces the orbit closure $Z\subset\Gr(r+1,W)$ in the end with a union of Richardson varieties (intersections of Schubert varieties) whose class is visibly equal to the right hand side of Theorem \ref{orbit_formula}.

\subsection{The degenerate components: first step}\label{sec_first_step}

We now set up the same degeneration on $\Coll(V,W)$. Fix $T$-invariant hyperplanes $M_1,\ldots,M_n\subset W$, and a line $\Lambda_1\subset V$. We introduce the following notation to be used for the rest of the paper: if $[\alpha_1,\alpha_2]$ is an interval, then $M_{[\alpha_1,\alpha_2]}$ denotes the intersection $M_{\alpha_1}\cap\cdots\cap M_{\alpha_2}$.

For some integer $\alpha\in[0,n-1]$, let $L_1,\ldots,L_{\alpha}\subset V$ be general hyperplanes containing $\Lambda_1$, and let $L_{\alpha+1},\ldots,L_n\subset V$ be general hyperplanes (with no other restrictions). We begin by defining subschemes $Y_{\alpha}$ and $Y_{(\alpha,0)}$ of $\Coll(V,W)$, along with their images $Z_{\alpha},Z_{(\alpha,0)}$ in $\Gr(r+1,W)$.

\begin{defn}
For $\alpha\in[0,n-1]$, define $Y_{\alpha}\subset \Coll(V,W)$ by the closure of the intersection
\begin{equation*}
\left(\bigcap_{i=1}^{n} \Inc(L_i,M_i)\right)\cap\Coll^\circ_{(r+1)}(V,W)
\end{equation*}
in $\Coll(V,W)$. Define $Z_{\alpha}=\pi(Y_{\alpha})\subset \Gr(r+1,W)$.
\end{defn}

Note that $Y=Y_0$ and $Z=Z_0$. In the language of degenerations of maps given above, one can regard $Y_{\alpha}$ as the closure of the locus of non-degenerate maps $f:\bP^1\to\bP^r$ with $f(p_i)=x_i$, after $\alpha$ of the points $x_i$ are specialized to lie on a hyperplane $H\subset\bP^r$. A general point $\phi\in Y_\alpha$ may be regarded as a $(d+1)\times(r+1)$ matrix 
\begin{equation*}
A_\alpha=
\begin{bmatrix}
0 & * & \cdots & *\\
\vdots & \vdots & \vdots & \vdots \\
0 & * & \cdots & *\\
* & * & \cdots & *\\
\vdots & \vdots & \vdots & \vdots \\
* & * & \cdots & *
\end{bmatrix}
\end{equation*}
whose first $\alpha$ rows have a 0 in the left-most column (corresponding to $\Lambda_1\subset V$), and whose other entries are generic. The assumption that $\alpha<n$ ensures that the left-most column of $A_\alpha$ is not identically zero. A general point $U\in Z_\alpha$ may be regarded as the span of the column vectors of $A_\alpha$ as above, and in fact, it is easily checked that $Z_\alpha$ is the $T$-orbit closure of $U$. The subscheme $Y_\alpha$ has the expected dimension of $n-1$, but the same is only true of $Z_\alpha$ if additionally $\alpha\le n-2$.

\begin{defn}
Suppose that $\alpha\in[r+1,n-1]$. Then, define $Y_{(\alpha,0)}\subset \Coll_{(r,1)}(V,W)$ to be the closure of the locus of complete collineations $\phi=\{\phi_j\}_{j=0}^{1}$ satisfying the following properties:
\begin{enumerate}
\item $\phi$ has type $(r,1)$, and $V_1=\ker(\phi_0)=\Lambda_1$,
\item the map $\phi_0:V/\Lambda_1\to W$ satisfies the incidence conditions $\Inc(L_i,M_i)$ for $i=1,2,\ldots,\alpha$,
\item the image $\im(\phi_0)$ is contained in $M_{[\alpha+1,n]}$, but none of $M_1,\ldots,M_\alpha$,
\item $\im(\phi_0)\cap M_{[1,\alpha-1]}=0$, but $\im(\phi)\cap M_{[1,\alpha-1]}\neq 0$.
\end{enumerate}
Define $Z_{(\alpha,0)}=\pi(Y_{(\alpha,0)})\subset \Gr(r+1,W)$.
\end{defn}

A general point of $Y_{(\alpha,0)}$ may be represented by a $(d+1)\times(r+1)$ matrix 
\begin{equation*}
A_{(\alpha,0)}=
\begin{bmatrix}
0 & * & \cdots & *\\
\vdots & \vdots & \vdots & \vdots \\
0 & * & \cdots & *\\
* & * & \cdots & *\\
* & 0 & \cdots & 0\\
\vdots & \vdots & \vdots & \vdots \\
* & 0 & \cdots & 0
\end{bmatrix}
\end{equation*}
whose first $\alpha-1$ rows have a 0 in the left-most column, and whose last $n-\alpha$ rows are zero in all other columns. The assumption that $\alpha>r$ ensures that the last $r$ columns are independent, if the non-zero entries are chosen generically. The map $\phi_0:V/\Lambda_1\to W$ is given by the last $r$ columns of $A_{(\alpha,0)}$, and the left-most column spans $\im(\phi)\cap M_{[1,\alpha-1]}$. A general point $U\in Z_{(\alpha,0)}$ is again given by the column span of $A_{(\alpha,0)}$, and $Z_{(\alpha,0)}$ is equal to the $T$-orbit closure of $U$. Both $Y_{(\alpha,0)},Z_{(\alpha,0)}$ have the expected dimension of $n-1$.

\begin{prop}\label{prop:degen_first}
Fix $\alpha\in[1,n-1]$. Then, there exists a 1-parameter degeneration of $Y_{\alpha-1}$ into a union of components including $Y_{\alpha}$ and $Y_{(\alpha,0)}$ with multiplicity 1. 

In particular, the same 1-parameter family degenerates $Z_{\alpha-1}$ into a union of components including $Z_{\alpha}$ and $Z_{(\alpha,0)}$ with multiplicity 1.

\end{prop}

By convention, if either component $Y_{\alpha},Y_{(\alpha,0)}$ is ``out of range,'' which is to say, undefined for a given value of $\alpha$, we simply ignore it from the conclusion of the first part of Proposition \ref{prop:degen_first}. We also ignore the subscheme $Z_{n-1}$, which has dimension strictly less than $n-1$, from the second part of the conclusion.

\begin{proof}
The degeneration in question sends the entry of $A_{\alpha-1}$ in the $\alpha$-th row and first column from a generic value (represented by $*$) to zero. Geometrically, this corresponds exactly to moving the point $x_\alpha$ onto $H$, or moving the hyperplane $L_\alpha$ to contain $\Lambda_1$. 

More precisely, let $Y^t_{\alpha-1}\subset\Coll(V,W)\times\bD$ be the family of subschemes over a disk $\bD$ such that, for $t\neq0$, a general point $\phi\in Y^t_{\alpha-1}$ takes the form
\begin{equation*}
A^t_{\alpha-1}=
\begin{bmatrix}
0 & * & \cdots & *\\
\vdots & \vdots & \vdots & \vdots \\
0 & * & \cdots & *\\
*t & * & \cdots & *\\
* & * & \cdots & *\\
\vdots & \vdots & \vdots & \vdots \\
* & * & \cdots & *
\end{bmatrix}
.
\end{equation*}
Taking all of the entries $*$ to be constant and sending $t\to 0$, we find the component $Y_{\alpha}$ in the limit of this degeneration with multiplicity 1.

On the other hand, the matrix
\begin{equation*}
(A^t_{\alpha-1})_0=
\begin{bmatrix}
0 & * & \cdots & *\\
\vdots & \vdots & \vdots & \vdots \\
0 & * & \cdots & *\\
*t & * & \cdots & *\\
*t & *t & \cdots & *t\\
\vdots & \vdots & \vdots & \vdots \\
*t & *t & \cdots & *t
\end{bmatrix}
\end{equation*}
is also a point of $Y^t_{\alpha-1}$, whose limit in $\Coll(V,W)$ is a general point of $Y_{(\alpha,0)}$. Thus, we also find the component $Y_{(\alpha,0)}$ in the limit of this degeneration with multiplicity 1.
\end{proof}

Geometrically, in the degeneration of $Y_{\alpha-1}$, the component $Y_{\alpha}$ represents the maps $f:\bP^1\to\bP^r$ that remain non-degenerate as $x_\alpha$ moves into $H$, and the component $Y_{(\alpha,0)}$ represents those that become degenerate. The degenerate maps are contained in $H$ and have base-points at $p_{\alpha+1},\ldots,p_n$ (captured by the data of $\phi_0$), but also come with the additional data of a secant vanishing at $p_1+\cdots+p_{\alpha-1}$ (captured by the data of a non-zero element of $\im(\phi)\cap M_{[1,\alpha-1]}$). 

\begin{cor}
There exists a degeneration of the generic $T$-orbit closure $Z$ into the union of components containing $Z_{(r+1,0)},\ldots,Z_{(n-1,0)}$, each with multiplicity 1.
\end{cor}

Note that we have not yet claimed that no other components appear.

\subsection{Iterating the degeneration}
 
We repeat the procedure of the previous section, moving  $L_1,\ldots,L_{\alpha-1}$ one at a time to general hyperplanes containing a plane $\Lambda_2\supset \Lambda_1$, and extracting components parametrizing $\phi$ of type $(r-1,1,1)$. We continue until reaching totally degenerate collineations. We describe only these objects obtained in the end.

Let now $\alpha=(\alpha_0,\cdots,\alpha_{r-1})$ denote a tuple of integers with $n>\alpha_0>\cdots>\alpha_{r-1}>1$. Write also $\alpha_{-1}=n+1$ and $\alpha_r=0$. 

We define a collection of matrices $A_\alpha$ as follows. We index the columns by the integers $0,1,\ldots,r$ and the rows by the integers $1,2,\ldots,n$. For $j=0,\ldots,r-1$, place generic entries, denoted $*$, in row $\alpha_j$ and in columns $j$ and $j+1$. Consider now all other rows: in the $k$-th row, where $k\in(\alpha_{j},\alpha_{j-1})$ for some $j\in[0,r]$, place a generic entry in column $j$. Finally, make all other entries of $A_\alpha$ zero. The case $r=3$, $n=12$, $m=3$, and $\alpha=(9,7,4)$ is shown below.
\begin{equation*}
A_\alpha=
\begin{bmatrix}
0 & 0 & 0 & *\\
0 & 0 & 0 & *\\
0 & 0 & 0 & *\\
0 & 0 & * & *\\
0 & 0 & * & 0\\
0 & 0 & * & 0\\
0 & * & * & 0\\
0 & * & 0 & 0\\
* & * & 0 & 0\\
* & 0 & 0 & 0\\
* & 0 & 0 & 0\\
* & 0 & 0 & 0\\
\end{bmatrix}
\end{equation*}

\begin{defn}
Fix $\alpha$ as above. Define $Y_\alpha\subset\Coll_{1^{r+1}}(V,W)$ by the locus of totally degenerate collineations whose general point $\phi=\{\phi_j\}_{j=0}^{r}$ has the property that $\im(\phi_j)$ is equal to the span of the last $j+1$ columns of a matrix of the form $A_\alpha$.

Define $Z_\alpha$ to be equal to $\pi(Y_\alpha)$, or equivalently, to be the closure of the locus given of subspaces given by the column span of a matrix of the form $A_\alpha$.
\end{defn}

Each $Y_\alpha,Z_\alpha$ has the expected dimension of $n-1$, and $Z_\alpha$ is additionally the $T$-orbit closure of a generic matrix $A_\alpha$. Iterating the degeneration of the previous section, we obtain:

\begin{prop}
There exists a degeneration of $Y$ into a union of components including $Y_\alpha$, each with multiplicity 1, and ranging over all tuples $\alpha$ defined above.

Furthermore, the resulting degeneration of $Z=\pi(Y)$ contains \emph{only} the components $Z_\alpha$.
\end{prop}

Checking no other components appear in the degeneration of $Z$ requires more work. This is a consequence of the fact that the associated matroid polytope (moment map image) of the toric variety $Z$ subdivides into a union of the matroid polytopes of the components $Z_\alpha$. See \cite[\S 7]{l_flag} for details. 

Now, $Z_\alpha$ is visibly an intersection of two Schubert varieties (that is, a Richardson variety). More precisely, let $F_M$ denote the flag
\begin{equation*}
0\subset M_{[1,n-1]}\subset \cdots\subset M_{[1,2]} \subset M_1\subset W,
\end{equation*}
and let $F'_M$ denote the transverse flag
\begin{equation*}
0\subset M_{[2,n]}\subset \cdots\subset M_{[n-1,n]}\subset M_n\subset W.
\end{equation*}
Let $\mu=(\mu_0,\ldots,\mu_{r-1},0)$ be the partition given by $\mu_j=(\alpha_j-1)-(r-j)$ for $j=0,1,\ldots,r-1$. Then, we have
\begin{equation*}
Z_\alpha=\Sigma^{F_M}_{\mu}\cap \Sigma^{F'_M}_{\overline{\mu}}.
\end{equation*}

In the language of maps to $\bP^r$, the first Schubert variety simultaneously packages all of the secancy conditions imposed on $f:\bP^1\to\bP^r$ upon iterated degeneration as a complete collineation, and the second packages all of the base-point conditions. We now put everything together:

\begin{proof}[Proof of Theorem \ref{orbit_formula}]
By the above discussion, we have
\begin{equation*}
\Gamma_{r,n,n-1}=[Z]=\sum_\alpha[Z_\alpha]=\sum_{\mu\subset(n-r-2)^r}\sigma_\mu\sigma_{\overline{\mu}}.
\end{equation*}
\end{proof}

We remark that the entire proof of Theorem \ref{orbit_formula} goes through without change $T$-equivariantly. Thus, one recovers in fact Berget-Fink's formula \cite[Theorem 5.1]{bf} in the equivariant cohomology of $\Gr(r+1,W)$, using instead the equivariant classes of the Schubert varieties $\Sigma^{F_M}_{\mu},\Sigma^{F'_M}_{\overline{\mu}}$.

\subsection{The coefficients $\Gamma^\lambda_{r,n,n-1}$}\label{sec:tev_coeffs}

The integers $\Gamma^\lambda_{r,n,n-1}$, and hence also $\Gamma^\lambda_{r,n,d}$, are understood. Klyachko \cite[Theorem 6]{k} proved that
\begin{equation}\label{klyachko_formula}
\Gamma^\lambda_{r,n,n-1}=\sum_{j=0}^{m(\lambda)}(-1)^j\binom{n}{j}|\SSYT_{r+1-j}(\lambda^j)|,
\end{equation}
where:
\begin{itemize}
\item for any partition $\lambda\subset(n-r-1)^{r+1}$ with $|\lambda|=r(n-r-2)$, we define $m(\lambda)\ge0$ to be the unique integer for which $\lambda_0=\cdots=\lambda_{m(\lambda)-1}=n-r-1$ and $\lambda_{m(\lambda)}<n-r-1$, and
\item for $j=0,1,\ldots,m(\lambda)-1$, we define $\lambda^j=(\lambda_j,\ldots,\lambda_r)$ to be the partition obtained by removing the first $j$ parts of $\lambda$ (which are all equal to $n-r-1$).
\end{itemize}

Proposition \ref{bound_schubert} shows that $\Gamma^\lambda_{r,n,n-1}$ should be the cardinality of a subset of $\SSYT_{r+1}(\lambda)$; note that $|\SSYT_{r+1}(\lambda)|$ is the first term in the alternating sum above. We indeed have:

\begin{thm}\cite[Theorem 4]{l_torus}\label{thm:stripless}
The coefficient $\Gamma^\lambda_{r,n,n-1}$ is equal to the cardinality of the subset of $\SSYT_{r+1}(\lambda)$ consisting of SSYTs with no $(i,i+1)$-strip of length $n-r-1$ for any $i=1,\ldots,r$, see Definition \ref{(i,i+1)-strip}.
\end{thm}

In particular, we have $\Gamma^\lambda_{r,n,n-1}=|\SSYT_{r+1}(\lambda)|$ unless $\lambda_0=n-r-1$, which is on the one hand clear from \eqref{klyachko_formula}, and on the other follows from combining Proposition \ref{increase_degree} and part (a) of Corollary \ref{asymptotic_formula}.

Combining Theorem \ref{thm:stripless} with Theorem \ref{thm:tevPr} gives the following combinatorial interpretation of $\Tev^{\bP^r}_{g,n,d}$. See also \cite[\S 4]{fl}, but we find it convenient to modify the objects slightly, in addition to adding the final two conditions below to handle geometric Tevelev degrees for all curve classes. Namely, the count $\Tev^{\bP^r}_{g,n,d}$ is equal to the number of fillings of the boxes of a $(r+1)\times(d-r)$ grid with:
\begin{itemize}
\item $rg$ red integers among \textcolor{red}{$1,2,\ldots,g$}, with each appearing exactly $r$ times
\item $r(n-r-2)$ blue integers among \textcolor{blue}{$1,2,\ldots,r+1$}, with each appearing any number of times,
\end{itemize}
subject to the following conditions:
\begin{itemize}
    \item the blue integers are top- and left- justified, i.e., they appear above the red integers in the same column and to the left of red integers in the same row,
       
    \item the red integers are strictly decreasing across rows and weakly decreasing down columns (that is, they form an SSYT after rotation by 180 degrees and conjugation),
   \item the blue integers are weakly increasing across rows and strictly increasing down columns (that is, they form an SSYT),
   \item the blue integers only appear in the leftmost $n-r-1$ columns of the grid, and
   \item no $(i,i+1)$-strip of length $n-r-1$ appears among the blue integers, for any $i=1,\ldots,r$.
\end{itemize}
An example filling is given in the case $(r,g,n,d)=(3,6,11,12)$ below.
\begin{center}
\begin{tabular}{ |c|c|c|c|c|c|c|c|c| }
\hline
\textcolor{blue}{1} & \textcolor{blue}{1} & \textcolor{blue}{1}& \textcolor{blue}{1} & \textcolor{blue}{3} & \textcolor{blue}{4} & \textcolor{blue}{4} &\textcolor{red}{4} & \textcolor{red}{2} \\
\hline
\textcolor{blue}{2} & \textcolor{blue}{2} & \textcolor{blue}{2} & \textcolor{blue}{2}&  \textcolor{red}{6} &  \textcolor{red}{5}  & \textcolor{red}{4}  &  \textcolor{red}{3}  & \textcolor{red}{1} \\
\hline
\textcolor{blue}{3} & \textcolor{blue}{3} & \textcolor{blue}{3}& \textcolor{blue}{4} & \textcolor{red}{6} &  \textcolor{red}{5}  & \textcolor{red}{3}  &  \textcolor{red}{2}  & \textcolor{red}{1} \\
\hline
\textcolor{blue}{4} & \textcolor{blue}{4} & \textcolor{blue}{4}& \textcolor{red}{6} &  \textcolor{red}{5}  & \textcolor{red}{4} & \textcolor{red}{3} &  \textcolor{red}{2}  & \textcolor{red}{1} \\
\hline
\end{tabular}
\end{center}

 \section{Curve counts in $\bP^2$}\label{p2_sec}

 \subsection{Outline}\label{sec:P2outline}
 
In this section, we extend the degeneration method from our torus-orbit calculation to prove Theorem \ref{thm:countP2}. That is, we determine the number of non-degenerate maps from a general pointed curve to $\bP^2$ incident at $n_0$ general points $x_1,\ldots,x_{n_0}$ and $n_1=n-n_0$ general lines $X_{n_0+1},\ldots,X_n$. We summarize the calculation.

\begin{enumerate}
\item[1.] Fix a line $\ell\in\bP^2$. Move the points $x_1,\ldots,x_{n_0}$ onto $\ell$ one at a time. The map $f:C\to\bP^2$ either degenerates to a multiple cover of $\ell$ after $\alpha_0\le n_0$ of the points are contained in $\ell$, or remains non-degenerate after all of the points are on $\ell$. If the former, proceed to step 2; if the latter, proceed to step 3.
\item[2.] The limit of $f$ retains, in addition to to the data of a multiple cover of $\ell$, the data of a non-zero section vanishing on $p_1+\cdots+p_{\alpha_0-1}$. Continue the degeneration by moving the points $x_1,\ldots,x_{\alpha_0-1}$ onto a fixed point $x\in\ell$. Either $f$ degenerates to a constant map to $x$ after $\alpha_1\le\alpha_0$ of the points are equal to $x$, or, a new phenomenon occurs: $f$ remains a multiple cover of $\ell$, but the section vanishing on $p_1+\cdots+p_{\alpha_1}$ becomes equal to that vanishing on $p_1+\cdots+p_{\alpha_0-1}$, which was previously not seen by the multiple cover $f$. In order for this to be possible, $f$ acquires base-points at $p_{\alpha_1+1},\ldots,p_{\alpha_0}$, in addition to those at $p_{n_0+1},\ldots,p_n$. If the former possibility, then go to step (a) below, and if the latter, go to step (b).
\begin{enumerate}
\item[(a)] $f$ is now a complete collineation of type $(1,1,1)$, and such objects can be enumerated (more precisely, pushed forward to $\Gr(3,W)$) via Schubert calculus as in the orbit closure calculation. 
\item[(b)] Move the points $X_{n_0+1}\cap \ell,\ldots,X_{n}\cap \ell$ onto $x$, until $f$ degenerates into a constant map. We again get collineations of type $(1,1,1)$ which can be enumerated.
\end{enumerate}
\item[3.] Fix a point $z\notin \ell$, and move the lines $X_{n_0+1},\ldots,X_{n}$ to contain $z$. Eventually, say, after $X_{n_0+1},\ldots,X_{\alpha_0}\ni z$, $f$ will degenerate to a constant map with image $z$. However, the limit of $f$ as a complete collineation will retain the information of a map $\widetilde{f}:C\to\bP^1$, which may be regarded as the limit of the maps obtained from $f$ by post-composition with projection from $z$.
\item[4.] Move the points $x_1,\ldots,x_{n_0},X_{n_0+1}\cap \ell,\ldots,X_{\alpha_0-1}\cap \ell$ onto $x\in\ell$ until the map $\widetilde{f}$ degenerates. From here, we will again be able to enumerate the degenerate objects. The situation is different depending on whether $\widetilde{f}$ degenerates before or after the $n_0$-th step, but we defer a discussion to later.
\end{enumerate}
 
 \subsection{Length 1 components}
 
 We now carry out the program described in \S\ref{sec:P2outline} on the space of complete collineations. Let $L_1,\ldots,L_{n_0}\subset V\cong\bC^3$ be general planes and $L_{n_0+1},\ldots,L_n\subset V$ be general lines. We are interested in the image of the locus
\begin{equation*}
Y:=\bigcap_{i=1}^{n}\Inc(L_i,M_i)\subset \Coll(V,W)
\end{equation*}
to $\Gr(3,W)$, which we denote by $Z$, and the class $[Z]\in H^{2(n+n_0)}(\Gr(3,W))$. We study this class by degenerating the $L_i$ and studying the corresponding degeneration of $Z$.

We first carry out the analogue of the degeneration in \S\ref{sec_first_step}. Fix a line $\Lambda_1\subset V$. For a fixed $\alpha_0\le n_0$, now assume that $L_1,\ldots,L_{\alpha_0}\supset \Lambda_1$; all $L_i$ are otherwise assumed to be general.  

\begin{defn}
For $\alpha_0\in [0,\min(n-1,n_0)]$, define $Y^{(3)}_{\alpha_0}$ by the closure of the intersection
\begin{equation*}
\left(\bigcap_{i=1}^{n} \Inc(L_i,M_i)\right)\cap\Coll^\circ_{(3)}(V,W)
\end{equation*}
in $\Coll(V,W)$. Define $Z^{(3)}_{\alpha_0}=\pi(Y^{(3)}_{\alpha_0})\subset\Gr(3,W)$.
\end{defn}

In particular, we have $Y=Y^{(3)}_{0}$ and $Z=Z^{(3)}_{0}$.

\begin{defn}
For $\alpha_0\in [3,\min(n-1,n_0)]$, define $Y^{(2,1)}_{\alpha_0}\subset\Coll_{(2,1)}(V,W)$ to be the closure of the locus of complete collineations $\phi:V\to W$ satisfying the following properties.
\begin{enumerate}
\item $\phi$ has type $(2,1)$, and $V_1=\ker(\phi_0)=\Lambda_1$,
\item the map $\phi_0:V/\Lambda_1\to W$ satisfies the incidence conditions $\Inc(L_i,M_i)$ for $i=1,2,\ldots,\alpha_0$, and the incidence conditions $\Inc(\langle L_i,\Lambda_1\rangle,M_i)$ for $i=n_0+1,\ldots,n$.
\item the image $\im(\phi_0)$ is contained in $M_{[\alpha_0+1,n_0]}$, but none of $M_1,\ldots,M_{\alpha_0}$ or $M_{n_0+1},\ldots,M_n$,
\item $\im(\phi_0)\cap M_{[1,\alpha_0-1]}=0$, but $\im(\phi)\cap M_{[1,\alpha_0-1]}\neq 0$.
\end{enumerate}
\end{defn}

Both $Y^{(3)}_{\alpha_0}$ and $Y^{(2,1)}_{\alpha_0}$ are checked, by building the usual incidence correspondences, to be irreducible and generically smooth of the expected codimension $n+n_0$ in $\Coll(V,W)$. We now consider a degeneration over $\bD$ in which, over the general point $t\in\bD$, only $L_1,\ldots,L_{\alpha_0-1}$ contain $\Lambda_1$, and over $0\in\bD$, $L_{\alpha_0}\supset \Lambda_1$; all other $L_i$ do not change.

\begin{prop}\label{flat_limit_21_case}
The flat limit $(Y^{(3)}_{\alpha_{0}-1})_0$ of the subscheme $Y^{(3)}_{\alpha_{0}-1}\subset\Coll(V,W)$, under the degeneration described above, contains the components $Y^{(3)}_{\alpha_0}$ (if $\alpha_0\neq n$) and $Y^{(2,1)}_{\alpha_0}$ (if $\alpha_0\ge 3$). Both components appear with multiplicity 1, and there no other components.
\end{prop}

\begin{proof}
The following matrices over $\bD$ exhibit a general point of either component as a limit of full rank collineations on the general fiber:
\begin{equation*}
A^t_{(3)}:=
\begin{bmatrix}
0 & a_{1,1}  & a_{1,2} \\
\vdots & \vdots & \vdots \\
0 & a_{\alpha_0-1,1} & a_{\alpha_0-1,2} \\
t & a_{\alpha_0,1}  & a_{\alpha_0,2} \\
a_{\alpha_0+1,0} & a_{\alpha_0+1,1}  & a_{\alpha_0+1,2} \\
\vdots & \vdots & \vdots \\
a_{n_0,0} & a_{n_0,1} & a_{n_0,2} \\
a_{n_0+1,0} & a_{n_0+1,1} & a_{n_0+1,2} \\
\vdots & \vdots & \vdots \\
a_{n,0} & a_{n,1} & a_{n,2} 
\end{bmatrix}
, A^t_{(2,1)}:=
\begin{bmatrix}
0 & a_{1,1}  & a_{1,2} \\
\vdots & \vdots & \vdots \\
0 & a_{\alpha_0-1,1} & a_{\alpha_0-1,2} \\
t & a_{\alpha_0,1}  & a_{\alpha_0,2} \\
ta_{\alpha_0+1,0} & ta_{\alpha_0+1,1}  & ta_{\alpha_0+1,2} \\
\vdots & \vdots & \vdots \\
ta_{n_0,0} & ta_{n_0,1} & ta_{n_0,2} \\
ta_{n_0+1,0}  & a_{n_0+1,1}+ta'_{n_0+1,1} & a_{n_0+1,2}+ta'_{n_0+1,1} \\
\vdots & \vdots & \vdots \\
ta_{n,0}  & a_{n,1}+ta'_{n,1} & a_{n,2}+ta'_{n,2} 
\end{bmatrix}
\end{equation*}
The first $n_0$ row vectors of each matrix are constrained to be scalar multiples of the given ones. For $i>n_0+1$, the $i$-th row is assumed implicitly to satisfy a single linear relation corresponding to the condition $\Inc(L_i,M_i)$. Taking $t=0$ in $A^t_{(3)}$ gives a general point of $Y^{(3)}_{\alpha_0}$. The limit in $\Coll(V,W)$ of $A^t_{(2,1)}$ as $t\to 0$ is given by the data of the map $\phi_0:V/\Lambda_1\to W$ obtained from setting $t=0$ in the right-most two columns, and the additional section of $\im(\phi)$ given by 
dividing the first column by $t$. In this way, we obtain in the limit a general point of $Y^{(2,1)}_{\alpha_0}$. The fact that both components appear with multiplicity 1 is reflected in the fact that both matrices have full rank over $k[t]/t^2$. 

It remains to check that there are no other components in the limit. The details are straightforward, but somewhat cumbersome, so we do not give them all.

First, we rule out components $Y'\subset (Y^{(3)}_{\alpha_{0}-1})_0$ for which $V_\ell\neq \Lambda_1$, where $\ell\in\{1,2\}$ is the length of a collineation corresponding to a general point. Let $\overrightarrow{r}$ be the type of such a collineation. Then, one checks that the closed conditions $\Inc(L_i,M_i)$ (taking $L_\alpha^0\supset \Lambda_1$ when $i=\alpha$) impose the expected number of conditions on $\Coll_{\overrightarrow{r}}(V,W)$, and thus, too many on $\Coll(V,W)$. This in particular rules out all components whose generic $\phi$ has type $(1,2)$.

Next, when a generic $\phi\in Y'$ has type $(2,1)$ with $V_1=\Lambda_1$, then we must at least have the condition $\im(\phi)\cap M_{[1,\alpha_0-1]}\neq0$, as this condition is satisfied on the generic fiber and is closed. One checks that requiring additionally that $\im(\phi_0)\cap M_{[1,\alpha_0-1]}\neq0$ imposes too many conditions when combined with the conditions $\Inc(L_i,M_i)$. (In particular, at least $\alpha_0-2$ of the hyperplanes $M_1,\ldots,M_{\alpha_0-1}$ must contain $\im(\phi_0)$.) Similarly, if we assume property (iv) in the definition of $Y^{(2,1)}_{\alpha_0}$, then requiring additionally that $\im(\phi_0)$ be contained in any of $M_1,\ldots,M_{\alpha_0},M_{n_0+1},\ldots,M_n$ imposes too many conditions. We therefore get no additional components $Y'$ in this way.

It is left to consider $\phi$ of type $(1,1,1)$ and with $V_2=\Lambda_1$. We first observe that if $V_1$ is not equal to one of $L_1,\ldots,L_{\alpha_0}$ or $\langle\Lambda_1,L_{n_0+1}\rangle,\ldots,\langle\Lambda_1,L_{n}\rangle$, then $\im(\phi_0)$ will need to be contained in $M_{[1,n]}=0$ owing to the conditions $\Inc(L_i,M_i)$, a contradiction. If $V_1$ is equal to one of $L_1,\ldots,L_{\alpha_0-1}$ -- without loss of generality, we may take $V_1=L_1$ -- then we will have $\im(\phi_0)\subset M_{[2,n]}$. Now, combining with the requirements that $\im(\phi_1)\subset M_{[\alpha_0+1,n_0]}$ (coming from $\Inc(L_i,M_i)$ for $i=\alpha_0+1,\ldots,n_0$) and $\im(\phi)\cap M_{[1,\alpha_0-1]}\neq0$ imposes too many conditions on $\phi$.

Finally, suppose instead that $V_1=L_{\alpha_0}$ (where we take the plane $L_{\alpha_0}\supset\Lambda_1$ in its special position); the case $V_1=\langle L_i,\Lambda_1\rangle$ with $i>n_0$ can be handled similarly. We may change basis on $V$ in such a way that a degeneration to $\phi$ takes the form
\begin{equation*}
\phi^t=
\begin{bmatrix}
0 & ta_{1,1}  & ta_{1,2} \\
\vdots & \vdots & \vdots \\
0 & ta_{\alpha_0-1,1} & ta_{\alpha_0-1,2} \\
t & 0  & 1 \\
ta_{\alpha_0+1,0} & ta_{\alpha_0+1,1}  & ta_{\alpha_0+1,2} \\
\vdots & \vdots & \vdots \\
ta_{n_0,0} & ta_{n_0,1} & ta_{n_0,2} \\
ta_{n_0+1,0}  & ta_{n_0+1,1} & ta_{n_0+1,2} \\
\vdots & \vdots & \vdots \\
ta_{n,0}  & ta_{n,1} & ta_{n,2}
\end{bmatrix}
\end{equation*}
to first order. Indeed, possibly after a base change, the limit of $\phi$ as a linear map is assumed to be zero upon restriction to the first two columns, representing $V_1=L_{\alpha_0}=\langle v_0,v_1\rangle$.

The $a_{i,j}$, with $i>n_0$, are allowed to be zero as long as $\phi^t$ generically has rank 3. We must have $V_1=\langle v_0,v_1\rangle=L_{\alpha_0}$ and $V_2=\langle v_0\rangle=\Lambda_1$. Because $\phi_1(V_2)=0$, the column vectors $\phi^t(v_0)$ and $\phi^t(v_2)$ must be linearly dependent to first order, see \S\ref{sec_limits}. By computing the $2\times 2$ minor obtained from rows $\alpha_0$ and $i=\alpha_0+1,\ldots,n$, we find that, in fact, $a_{\alpha_0+1,0}=\cdots=a_{n,0}=0$.

Next, assume for the moment that $a_{1,1},\ldots,a_{\alpha_0-1,1},a_{\alpha_0+1,1},\ldots,a_{n,1}$ are not all zero. Then, in the limit of $\phi=\lim_{t\to 0}\phi^t$ as a complete collineation, we have
\begin{equation*}
\im(\phi_0)=\text{span}
\begin{bmatrix}
0\\
 \vdots \\
0 \\
1 \\
0 \\
\vdots \\
0
\end{bmatrix}
,
\im(\phi_1)=\text{span}
\begin{bmatrix}
0 & a_{1,1}   \\
\vdots & \vdots \\
0 & a_{\alpha_0-1,1}  \\
1 & 0  \\
0& a_{\alpha_0+1,1}   \\
\vdots & \vdots \\
0 & a_{n,1} 
\end{bmatrix}
.
\end{equation*}
If at least $\alpha_0-2$ of $a_{1,1},\ldots,a_{\alpha_0-1,1}$ are equal to zero, say,  $a_{1,1}=\cdots=,a_{\alpha_0-2,1}=0$, then we obtain the condition that $\im(\phi_1)$ is contained in $M_{[1,\alpha_0-2]}$. The conditions $\Inc(L_i,M_i)$ for $i=\alpha_0+1,\ldots,n_0$  force further that $\im(\phi_1)\subset M_{[\alpha_0+1,n_0]}$. Combining with the condition $\im(\phi_0)\subset M_{[1,\alpha_0-1]}\cap M_{[\alpha_0+1,n]}$, we obtain too many conditions on $\phi$. 

Thus, we may conclude that
\begin{equation*}
\im(\phi)=
\begin{bmatrix}
0 & a_{1,1}  & a_{1,2} \\
\vdots & \vdots & \vdots \\
0 & a_{\alpha_0-1,1} & a_{\alpha_0-1,2} \\
1 & 0  & 0 \\
0 & a_{\alpha_0+1,1}  & a'_{\alpha_0+1,2} \\
\vdots & \vdots & \vdots \\
0  & a_{n,1} & a'_{n,2}
\end{bmatrix}
\end{equation*}
where the entries $a'_{i,2}$ depend on the (undepicted) second derivatives of the $(i,0)$-entry of $\phi^t$. In particular, the $2\times(\alpha_0-1)$ matrix appearing in the top right has full rank, as at least two of its rows are non-zero, and all of its rows are constrained to be scalar multiples of the given ones.

We therefore read off the following condition on the limit collineation $\phi$: there exists a rank 2 map $\phi^\perp:\langle v_1,v_2\rangle\to W$ as above satisfying the conditions $\Inc(L_i\cap\langle v_1,v_2\rangle,M_i)$ for $i=1,\ldots,\alpha_0-1$, and furthermore, with $\phi^\perp(v_1)\subset\im(\phi_1)$ and $\im(\phi^\perp)\subset\im(\phi)$. A parameter count now shows that the space of $\phi$ satisfying both this condition and in addition the condition $\im(\phi_0)\subset M_{[1,\alpha_0-1]}\cap M_{[\alpha_0+1,n]}$ has too small a dimension.

Finally, if $a_{1,1},\ldots,a_{\alpha_0-1,1},a_{\alpha_0+1,1},\ldots,a_{n,1}$ are instead all zero, then the columns $\phi^t(v_1),\phi^t(v_2)$ are also dependent to first order. We may then repeat the argument by passing to second order in the first two columns, and iterate.
\end{proof}

\begin{cor}
We have
\begin{equation*}
[Y]=[Y^{(3)}_{n_0}]+\sum_{\alpha_0=3}^{n_0}[Y^{(2,1)}_{\alpha_0}]
\end{equation*}
as cycles on $\Coll(V,W)$. (If $n_0=n$, we take the first term to be zero.)
\end{cor}

We next study the component $Y^{(3)}_{n_0}$ under further degeneration. Fix general planes $L_1,\ldots,L_{n_0}$ containing $\Lambda_1$. Fix now a \emph{general} plane $\Lambda_2$ (not containing $\Lambda_1$), and, for some integer $\alpha_0>n_0$, suppose that $L_{n_0+1},\ldots,L_{\alpha_0}\subset \Lambda_2$ are general lines, and that $L_{\alpha_0+1},\ldots,L_n$ are further general lines with no additional constraints.

As before, denote by $Y^{(3)}_{\alpha_0}$ the closure in $\Coll(V,W)$ of the locus of \emph{full rank} collineations $\phi:V\to W$ satisfying each of the conditions $\Inc(L_i,M_i)$. 

We now introduce additional parameter spaces.

\begin{defn}
Suppose that $\alpha_0\ge n_0+3$. Define the closed subvariety 
\begin{equation*}
\wt{Y}_{\alpha_0}\subset \bP W\times\Coll(\Lambda_2,W)\times\Gr(3,W)
\end{equation*}
to be the set of points $(w_0,\wt{\phi_1},U)$ satisfying:
\begin{enumerate}
\item $w_0\in M_{[1,n_0]}\cap M_{[\alpha_0+1,n]}$,
\item $\wt{\phi_1}$ satisfies the incidence conditions $\Inc(L_i\cap \Lambda_2,M_i)$ for $i=1,2,\ldots,n_0$ and $\Inc(L_i,M_i)$ for $i=n_0+1,\ldots,\alpha_0-1$.
\item $U\supset \langle w_0,\im(\wt{\phi_1})\rangle$.
\end{enumerate}
\end{defn}

It is routine to check that $\wt{Y}_{\alpha_0}$ is irreducible and generically smooth of the expected dimension 
\begin{equation*}
(\alpha_0-n_0-1)+[(2n-1)-(\alpha_0-1)]=(3n-1)-(n+n_0),
\end{equation*}
and that a general point satisfies all of the needed conditions ``generically.'' That is, the line $w_0$ lies in no other $M_i$, the collineation $\wt{\phi_1}$ is of type $(2)$ and has image contained in no $M_i$, and that $w_0\notin \im(\wt{\phi_1})$, so in fact $U=\langle w_0,\im(\wt{\phi_1})\rangle$. Define $\wt{\pi}:\wt{Y}_{\alpha_0}\to\Gr(3,W)$ by projection to the last factor.

There is a \emph{rational} map $\psi:\wt{Y}_{\alpha_0}\dashrightarrow\Coll_{(1,2)}(V,W)$ sending $(w_0,\wt{\phi_1},U)$ to the unique collineation $\phi$ with the properties:
\begin{enumerate}
\item $\phi$ has type $(1,2)$, and $V_1=\ker(\phi_0)=\Lambda_2$,
\item $\im(\phi_0)=w_0$,
\item $\phi_1$ is given by composing $\wt{\phi_1}$ with the quotient map $W\to W/w_0$.
\end{enumerate}
(iii) makes sense if $\wt{\phi_1}$ has type $(2)$ and $\langle w_0\rangle\notin \im(\wt{\phi_1})$. However, the definition of $\psi$ can be extended further.
\begin{itemize}
\item If $\wt{\phi_1}$ has type $(2)$ and $\langle w_0\rangle\in \im(\wt{\phi_1})$, then $\phi_1$ may be replaced by the collineation of type $(1,1)$ for which the rank 1 map $(\phi_1)_0$ is given by post-composition with the quotient $W\to W/w_0$, and $\im(\phi_1)$ is determined by $U$.
\item If instead $\wt{\phi_1}$ has type $(1,1)$ and $\langle w_0\rangle\notin \im(\wt{\phi_1})$, then one can make sense of the quotient of $\wt{\phi_1}$ by $w_0$, again making $\phi_1$ of type $(1,1)$.
\item More generally, if $\wt{\phi_1}$ has type $(1,1)$ and $\im((\wt{\phi_1})_{0})\neq w_0$, then one can define $(\phi_{1})_0$ by the quotient of $\langle(\wt{\phi_1})_{0},w_0\rangle$ by $w_0$, and $\im(\phi_1)$ by $U$.
\end{itemize}
Finally, if $\wt{\phi_1}$ has type $(1,1)$ and $\im((\wt{\phi_1})_{0})=w_0$, then $\psi$ is indeterminate, but this will not be a problem for us.
\begin{defn}
Define $Y^{(1,2)}_{\alpha_0}$ to be the closure in $\Coll(V,W)$ of the image of $\wt{Y}_{\alpha_0}$ under $\psi$.
\end{defn}

The map $\psi$ is birational onto its image. Indeed, if $(w_0,\wt{\phi_1},\langle w_0,\wt{\phi_1}\rangle)$ is a general point of $\wt{Y}_{\alpha_0}$, at which, in particular, $w_0$ and $\im(\wt{\phi_1})$ are not contained in either $M_{n_0+1}$ or $M_{n_0+2}$, then replacing $\wt{\phi_1}$ with a different lift of $\phi_1=\wt{\phi_1}/w_0$ will result in a collineation $\wt{\phi_1}'$ no longer satisfying the incidence conditions $\Inc(L_i,M_i)$. (Note here that we use $\alpha_0\ge n_0+3$.) In particular, $Y^{(1,2)}_{\alpha_0}$ is irreducible and generically smooth of codimension $n+n_0$ in $\Coll(V,W)$.

Take now a degeneration over $\bD$ in which, initially, $L_1,\ldots,L_{n_0}\supset \Lambda_1$, and $L_{n_0+1},\ldots,L_{\alpha_0-1}\subset\Lambda_2$, and then $L_{\alpha_0}$ is moved into $\Lambda_2$.

\begin{prop}\label{flat_limit_12_case}
The flat limit of the subscheme $Y^{(3)}_{\alpha_{0}-1}\subset\Coll(V,W)$ under the degeneration described above contains the components $Y^{(3)}_{\alpha_0}$ and $Y^{(1,2)}_{\alpha_0}$ (if $\alpha\ge n_0+3$). Both components appear with multiplicity 1, and there are no other components.
\end{prop}

\begin{proof}
The following matrices over $\bD$ exhibit a general point of either component as a limit of full rank collineations on the general fiber:
\begin{equation*}
\begin{bmatrix}
0 & a_{1,1}  & a_{1,2} \\
\vdots & \vdots & \vdots \\
0 & a_{n_0,1} & a_{n_0,2} \\
a_{n_0+1,0} & a_{n_0+1,1}  & a_{n_0+1,2} \\
\vdots & \vdots & \vdots \\
a_{\alpha_0-1,0} & a_{\alpha_0-1,1}  & a_{\alpha_0-1,2} \\
a_{\alpha_0,0} & a_{\alpha_0,1}+ta'_{\alpha_0,1}   &  a_{\alpha_0,2}+ta'_{\alpha_0,2}  \\
a_{\alpha_0+1,0} & a_{\alpha_0+1,1}  & a_{\alpha_0+1,2} \\
\vdots & \vdots & \vdots \\
a_{n,0} & a_{n,1} & a_{n,2} 
\end{bmatrix}
,
\begin{bmatrix}
0 & ta_{1,1}  & ta_{1,2} \\
\vdots & \vdots & \vdots \\
0 & ta_{n_0,1} & ta_{n_0,2} \\
a_{n_0+1,0} & ta_{n_0+1,1}  & ta_{n_0+1,2} \\
\vdots & \vdots & \vdots \\
a_{\alpha_0-1,0} & ta_{\alpha_0-1,1}  & ta_{\alpha_0-1,2} \\
a_{\alpha_0,0} & ta_{\alpha_0,1}   & ta_{\alpha_0,2}  \\
ta_{\alpha_0+1,0} & ta_{\alpha_0+1,1}  & ta_{\alpha_0+1,2} \\
\vdots & \vdots & \vdots \\
ta_{n,0} & ta_{n,1} & ta_{n,2} 
\end{bmatrix}
\end{equation*}
The first $n_0$ rows each satisfy a linear relation in the rightmost two columns in addition to the vanishing in the first column. Rows $n_0+1$ through $\alpha_0-1$ also satisfy a linear relation in the rightmost two columns. The last $n-\alpha_0$ rows each satisfy a single linear relation involving all three columns. Finally, row $\alpha_0$ satisfies a linear relation of the form $tv_0+\gamma_1 v_1+\gamma_2 v_2=0$.

The matrix on the left has full rank upon substituting $t=0$, which determines the limit of the full rank maps $\phi^t$ in $\Coll(V,W)$. The corresponding component $Y^{(3)}_{\alpha_0}$ appears with multiplicity 1 because it is cut out by linear equations on $\bP\Hom(V,W)^\circ$. On the right, the rightmost two columns become zero upon substituting $t=0$; substituting $t=0$ in the left-most column gives the vector $w_0=\im(\phi_0)$, and dividing the other two columns by $t$ gives the lift $\wt{\phi_1}:\Lambda_2\to W$. Here, the multiplicity 1 statement amounts to the fact that the 1-parameter family of \emph{matrices} defined by $t$ is transverse to the locus of rank 1 matrices (embedded in $\bP\Hom(V,W)$ by a Segre embedding) when $t=0$, hence the same is true in $\Coll(V,W)$ after blowing up this locus.

Finally, one needs to argue that there are no further components in the limit, by following the strategy of Proposition \ref{flat_limit_21_case}. Note in particular that one can rule out components where $V_\ell=\Lambda_1$ by the same calculations appearing there, except now that $\im(\phi)$ is now constrained to intersect $M_{[1,\alpha_0]}$ non-trivially, not just $M_{[1,\alpha_0-1]}$. Thus, one finds that the only components that can appear at the boundary of $\Coll(V,W)$ must have $V_1=\Lambda'_2$, and that all candidates other than $Y^{(1,2)}_{\alpha_0}$ will be over-constrained. The details are omitted.
\end{proof}

\begin{cor}
We have
\begin{equation*}
[Y]=[Y^{(3)}_{n}]+\sum_{\alpha_0=3}^{n_0}[Y^{(2,1)}_{\alpha_0}]+\sum_{\alpha_0=n_0+3}^{n}[Y^{(1,2)}_{\alpha_0}]
\end{equation*}
as cycles on $\Coll(V,W)$.
\end{cor}

We will study the components $Y^{(2,1)}_{\alpha_0}$ and $Y^{(1,2)}_{\alpha_0}$ in the next two sections. However, we now observe that we can, for our purposes, ignore the full rank component.

\begin{prop}
The class $[Y^{(3)}_{n}]$ pushes forward to 0 under $\pi$.
\end{prop}

\begin{proof}
It suffices to show that the restriction of $\pi$ to $Y^{(3)}_{n}$ has positive-dimensional fibers. Let $\phi:V\to W$ be a general point of $Y^{(3)}_{n}$. We may choose a basis $v_0,v_1,v_2$ of $V$ for which $\Lambda_1=\langle v_0\rangle$ and $\Lambda_2=\langle v_1,v_2\rangle$. The map $\phi$ is determined the triple $(\phi(v_0),\phi(v_1),\phi(v_2))\in W^3$. Now, multiplying $\phi(v_1),\phi(v_2)$ by the same scalar $\gamma\neq0$ gives a 1-dimensional locus of elements of $\Coll(V,W)$ that on the one hand still lie in $Y^{(3)}_{n}$, and on the other hand map to the same point of $\Gr(3,W)$ under $\pi$.
\end{proof}

 \subsection{Degenerations of $Y^{(2,1)}$}
 
We now study degenerations of $Y^{(2,1)}_{\alpha_0}$, defined in the previous section, in order to compute its integral in $\Gr(3,W)$. Recall that we have $L_1,\ldots,L_{\alpha_0}\supset \Lambda_1$, for some fixed $\alpha_0\in[3,n_0]$.
 
Fix a general plane $\Lambda'_2\supset\Lambda_1$ (we use the notation $\Lambda'_2$ to distinguish from $\Lambda_2$, used in the previous section). We will move $L_1,\ldots,L_{\alpha_0}$ successively to be equal to $\Lambda'_2$. Suppose that $L_1,\ldots,L_{\alpha_1}$ have all been made equal to $\Lambda'_2$. If $0\le\alpha_1\le \alpha_0-2$, we first define $Y^{(2,1)}_{(\alpha_0,\alpha_1)}\subset\Coll(V,W)$ exactly in the same way as $Y^{(2,1)}_{\alpha_0}=:Y^{(2,1)}_{(\alpha_0,0)}$, except that now the planes $L_1,\ldots,L_{\alpha_1}$ are all equal. The subscheme $Y^{(2,1)}_{(\alpha_0,\alpha_1)}$ is irreducible and generically smooth of the expected codimension $n+n_0$.

As soon as $\alpha_1=\alpha_0-1$, the condition that $\im(\phi_{0})\cap M_{[1,_{\alpha_0}-1]}=0$ can no longer be satisfied if $\phi_0$ remains of rank 2, because $\phi_0$ satisfies the conditions $\Inc(L_i,M_i)=\Inc(\Lambda'_2,M_i)$ for $i=1,2,\ldots,\alpha_1-1=\alpha_0-1$. Furthermore, a dimension count shows that requiring that $\im(\phi)$ contain \emph{two} sections in $M_{[1,\alpha_0-1]}$ (one of which is given by $\phi_0(\Lambda'_2)$) imposes too many conditions.

On the other hand, starting from $Y^{(2,1)}_{(\alpha_0,\alpha_0-2)}$, if $L_{\alpha_0-1}$ moves to contain $\Lambda_1$, then the section initially in $(\im(\phi)-\im(\phi_0))\cap M_{[1,\alpha_0-1]}$ may ``jump'' to $\im(\phi_0)$ in the limit, becoming the ``secant'' $\phi_0(\Lambda'_2)$. The general fiber necessarily has two sections in $\im(\phi)\cap M_{[1,\alpha_0-2]}$, so the special fiber must as well. In fact, the same phenomenon may occur upon degeneration of any $Y^{(2,1)}_{(\alpha_0,\alpha_1-1)}$ with $1\le \alpha_1\le\alpha_0-1$, but the conditions $\Inc(L_i,M_i)$ for $i=\alpha_1+1,\ldots,\alpha_0-1$ force $M_{\alpha_1+1},\ldots,M_{\alpha_0-1}$ to be bp-hyperplanes for $\phi_0$.

We therefore make the following definition.
\begin{defn}
If $2\le\alpha_1\le\alpha_0-1$, we define the subscheme $Y^{(2,1)-\sec}_{(\alpha_0,\alpha_1)}\subset\Coll(V,W)$ to be the closure of the locus of collineations $ \phi$ with the following properties.
\begin{enumerate}
\item $\phi$ has type $(2,1)$, and $V_1=\Lambda_{1}$ for $j=1,2,\ldots,m$,
\item $\phi_0:V/\Lambda_1\to W$ satisfies the incidence conditions $\Inc(L_i,M_i)$ for $i=1,\ldots,\alpha_1$ and $\alpha_0$, and the incidence conditions $\Inc(\langle L_i,\Lambda_1\rangle,M_i)$ for $i=n_0+1,\ldots,n$.
\item $\im(\phi_0)$ is contained in $M_{[\alpha_1+1,\alpha_0-1]}$ and $M_{[\alpha_0+1,n_0]}$, but no other $M_i$,
\item $\dim(\im(\phi)\cap M_{[1,\alpha_1-1]})=2$ and $\dim(\im(\phi_0)\cap M_{[1,\alpha_1-1]}=1)$ (in fact, $\dim(\im(\phi_0)\cap M_{[1,\alpha_1]})=1$, with the intersection given by $\phi_0(\Lambda'_2)$).
\end{enumerate}
\end{defn}
In the language of maps to $\bP^2$, such a $\phi$ corresponds to $f:\bP^1\to H\subset\bP^2$ with base-points at $p_{\alpha_1+1},\ldots,p_{\alpha_0-1},p_{\alpha_0+1},\ldots,p_{n_0}$, sending $p_1,\ldots,p_{\alpha_1}$ all to the same point $x$ (so that the divisor $p_1+\cdots+p_{\alpha_1}$ is a multi-secant), and with $f(p_i)=X_i$ for $i=\alpha_0,n_0+1,\ldots,n$. The collineation $\phi$ additionally retains the data of a 2-dimensional subspace of $M_{[1,\alpha_1-1]}$ in its image, which is not seen by the map $f$.

We give the na\"{i}ve dimension count:
\begin{itemize}
\item property (i) imposes $3$ conditions,
\item property (ii) imposes $\alpha_1+1+(n-n_0)$ conditions,
\item property (iii) imposes $2(n_0-\alpha_1-1)$ conditions, and
\item property (iv) imposes $\alpha_1-2$ \emph{additional} conditions.
\end{itemize}
In fact, $Y^{(2,1)-\sec}_{(\alpha_0,\alpha_1)}$ is irreducible and generically smooth of the expected codimension $n+n_0$, proven using the usual method.

These components make sense in the setting of the previous section, when $n_0=n$, and, in fact, do appear in the limit of the component $Y^{(2,1)}_{(\alpha_0,\alpha_1-1)}$ under the same degeneration. However, we will see later (Proposition \ref{21sec-0integral}) that the reason that they play no role there is that they push forward to zero under $\pi$.

We now define the further degenerate loci on $\Coll(V,W)$ whose integrals can be computed directly. In analogy with the loci $Y^{(1,1,1)}$ of the previous section, we have:
\begin{defn}
If $2\le \alpha_1\le \alpha_0-1$, define $Y^{(1,1,1)}_{(\alpha_0,\alpha_1)}$ to be the closure of the locus of collineations $\phi$ satisfying the following properties.
\begin{enumerate}\label{Y111_def}
\item $\phi$ has type $(1,1,1)$, with $V_1=\Lambda'_2$ and $V_2=\Lambda_1$,
\item $\im(\phi_0)\subset M_{[\alpha_1+1,n]}$ and $\im(\phi_1)\subset M_{[\alpha_0+1,n_0]}$, 
\item $\im(\phi_1)\cap M_{[1,\alpha_1-1]}\neq0$ and $\im(\phi)\cap M_{[1,\alpha_0-1]}\neq0$, but $\im(\phi_0)\cap M_{[1,\alpha_1-1]}=0$ and $\im(\phi_1)\cap M_{[1,\alpha_0-1]}=0$. (In particular, $\dim(\im(\phi)\cap M_{[1,\alpha_1-1]})=2$.)
\end{enumerate}
\end{defn}

We can now state:

\begin{prop}\label{21_degen_comps}
Fix some $\alpha_1\le \alpha_0-1$. Upon the degeneration of $L_{\alpha_1}\to \Lambda'_2$, the flat limit of $Y^{(2,1)}_{(\alpha_0,\alpha_1-1)}$ consists of the components: 
\begin{itemize}
\item $Y^{(2,1)}_{(\alpha_0,\alpha_1)}$, if $\alpha_1\le \alpha_0-2$,
\item $Y^{(2,1)-\sec}_{(\alpha_0,\alpha_1)}$, if $\alpha_1\ge 2$, and
\item $Y^{(1,1,1)}_{(\alpha_0,\alpha_1)}$, if $\alpha_1\ge2$.
\end{itemize}
Furthermore, all of these components appear with multiplicity 1, and there are no others.

In particular, we have, as cycles in $\Coll(V,W)$:
\begin{equation*}
[Y^{(2,1)}_{\alpha_0}]=[Y^{(2,1)}_{(\alpha_0,0)}]=\sum_{\alpha_1=2}^{\alpha_0-1}[Y^{(2,1)-\sec}_{(\alpha_0,\alpha_1)}]+\sum_{\alpha_1=2}^{\alpha_0-1}[Y^{(1,1,1)}_{(\alpha_0,\alpha_1)}].
\end{equation*}
\end{prop}

The proof follows the same strategy as that of Proposition \ref{flat_limit_21_case} and is omitted.

We finally consider degenerations of $Y^{(2,1)-\sec}_{(\alpha_0,\alpha_1)}$. Write $\oL_{i}:=\langle L_i,\Lambda_1\rangle/\Lambda_1$ for $i=n_0+1,\ldots,n$, and we abusively write $\Lambda'_2$ for $\Lambda'_2/\Lambda_1$. It will also be convenient to set $\oL_{n+1}:=L_{\alpha_0}/\Lambda_1$.

We now successively move the lines $\oL_{n_0+1},\ldots,\oL_{n+1}$ to be equal to $\Lambda'_2$; suppose that $\oL_{n_0+1},\ldots,\oL_{\alpha_2}=\Lambda'_2$ and the other $\oL_i$ are general. Define now $Y^{(2,1)-\sec}_{(\alpha_0,\alpha_1,\alpha_2)}$ in exactly the same way as $Y^{(2,1)-\sec}_{(\alpha_0,\alpha_1)}$, except with the new arrangement of $\oL_i$.

However, it may also happen that a generic $\phi$ on $Y^{(2,1)-\sec}_{(\alpha_0,\alpha_1,\alpha_2-1)}$ degenerates at this step to one of type $(1,1,1)$. 
\begin{defn}
If $\alpha_1\le \alpha_2\le n$, define $Y^{(1,1,1)-\sec}_{(\alpha_0,\alpha_1,\alpha_2)}$ to be the closure of the locus of collineations $\phi$ satisfying the following properties.
\begin{enumerate}
\item $\phi$ has type $(1,1,1)$, with $V_1=\Lambda'_2$ and $V_2=\Lambda_1$,
\item $\im(\phi_0)\subset M_{[\alpha_1+1,n_0]}\cap M_{[\alpha_2+1,n]}$ and $\im(\phi_1)\subset M_{[\alpha_1+1,\alpha_0-1]}\cap M_{[\alpha_0+1,n_0]}$, 
\item $\dim(\im(\phi)\cap M_{[1,\alpha_1-1]})=2$ and $\dim(\im(\phi_1)\cap (M_{[1,\alpha_1]}\cap M_{[n_0+1,\alpha_2-1]}))=1$, but $\im(\phi_0)\cap M_{[1,\alpha_1-1]}=0$.
\end{enumerate}
\end{defn}

Note that in the above definition, we have not allowed $\alpha_2=n+1$, that is, for $\oL_{n+1}$ to become equal to $\Lambda'_2$. This is explained by the following:

\begin{prop}\label{21sec-0integral}
The class $[Y^{(2,1)-\sec}_{(\alpha_0,\alpha_1,n)}]$ pushes forward to 0 under $\pi$.
\end{prop}

\begin{proof}
Given a $\phi=(\phi_0,\phi_1)\in Y^{(2,1)-\sec}_{(\alpha_0,\alpha_1,n)}$, there exist infinitely many collineations in $Y^{(2,1)-\sec}_{(\alpha_0,\alpha_1,n)}$ in the same fiber of $\phi$. Indeed, one can replace $\phi_0$ with a 1-parameter family of maps still satisfying the conditions $\Inc(\Lambda'_2,M_i)$ for $i=1,2,\ldots,\alpha_1,n_0+1,\ldots,n$, in addition to the condition $\Inc(L_i,M_i)$ for $i=\alpha_0$, without changing $\im(\phi_0)$. Geometrically, this corresponds to the fact that a map $f:\bP^1\to\bP^1$ constrained to send $\alpha_1+(n-n_0)$ points to $0\in\bP^1$ and one point to $\infty\in\bP^1$ may be translated to infinitely many more by post-composing with the $\bC^{*}$-action on the target.
\end{proof}

In particular, when $n_0=n$, the class $[Y^{(2,1)-\sec}_{(\alpha_0,\alpha_1)}]$ already pushes forward to zero under $\pi$, so these subschemes do not contribute in the orbit closure calculation of the previous section (as we already saw, a posteriori).

We can now describe the degeneration of the $Y^{(2,1)-\sec}_{(\alpha_0,\alpha_1)}$ in a straightforward way.

\begin{prop}
Fix some $\alpha_2$ with $n_0\le\alpha_2\le n$. For the purposes of this statement, we write $Y^{(2,1)-\sec}_{(\alpha_0,\alpha_1,n_0-1)}:=Y^{(2,1)-\sec}_{(\alpha_0,\alpha_1)}$.

Then, the flat limit of the subscheme $Y^{(2,1)-\sec}_{(\alpha_0,\alpha_1,\alpha_2-1)}$ under the degeneration of $\oL_{\alpha_2}\mapsto\Lambda'_2$ contains the components:
\begin{itemize}
\item $Y^{(2,1)-\sec}_{(\alpha_0,\alpha_1,\alpha_2)}$, if $\alpha_2\le n-1$, and
\item  $Y^{(1,1,1)-\sec}_{(\alpha_0,\alpha_1,\alpha_2)}$,
\end{itemize}
both with multiplicity 1, and no other components with non-zero push-forward under $\pi$.
\end{prop}
In particular, we have
\begin{equation*}
[Y^{(2,1)-\sec}_{(\alpha_0,\alpha_1)}]=\sum_{\alpha_2=n_0+1}^{n}[Y^{(1,1,1)-\sec}_{(\alpha_0,\alpha_1,\alpha_2)}]
\end{equation*}
\emph{after push-forward by $\pi$}.

Combining with Proposition \ref{21_degen_comps}, we conclude:

\begin{cor}
Modulo cycles pushing forward to 0 under $\pi$, we have:
\begin{equation*}
[Y^{(2,1)}_{\alpha_0}]=\sum_{\alpha_1=2}^{\alpha_0-1}[Y^{(1,1,1)}_{(\alpha_0,\alpha_1)}]+\sum_{\alpha_1=2}^{\alpha_0-1}\sum_{\alpha_2=n_0+1}^{n}[Y^{(1,1,1)-\sec}_{(\alpha_0,\alpha_1,\alpha_2)}].
\end{equation*}
for all $\alpha_0\in[3,n_0]$.
\end{cor}
We will compute the pushforwards of the terms on the right hand side to $\Gr(3,W)$ in \S\ref{P2_degen_integrals}.

 \subsection{Degenerations of $Y^{(1,2)}$}
 
Because $Y^{(1,2)}_{\alpha_0}$ and $\wt{Y}_{\alpha_0}$ are birational over $\Gr(3,W)$, it suffices to consider the push-forwards of $\wt{Y}_{\alpha_0}$ under $\wt{\pi}$. We therefore study the $\wt{Y}_{\alpha_0}$ under degeneration. 

Recall that, for any $\alpha_0\ge n_0+3$, the subscheme $\wt{Y}_{\alpha_0}\subset \bP W\times\Coll(\Lambda_2,W)\times\Gr(3,W)$ is the closure of the locus of $(w_0,\wt{\phi_1},U)$ satisfying:
\begin{itemize}
\item $w_0\in M_{[1,n_0]}\cap M_{[\alpha_0+1,n]}$, and $w_0$ is contained in no other $M_{i}$,
\item $\wt{\phi_1}$ satisfies the incidence conditions $\Inc(L_i,M_i)$ for $i=1,2,\ldots,\alpha_0-1$, and $\im(\wt{\phi_1})\not\subset M_i$ for all such $i$,
\item $U=\langle w_0,\im(\wt{\phi_1})\rangle$ (in particular, $w_0\notin \im(\wt{\phi_1})$). 
\end{itemize}
We have abusively replaced $L_i\cap \Lambda_2$ with simply $L_i$ for $i=1,2,\ldots,n_0$, in order to simplify notation. 

Now, fix a general line $\Lambda'_1\subset\Lambda_2$. We will degenerate the $L_i$ to become equal to the $\Lambda'_i$. For some $\alpha_1\in[1,\alpha_0-1]$, suppose that $L_1=\cdots=L_{\alpha_1}=\Lambda'_1$, and the $L_i$ are otherwise general. Then, we define $\wt{Y}_{(\alpha_0,\alpha_1)}$ by the same properties as above, with the $L_i$ now in special position. The definition is the same whether $\alpha_1\le n_0$ or $\alpha_1> n_0$.

The $\wt{Y}_{(\alpha_0,\alpha_1)}$ are easily seen to be irreducible and generically smooth of the correct dimension $2n-n_0-1$. Furthermore, upon the degeneration $L_{\alpha_1}\to \Lambda'_1$, the component $\wt{Y}_{(\alpha_0,\alpha_1)}$ appears in the limit of $\wt{Y}_{(\alpha_0,\alpha_1-1)}$ with multiplicity 1.

\begin{lem}\label{tilde-zero-pushfoward}
$[\wt{Y}_{(\alpha_0,\alpha_0-2)}]$ and $[\wt{Y}_{(\alpha_0,\alpha_0-1)}]$ push forward to zero under $\wt{\pi}$.
\end{lem}

\begin{proof}
One verifies that the restriction of $\wt{\pi}$ to $\wt{Y}_{(\alpha_0,\alpha_0-2)}$ and $\wt{Y}_{(\alpha_0,\alpha_0-1)}$ have positive-dimensional fibers.
\end{proof}

We now describe the further components arising in this degeneration. There are two main ways in which a point $(w_0,\wt{\phi_1},U)$ can become degenerate: either $w_0$ can end up inside $\im(\wt{\phi_1})$, or $\wt{\phi_1}$ can degenerate into a collineation of type $(1,1)$ (It will follow from the considerations below that both cannot happen at once.)

Fix $\alpha_1\ge2$, and suppose further that $\alpha_1\le n_0$. We consider the conditions on $(w_0,\wt{\phi_1},U)$ in the limit of a general point of $\wt{Y}_{(\alpha_0,\alpha_1-1)}$, for which $w_0\in\im(\wt{\phi_1})$. We still have that $w_0\in M_{[1,n_0]}\cap M_{[\alpha_0+1,n]}$, and the incidence conditions $\Inc(L_i,M_i)$, for $i=1,2,\ldots,\alpha_0-1$, on $\wt{\phi_1}$. In particular, we have $\wt{\phi_1}(\Lambda'_2)\in M_{[1,\alpha_1]}$. However, we also have $w_0\in\im(\wt{\phi_1})$, so we should expect in fact $\wt{\phi_1}(\Lambda'_2)=w_0$ (up to scaling), and so $\wt{\phi_1}(\Lambda'_2)\subset M_{[1,n_0]}\cap M_{[\alpha_0+1,n]}$. (The only other possibility is that $\im(\wt{\phi_1})$ be contained in $M_{[1,\alpha_1]}$, but one can rule this out by a parameter count.)

On the other hand, because $L_{\alpha+1},\ldots,L_{n_0}\neq \Lambda'_2$, this forces $\im(\wt{\phi_1})\subset M_{[\alpha_1+1,n_0]}$. Finally, it is true on the general fiber that $\dim(U\cap M_{[1,\alpha_1-1]})=2$, so the same must be true on the special fiber. We now define:

\begin{defn}
Fix $\alpha_1\in[2,n_0]$. We define $\wt{Y}^{\sec}_{(\alpha_0,\alpha_1)} \subset \bP W\times\Coll(\Lambda_2,W)\times\Gr(3,W)$ to be the closure of the locus of $(w_0,\wt{\phi_1},U)$ satisfying:
\begin{enumerate}
\item $w_0\in M_{[1,n_0]}\cap M_{[\alpha_0+1,n]}$, and $w_0$ is contained in no other $M_{i}$,
\item $\wt{\phi_1}(\Lambda'_2)=w_0$, and $\wt{\phi_1}$ satisfies the further incidence conditions $\Inc(L_i,M_i)$ for $i=n_0+1,\ldots,\alpha_0-1$, as well as $\im(\wt{\phi_1})\subset M_{[\alpha_1+1,n_0]}$,
\item $\dim(U\cap M_{[1,\alpha_1-1]})=2$.
\end{enumerate}
\end{defn}

The space of $\wt{\phi_1}$ satisfying the first two properties has dimension
\begin{equation*}
(2n-1)-2(n_0-\alpha_1)-(n-1-n_0+\alpha_1).
\end{equation*}
Then, the space of $U$ containing a generic such $\im(\wt{\phi_1})$ and satisfying the last property has dimension $(n-3)-(\alpha_1-2)$. Therefore, $\wt{Y}^{\sec}_{(\alpha_0,\alpha_1)}$ has dimension $2n-n_0-1$. The usual incidence correspondence argument shows that it is in fact generically smooth and irreducible of this dimension. Moreover, $\wt{Y}^{\sec}_{(\alpha_0,\alpha_1)}$ appears in the flat limit of $\wt{Y}_{(\alpha_0,\alpha_1-1)}$ with multiplicity 1.

Suppose instead that we consider the analogous situation when $\alpha_1>n_0$. Then, we see on the one hand that $\wt{\phi_1}$ satisfies incidence conditions $\Inc(L_i,M_i)$ for $i=1,2,\ldots,\alpha_0-1$ and $\Inc(\Lambda'_2,M_i)$ for $i=\alpha_0+1,\ldots,n$ (as $\wt{\phi_1}=w_0$), and on the other that $\dim(U\cap M_{[1,n_0]})=2$. This imposes too many conditions on $\bP W\times\Coll(\Lambda_2,W)\times\Gr(3,W)$, and we see no such limit components.

We next consider the other degenerate possibility, that $\wt{\phi_1}$ degenerates into a collineation of type $(1,1)$, with $\ker((\wt{\phi_1})_{0})=\Lambda'_1$ (other possibilities are ruled out by via straightforward parameter counts). Suppose first that $\alpha_1\le n_0+1$. We consider the conditions imposed just on $U$. We have:
\begin{itemize}
\item $\dim(U\cap M_{[1,n_0]\cap [\alpha_0+1,n]}) \ge 1$, because $w_0\in U$,
\item $\dim(U\cap M_{[1,\alpha_0-1]}) \ge 2$, because the same is true on the general fiber,
\item $\dim(U\cap M_{[\alpha_1+1,\alpha_0-1]}) \ge 1$, due to the incidence conditions $\Inc(L_i,M_i)$ for $i=\alpha_1+1,\ldots,\alpha_0-1$.
\end{itemize}
The first two constraints impose $(n-\alpha_0+n_0-2)+(\alpha_1-2)$ conditions on $\Gr(3,W)$, and the last imposes $\alpha_0-\alpha_1-3$ more, for $n+n_0-7$ in total. As a result, one cannot obtain a non-trivial contribution in $H^{2(n+n_0-8)}(\Gr(3,W))$ after push-forward by $\wt{\pi}$.

Suppose on the other hand that $\alpha_1\ge n_0+2$. We may define:
\begin{defn}
Fix $\alpha_1\in[n_0+2,\alpha_0-1]$. We define $\wt{Y}^{(1,1)}_{(\alpha_0,\alpha_1)} \subset \bP W\times\Coll(\Lambda_2,W)\times\Gr(3,W)$ to be the closure of the locus of $(w_0,\wt{\phi_1},U)$ satisfying:
\begin{enumerate}
\item $w_0\in M_{[1,n_0]}\cap M_{[\alpha_0+1,n]}$,
\item $\wt{\phi_1}$ has type $(1,1)$ with $\ker((\wt{\phi_1})_0)=(\Lambda_2)_0=\Lambda'_1$ and $
\im((\wt{\phi_1})_0)\in M_{[\alpha_1+1,\alpha_0-1]}$,
\item $\im(\wt{\phi_1})\cap M_{[1,\alpha_1-1]}\neq0$,
\item $w_0\notin \im(\wt{\phi_1})$. In particular, $U=\langle w_0,\im(\wt{\phi_1})\rangle$ and $\dim(U\cap M_{[1,n_0]})\neq2$.
\end{enumerate}
\end{defn}

The subscheme $\wt{Y}^{(1,1)}_{(\alpha_0,\alpha_1)}$ is generically smooth and irreducible of dimension
\begin{equation*}
(\alpha_0-n_0-1)+(n-\alpha_0+\alpha_1-2)+(n-\alpha_1+2)=2n-n_0-1,
\end{equation*}
and appears in the limit of $\wt{Y}_{(\alpha_0,\alpha_1-1)}$ with multiplicity 1. One can also check that if it is required instead that $w_0\in \im(\wt{\phi_1})$, then we get too many conditions. Also, once $L_{\alpha_0-3}=\Lambda'_2$, the further components $\wt{Y}_{\alpha_0,\alpha_0-2},\wt{Y}_{\alpha_0,\alpha_0-1}$ do not contribute, by Lemma \ref{tilde-zero-pushfoward}.

The degeneration of $\wt{Y}_{(\alpha_0,\alpha_1-1)}$ is now summarized as follows.

\begin{prop}
Fix any $\alpha_0\in[n_0+3,n]$ and $\alpha_1\in[1,\alpha_0-2]$. Up to cycles pushing forward to zero under $\wt{\pi}$, in the limit as $L_{\alpha_1}\to \Lambda'_1$, the subscheme $\wt{Y}_{(\alpha_0,\alpha_1-1)}$ degenerates to a union of the components:
\begin{itemize}
\item $\wt{Y}_{(\alpha_0,\alpha_1)}$, if $\alpha_1\le \alpha_0-3$ (see Lemma \ref{tilde-zero-pushfoward}),
\item $\wt{Y}^{\sec}_{(\alpha_0,\alpha_1)}$, if $2\le \alpha_1\le n_0$,
\item $\wt{Y}^{(1,1)}_{(\alpha_0,\alpha_1)}$, if $\alpha_1\ge n_0+2$.
\end{itemize}
each with multiplicity 1.
\end{prop}

We finally, consider the $\wt{Y}^{\sec}_{(\alpha_0,\alpha_1)}$ under degeneration, where the $L_i$, $i=n_0+1,\ldots,\alpha_0-1$ are made equal to $\Lambda'_1$, one by one. We obtain the following degenerate components in the limits (each with multiplicity 1):

\begin{defn}
Fix $\alpha_0\in[n_0+3,n]$, $\alpha_1\in[2,n_0]$, and $\alpha_2\in[n_0+1,\alpha_0-1]$. Define $\wt{Y}^{(1,1)-\sec}_{(\alpha_0,\alpha_1,\alpha_2)} \subset \bP W\times\Coll(\Lambda_2,W)\times\Gr(3,W)$ to be the closure of the locus of $(w_0,\wt{\phi_1},U)$ satisfying:
\begin{enumerate}
\item $w_0\in M_{[1,\alpha_2-1]}\cap M_{[\alpha_0+1,n]}$, 
\item $\wt{\phi_1}$ has type $(1,1)$ with $\ker((\wt{\phi_1})_{0})=\Lambda'_1$ and $
\im((\wt{\phi_1})_0)=w_0$,
\item $\im(\wt{\phi_1})\subset M_{[\alpha_1+1,n_0]}$ and $\im(\wt{\phi_1})\cap M_{[\alpha_2+1,\alpha_0-1]}\neq0$,
\item $\dim(U\cap M_{[1,\alpha_1-1]})=2$.
\end{enumerate}
\end{defn}

Furthermore, once $L_{\alpha_0-2}=\Lambda'_1$, the locus of $(w_0,\wt{\phi_1},U)$ where the $\wt{\phi_1}$ still have rank 2 pushes forward to 0 under $\wt{\pi}$. We therefore conclude:

\begin{prop}
Modulo cycles pushing forward to 0 under $\wt{\pi}$, we have
\begin{equation*}
[\wt{Y}_{\alpha_0}]=\sum_{\alpha_1=n_0+2}^{\alpha_0-2} [\wt{Y}^{(1,1)}_{(\alpha_0,\alpha_1)}]+\sum_{\alpha_1=2}^{n_0}\sum_{\alpha_2=n_0+1}^{\alpha_0-2}[\wt{Y}^{(1,1)-\sec}_{(\alpha_0,\alpha_1,\alpha_2)}].
\end{equation*}
\end{prop}
  
 \subsection{Pushing forward to $\Gr(3,n)$}\label{P2_degen_integrals}
 
It is left to compute the pushforwards of the four families of components parametrizing totally degenerate collineations found above. In the entirety of this section, the SSYTs we will considered will always be filled with the entries $1,2,3$ (not all necessarily appearing).

\subsubsection{$Y^{(1,1,1)}$} As in the case $n=n_0$, the image $\pi\left(Y^{(1,1,1)}_{(\alpha_0,\alpha_1)}\right)$ is a generically transverse intersection of two Schubert cycles, of classes $\sigma_{n-\alpha_1-2,n_0-\alpha_0-1}$ (corresponding to property (ii) in Definition \ref{Y111_def}) and $\sigma_{\alpha_0-3,\alpha_1-2}$ (corresponding to property (iii)). Thus, we have
\begin{equation*}
\sum_{2\le \alpha_1<\alpha_0\le n_0}[Y^{(1,1,1)}_{(\alpha_0,\alpha_1)}]=\sum_{2\le \alpha_1<\alpha_0\le n_0}\sigma_{n-\alpha_1-2,n_0-\alpha_0-1}\cdot\sigma_{\alpha_0-3,\alpha_1-2}.
\end{equation*}

\begin{prop}\label{LR_strips}
Let $\lambda\subset(n-3)^3$ be a partition with $|\lambda|=n+n_0-8$. Then, the coefficient of $\sigma_\lambda$ in
\begin{equation*}
\sum_{2\le \alpha_1<\alpha_0\le n_0}\sigma_{n-\alpha_1-2,n_0-\alpha_0-1}\cdot\sigma_{\alpha_0-3,\alpha_1-2}.
\end{equation*}
is equal to the number of SSYTs of shape $\lambda$ such that
\begin{itemize}
\item $\lambda$ has no $(1,2)$-strip of length $n_0-3$, and
\item $\lambda$ has no $(2,3)$-strip of (maximal) length $n-3$.
\end{itemize}
\end{prop}
The first condition is equivalent to the requirement that the total number of 1's and 2's in the first row of $\lambda$ is at most $n_0-3$.

\begin{proof}
This is an application of Coskun's geometric Littlewood-Richardson rule \cite{coskun} similar to the main calculation of \cite{l_torus}. Choose a basis $w_1,\ldots,w_n$ of $W$ dual to the basis of hyperplanes $M_1,\ldots,M_n$. A general point of $\pi\left(Y^{(1,1,1)}_{(\alpha_0,\alpha_1)}\right)$ has a basis $u_0,u_1,u_2$ with:
\begin{itemize}
\item $u_0\in\langle w_1,\ldots,w_{\alpha_1}\rangle$,
\item $u_1\in\langle w_{\alpha_1},\ldots,w_{\alpha_0},w_n,\ldots,w_{n_0+1} \rangle$,
\item $u_2\in \langle w_{\alpha_0},\ldots,w_n\rangle$.
\end{itemize}
The corresponding subscheme of $\Gr(3,W)$ is associated to the \emph{Mondrian tableau} $\cT$ depicted in Figure \ref{mondrian1}. The tableau $\cT$ consists of the data of three squares $S_0,S_1,S_2$ corresponding to the conditions on the vectors $u_0,u_1,u_2$, respectively, and the integers appearing in each square correspond to the basis vectors spanning the subspaces of $W$ prescribed to contain the $u_j$. Note that we have reversed the order of $\alpha_0+1,\ldots,n$ to make the squares contiguous.
\begin{figure}
\caption{The \emph{Mondrian tableau} $\mathcal{T}$ associated to the image of $Y^{(1,1,1)}_{(\alpha_0,\alpha_1)}$ in $\Gr(3,W)$.}\label{mondrian1}
\begin{tikzpicture} [xscale=0.35,yscale=0.35]
               
                \node at (2.5,2.5) {\tiny $1$};     
		 \node at (4,4) {\tiny $\iddots$};
                \node at (5.5,5.5) {\tiny $\alpha_1$};
                \node at (7,7) {\tiny $\iddots$};
                   \node at (8.5,8.5) {\tiny $\alpha_0$};
                 \node at (9.5,9.5) {\tiny $n$};
                 \node at (11,11) {\tiny $\iddots$};
		   \node at (12,12.5) {\tiny $n_0{+}1$};
                  \node at (13.5,13.5) {\tiny $n_0$};
                   \node at (15,15) {\tiny $\iddots$};
                    \node at (16,16.5) {\tiny $\alpha_0{+}1$};               

\draw (2,2) -- (6,2) -- (6,6) -- (2,6) -- (2,2);
\draw (5,5) -- (13,5) -- (13,13) -- (5,13) -- (5,5);
\draw (8,8) -- (17,8) -- (17,17) -- (8,17) -- (8,8);
\end{tikzpicture}
\end{figure}

Consider now the Mondrian tableaux $\mathcal{T}',\mathcal{T}''$ depicted in Figure \ref{mondrian2}. These define subschemes $Z',Z''$, respectively, of $\Gr(3,W^+)$, where $W^+$ is obtained from $W$ by adding an additional basis vector $u_0$.

\begin{figure}
\caption{Mondrian tableaux $\mathcal{T}',\mathcal{T}''$ defining subschemes $Z',Z''$ of $\Gr(3,W^+)$.}\label{mondrian2}
\begin{tikzpicture} [xscale=0.35,yscale=0.35]
               
                \node at (1.5,1.5) {\tiny $0$};     
		 \node at (3,3) {\tiny $\iddots$};
		  \node at (4,4.5) {\tiny $\alpha_1{-}1$};     
                \node at (5.5,5.5) {\tiny $\alpha_1$};
                \node at (7,7) {\tiny $\iddots$};
                   \node at (8.5,8.5) {\tiny $\alpha_0$};
                 \node at (9.5,9.5) {\tiny $n$};
                 \node at (11,11) {\tiny $\iddots$};
		   \node at (12,12.5) {\tiny $n_0{+}1$};
                  \node at (13.5,13.5) {\tiny $n_0$};
                   \node at (15,15) {\tiny $\iddots$};
                    \node at (16,16.5) {\tiny $\alpha_0{+}1$};               

\draw (1,1) -- (1,5) -- (5,5) -- (5,1) -- (1,1);
\draw (5,5) -- (13,5) -- (13,13) -- (5,13) -- (5,5);
\draw (8,8) -- (17,8) -- (17,17) -- (8,17) -- (8,8);
\end{tikzpicture}
\begin{tikzpicture} [xscale=0.35,yscale=0.35]
               
                \node at (1.5,1.5) {\tiny $0$};     
		 \node at (3,3) {\tiny $\iddots$};
		  \node at (4,4.5) {\tiny $\alpha_1{-}1$};     
                \node at (5.5,5.5) {\tiny $\alpha_1$};
                \node at (7,7) {\tiny $\iddots$};
                   \node at (8.5,8.5) {\tiny $\alpha_0$};
                 \node at (9.5,9.5) {\tiny $n$};
                 \node at (11,11) {\tiny $\iddots$};
		   \node at (12,12.5) {\tiny $n_0{+}1$};
                  \node at (13.5,13.5) {\tiny $n_0$};
                   \node at (15,15) {\tiny $\iddots$};
                    \node at (16,16.5) {\tiny $\alpha_0{+}1$};               

\draw (5,5) -- (6,5) -- (6,6) -- (5,6) -- (5,5);
\draw (1,1) -- (13,1) -- (13,13) -- (1,13) -- (1,1);
\draw (8,8) -- (17,8) -- (17,17) -- (8,17) -- (8,8);
\end{tikzpicture}
\end{figure}

Note that $\cT$ is obtained from $\cT'$ from by shifting $S_0$ to the northeast by one unit, and $\cT''$ is obtained from $\cT'$ by replacing $S_0$ and $S_1$ with the union (or span) $S_0\cup S_1$ in $\cT'$ and the intersection $S_0\cap S_1$ in $\cT$. The geometric Littlewood-Richardson rule \cite[Theorem 3.32]{coskun} implies that
\begin{equation*}
[Z']=[Z'']+\left[\pi\left(Y^{(1,1,1)}_{(\alpha_0,\alpha_1)}\right)\right],
\end{equation*}
on $\Gr(3,W^+)$, where the last term is pushed forward from $\Gr(3,W)$ to $\Gr(3,W^+)$. via the inclusion $W\subset W^+$. The geometric content of this relation is that, in the limit of the degeneration of the basis of $W^+$ sending $w_0\to tw_0+(1-t)w_{\alpha_1}$ over $t\in\bD$, the subvariety $Z'$ breaks into a union of $Z''$ and $\pi(Y^{(1,1,1)}_{(\alpha_0,\alpha_1)})$.

We now compute $[Z']$ and $[Z'']$. First, $Z'$ is the generically transverse intersection of a Schubert variety of class $\sigma_{n-\alpha_1-1}$ (corresponding to $S_0$) and a subscheme of class $(\sigma_{\alpha_0-2}\cdot\sigma_{\alpha_1+n_0-\alpha_0-2})_{\lambda_0\le n_0-3}$ (corresponding to $S_0$ and $S_1$). The computation of the latter class itself follows from the geometric Pieri rule, as the subscheme in question is defined by Schubert cycles of classes $\sigma_{\alpha_0-2},\sigma_{\alpha_1+n_0-\alpha_0-2}$, and applying \cite[Algorithm 3.12]{coskun} will result in an inner square of size at least $n-n_0+2$. Alternatively, one can view this subscheme as pushed forward from $\Gr(3,\langle w_{\alpha_1},\ldots,w_{n}\rangle)$, on which it is given by a \emph{transverse} intersection of Schubert cycles. 

Now, by the Pieri rule, the coefficient of $\sigma_\lambda$ in $(\sigma_{\alpha_0-2}\cdot\sigma_{\alpha_1+n_0-\alpha_0-2})_{\lambda_0\le n_0-3}\cdot\sigma_{n-\alpha_1-1}$ is given by the number of SSYTs of shape $\lambda$ with $(c_1,c_2,c_3)=(\alpha_0-2,\alpha_1+n_0-\alpha_0-2,n-\alpha_1-1)$, with the property that at most $n_0-3$ of the boxes in the first row are filled with 1's and 2's (equivalently, there is no $(1,2)$-strip of length $n_0-2$). Summing over all $\alpha_0,\alpha_1$ with $2\le \alpha_1<\alpha_0\le n_0$, we get exactly the SSYTs with $1,2,3$ each appearing any number of times, with no such $(1,2)$-strips.

On the other hand, $Z''$ is the generically transverse intersection of Schubert varieties of classes $\sigma_{(n-2,n_0-\alpha_0-1)}$ (corresponding to the two new squares replacing $S_0,S_1$) and $\sigma_{\alpha_0-2}$. By the same argument as in \cite[Proposition 11]{l_torus}, the coefficient of $\sigma_{\lambda}$ in the product of these two classes is equal to the number of SSYTs of shape $\lambda$ with $(c_1,c_2,c_3)=(\alpha_0-2,\alpha_1+n_0-\alpha_0-2,n-\alpha_1-1)$ and with a $(2,3)$-strip of maximal length $n-2$. Summing over all $\alpha_0,\alpha_1$ with $2\le \alpha_1<\alpha_0\le n_0$, we get exactly the SSYTs with a $(2,3)$-strip of maximal length.

Finally, to obtain $\left[\pi\left(Y^{(1,1,1)}_{(\alpha_0,\alpha_1)}\right)\right]$, we subtract $[Z'']$ from those of $[Z']$ and divide by $\sigma_{1^3}$. (One can see directly that the cofficient of $\sigma_\lambda$ in the difference is zero when $\lambda_2=0$, but this was known a priori.) The conclusion follows.
\end{proof}

\subsubsection{$Y^{(1,1,1)-\sec}$}

\begin{prop}
Let $\lambda\subset(n-3)^3$ be a partition with $|\lambda|=n+n_0-8$. If $\lambda_0\le n-4$, then, the coefficient of $\sigma_\lambda$ in
\begin{equation*}
\sum_{\alpha_0=3}^{n_0}\sum_{\alpha_1=2}^{\alpha_0-1}\sum_{\alpha_2=n_0+1}^{n}\pi_{*}\left([Y^{(1,1,1)-\sec}_{(\alpha_0,\alpha_1,\alpha_2)}]\right)
\end{equation*}
is equal to the number of SSYTs of shape $\lambda$ such that
\begin{itemize}
\item $c_2\le n_0-3$, and 
\item $\lambda$ has a $(1,2)$-strip of length (at least) $n_0-3$.
\end{itemize}

If instead $\lambda_0=n-3$, then the coefficient of $\sigma_\lambda$ is zero.
\end{prop}
Note that if $\lambda$ has a $(1,2)$-strip of length at least $n_0-3$, then such a strip already exists in the first row of $\lambda$. Thus, this condition is equivalent to the stipulation that the total number of 1's and 2's in the first row of $\lambda$ is at least $n_0-3$.

\begin{proof}
A general point of $\pi\left(Y^{(1,1,1)-\sec}_{(\alpha_0,\alpha_1,\alpha_2)}\right)$ has a basis $u_0,u_1,u_2$ with the property that:
\begin{itemize}
\item $u_0\in M_{[\alpha_1+1,n_0]}\cap M_{[\alpha_2+1,n]}$,
\item $u_1\in M_{[1,\alpha_0-1]} \cap M_{[\alpha_0+1,\alpha_2-1]}$,
\item $u_2\in M_{[1,\alpha_1-1]}$.
\end{itemize}
Furthermore, $\pi$ is generically injective on $Y^{(1,1,1)-\sec}_{(\alpha_0,\alpha_1,\alpha_2)}$.

The condition on $u_0$ corresponds to a Schubert subvariety of $\Gr(3,W)$ of class $\sigma_{(n-\alpha_2)+(n_0-\alpha_1)-2}$, and the conditions on $u_1,u_2$ correspond to one of class $\sigma_{\alpha_2-4,\alpha_1-2}$. These two subschemes are not in general position with respect to one another, due to the fact that $M_{\alpha_1}$ does not appear in any of the intersections constraining the $u_j$. Dually, if we choose a dual basis $w_1,\ldots,w_n$ to the $M_i$, then $w_{\alpha_1}$ appears in all of the subspaces constraining the $u_j$.

However, passing from $W$ to $W/w_j$ and taking the same conditions on $\Gr(3,W/W_j)$, one does obtain a generically transverse intersection of Schubert subvarieties of class $\sigma_{\alpha_2-4,\alpha_1-2}\cdot \sigma_{(n-\alpha_2)+(n_0-\alpha_1)-2}$. One can in particular apply the geometric Littlewood-Richardson algorithm to deform this intersection into a union of Schubert varieties $\Sigma_\lambda$ in $\Gr(3,W/w_{\alpha_1})$ defined with respect to the basis $w_1,\ldots,w_{\alpha_1-1},w_{\alpha_1+1},\ldots,w_n$. More precisely, the $\Sigma_\lambda$ are defined by degeneracy conditions $\dim(U\cap W'_j)\ge j$ for $j=1,2,3$, where the subspaces $W'_3\subset W'_2\subset W'_1\subset W/w_{\alpha_1}$ are each spanned by some subset of the $w_1,\ldots,w_{\alpha_1-1},w_{\alpha_1+1},\ldots,w_n$.

One can now run exactly the same algorithm in $\Gr(3,W)$ with the conditions on $u_0,u_1,u_2$ above, keeping track of the fact that the constraining subspaces now all contain the additional basis vector $w_{\alpha_1}$. In particular, one obtains in the end precisely the same union of Schubert varieties $\Sigma_\lambda$, now defined on $\Gr(3,W)$ by the conditions $\dim(U\cap W_j)\ge j$ for $j=1,2,3$, where $W_j$ is the pullback of $W'_j$ under the quotient by $w_{\alpha_1}$. The class in question is therefore given by computing the product $\sigma_{\alpha_2-4,\alpha_1-2}\cdot \sigma_{(n-\alpha_2)+(n_0-\alpha_1)-2}$ in $\Gr(3,W/W_j)$, and then applying the inclusion $$H^{2(n+n_0-8)}(\Gr(3,W/W_j))\to H^{2(n+n_0-8)}(\Gr(3,W))$$ defined by $\sigma_\lambda\mapsto\sigma_\lambda$. Equivalently, one computes
\begin{equation*}
(\sigma_{\alpha_2-4,\alpha_1-2}\cdot \sigma_{(n-\alpha_2)+(n_0-\alpha_1)-2})_{\lambda_0\le n-4},
\end{equation*}
by which we mean that the product is computed on $\Gr(3,W)$ and only the terms with $\sigma_\lambda$ with $\lambda_0\le n-4$ are kept.

It remains to match the terms in this product, as determined by the Pieri rule, with the described SSYTs after summing over all $(\alpha_0,\alpha_1,\alpha_2)$ with $2\le\alpha_1<\alpha_0\le n_0<\alpha_2\le n$. This is done as follows: given  $(\alpha_0,\alpha_1,\alpha_2)$ fixed, a term $\sigma_\lambda$ in the product $\sigma_{\alpha_2-4,\alpha_1-2}\cdot \sigma_{(n-\alpha_2)+(n_0-\alpha_1)-2}$ is given by an inclusion of Young tableaux $(\alpha_2-4,\alpha_1-2)\subset\lambda$, where the boxes of $\lambda-(\alpha_2-4,\alpha_1-2)$ are filled with 3's, no two in the same column. Then, the value of $\alpha_0$ further determines the filling of the remaining boxes corresponding to the shape $(\alpha_2-4,\alpha_1-2)$ with 1's and 2's: one takes the unique such filling with $(c_1,c_2)=(\alpha_2-\alpha_0+\alpha_1-3,\alpha_0-3)$.

Because $\alpha_2-4\ge n_0-3$, the condition on $(1,2)$ is automatically satisfied, and we also have $c_2\le \alpha_0-3$. Conversely, given a SSYT as in the statement of the Proposition, the values of $\alpha_0,\alpha_1,\alpha_2$ are uniquely determined and satisfy the needed chain of inequalities. (In particular, the assumption that $\lambda_0\le n-4$ gives that $\alpha_2\le n$.) This completes the proof.
\end{proof}

 \subsubsection{$\wt{Y}^{(1,1)}$}

\begin{prop}
Let $\lambda\subset(n-3)^3$ be a partition with $|\lambda|=n+n_0-8$. Then, the coefficient of $\sigma_\lambda$ in
\begin{equation*}
\sum_{\alpha_0=n_0+3}^{n}\sum_{\alpha_1=n_0+2}^{\alpha_0-2}\pi_{*}\left([\wt{Y}^{(1,1)}_{(\alpha_0,\alpha_1)}]\right)
\end{equation*}
is equal to the number of SSYTs of shape $\lambda$ with at least $n_0-1$ 2's in the second row.
\end{prop}

In particular, we need $\lambda_1\ge n_0-1$ in order for this number to be non-zero, or equivalently, $\lambda_0+\lambda_2\le n-7$.

\begin{proof}
A general point of $\wt{\pi}\left(Y^{(1,1)}_{(\alpha_0,\alpha_1)}\right)$ has a basis $u_0,u_1,u_2$ with the property that:
\begin{itemize}
\item $u_0\in M_{[1,n_0]}\cap M_{[\alpha_0+1,n]}$,
\item $u_1\in  M_{[\alpha_1+1,\alpha_0-1]}$,
\item $u_2\in M_{[1,\alpha_1-1]}$.
\end{itemize}
We see that if $\alpha_1\ge n_0+2$, $\alpha_0\ge \alpha_1+3$, and $\alpha_0<n$, then $\wt{\pi}$ is generically injective on $Y^{(1,1,1)-\sec}_{(\alpha_0,\alpha_1,\alpha_2)}$; otherwise, $\wt{\pi}$ has positive-dimensional fibers and the corresponding summand above is zero. In the former case, we again apply the geometric Littlewood-Richardson rule implicitly.

The subvariety of $\Gr(3,n)$ of $U\subset W$ having a 2-dimensional subspace of the form $\langle u_0,u_2\rangle$ as above has class $\left( \sigma_{n-\alpha_0+n_0-2}\cdot\sigma_{\alpha_1-3}\right)_{\lambda_1\ge n_0-1}$. Then, the locus of $U\subset W$ containing a vector $u_1$ as above is a Schubert variety of class $\sigma_{\alpha_0-\alpha_1-3}$. These two subvarieties of $\Gr(3,n)$ are defined with respect to disjoint subsets of the hyperplanes $M_1,\ldots,M_n$, so their intersection has class
\begin{equation*}
\left( \sigma_{n-\alpha_0+n_0-2}\cdot\sigma_{\alpha_1-3}\right)_{\lambda_1\ge n_0-1}\cdot \sigma_{\alpha_0-\alpha_1-3}.
\end{equation*}

For fixed $\alpha_0,\alpha_1$, the coefficient of $\sigma_\lambda$ is equal to the number of SSYT of shape $\lambda$ with $(c_1,c_2,c_3)=(n-\alpha_0+n_0-2,\alpha_1-3,\alpha_0-\alpha_1-3)$ with at least $n_0-1$ 2's in the second row. Note that $\alpha-3\ge n_0-1$. Conversely, given such a SSYT with $(c_1,c_2,c_3)$ arbitrary, we may uniquely recover $\alpha_1=c_2+3\ge n_0+2$ and $\alpha_0=c_3+\alpha_1+3=c_2+c_3+6\le n-1$. Thus, summing over all $\alpha_0,\alpha_1$ gives the claimed SSYTs.
\end{proof}

\subsubsection{$\wt{Y}^{(1,1)-\sec}$}

\begin{prop}
Let $\lambda\subset(n-3)^3$ be a partition with $|\lambda|=n+n_0-8$. If $\lambda_0\le n-5$, then, the coefficient of $\sigma_\lambda$ in
\begin{equation*}
\sum_{\alpha_0=n_0+3}^{n}\sum_{\alpha_1=2}^{n_0}\sum_{\alpha_2=n_0+1}^{\alpha_0-2}[\wt{Y}^{(1,1)-\sec}_{(\alpha_0,\alpha_1,\alpha_2)}].
\end{equation*}
is equal to the number of SSYTs of shape $\lambda$ with $c_2\ge n_0-2$, but with at most $n_0-2$ 2's in the second row.

If instead $\lambda_0\ge n-4$, then this coefficient is zero.
\end{prop}

\begin{proof}
A general point of $\wt{\pi}\left(Y^{(1,1)}_{(\alpha_0,\alpha_1)}\right)$ has a basis $u_0,u_1,u_2$ with the property that:
\begin{itemize}
\item $u_0\in M_{[1,\alpha_2-1]}\cap M_{[\alpha_0+1,n]}$,
\item $u_1\in M_{[\alpha_1+1,n_0]} \cap  M_{[\alpha_2+1,\alpha_0-1]}$,
\item $u_2\in M_{[1,\alpha_1-1]}$.
\end{itemize}
Moreover, the map $\wt{\pi}$ is generically injective on $Y^{(1,1,1)-\sec}_{(\alpha_0,\alpha_1,\alpha_2)}$.

The conditions on $u_0,u_2$ define a Schubert variety of class $\sigma_{n-\alpha_0+\alpha_2-3,\alpha_1-2}$, and that on $u_1$ defines a Schubert variety of class $\sigma_{\alpha_0-\alpha_2+n_0-\alpha_1-3}$. These Schubert varieties are not defined with respect to transverse flags, but this is true upon passing to the quotient of $W$ by vectors dual to the hyperplanes $M_{\alpha_2}$ and $M_{\alpha_0}$, which do not appear in any of the intersections above. Thus, the desired class is
\begin{equation*}
(\sigma_{n-\alpha_0+\alpha_2-3,\alpha_1-2}\cdot \sigma_{\alpha_0-\alpha_2+n_0-\alpha_1-3})_{\lambda_0\le n-5}.
\end{equation*}

If $\lambda_0\le n-5$ and $\alpha_0,\alpha_1,\alpha_2$ are fixed, then each term $\sigma_\lambda$ in the expansion of the above product corresponds to an inclusion of Young tableaux $(n-\alpha_0+\alpha_2-3,\alpha_1-2)\subset\lambda$, whose complement is filled by 3's, no two in the same column. Then, we fill of the remaining boxes in the unique semi-standard way with $(c_1,c_2)=(n-\alpha_0+\alpha_1-2,\alpha_2-3)$; this is possible because $n-\alpha_0+\alpha_1-2\le n-\alpha_0+\alpha_2-3$, as $\alpha_2>n_0\ge \alpha_1$. Furthermore, note that $c_2=\alpha_2-3\ge n_0-2$, and that the number of 2's in the second row is $\alpha_1-2\le n_0-2$.

Conversely, given an SSYT of shape $\lambda$ of the desired form, then we take $\alpha_1$ to be 2 more than the number of 2's in the second row, so that $2\le \alpha_1\le n_0$, we take $\alpha_2=c_2+3\ge n_0+1$, and $\alpha_0=n+\alpha_1-2-c_1$. Then, $\alpha_0\le n$ because $c_1$ at least the number of 2's in the second row, and the statement $\alpha_0\ge \alpha_2+2$ reduces to the fact that the total number of 1's and 2's is at most $n-5$. This completes the proof.
\end{proof}

\subsubsection{Proof of Theorem \ref{thm:countP2}} 
It suffices to assume, as we have throughout, that $d=n-1$. We combine the SSYTs underlying the contributions to $[Z]$ coming from components of the four types $Y^{(1,1,1)},Y^{(1,1,1)-\sec},\wt{Y}^{(1,1)},\wt{Y}^{(1,1)-\sec}$. We will see that the subsets of SSYTs of a given shape $\lambda$ from each case are disjoint, and that their union gives exactly the subsets described in Theorem \ref{thm:countP2}.

First, suppose that $\lambda_0\le n-5$. Then, we get from $\wt{Y}^{(1,1)}$ all SSYTs with at least $n_0-1$ 2's in the 2nd row. Of the remaining SSYTs, with at most $n_0-2$ 2's in the 2nd row, we then get from $\wt{Y}^{(1,1)-\sec}$ those with $c_2\ge n_0-2$. The remaining SSYTs are exactly those with $c_2\le n_0-3$. Of those, we get from $Y^{(1,1,1)-\sec}$ the SSYTs with a $(1,2)$-strip of length $n_0-3$, and the only remaining SSYTs are those without such a $(1,2)$-strip, obtained from $Y^{(1,1,1)}$. Note here that there are automatically no $(2,3)$-strips of length $n-3$, and also that an SSYT without a $(1,2)$-strip of length $n_0-3$ necessarily has $c_2\le n_0-3$. In all, we get exactly the set of all SSYTs, each appearing exactly once in one of the four families of components.

Next, suppose that  $\lambda_0= n-4$. We only get contributions from $Y^{(1,1,1)},Y^{(1,1,1)-\sec}$. In both cases, we get SSYTs with $c_2\le n_0-3$. Then, the SSYTs coming from $Y^{(1,1,1)-\sec}$ are those with a $(1,2)$-strip of length  $n_0-3$, and those coming from $Y^{(1,1,1)}$ are exactly those without such a strip.

Finally, suppose that $\lambda_0=n-3$. Then, the only SSYTs come from $Y^{(1,1,1)}$, and here we get precisely the claimed subset. $\square$\\

As a check, Proposition \ref{asymptotic_formula} under assumption (c) gives that $\Gamma_{2,\overrightarrow{n},d}^{\lambda}=|\SSYT_3(\lambda)|$ whenever $n\ge d+3$; by Proposition \ref{increase_degree}, this is equivalent to the statement that $\Gamma_{2,\overrightarrow{n},n-1}^{\lambda}=|\SSYT_3(\lambda)|$ when $\lambda_0\le n-5$.

\end{document}